\documentclass[a4paper]{article}

\usepackage[english]{babel}

\usepackage[utf8]{inputenc}
\usepackage[utf8]{inputenc}
\usepackage{amsmath}
\usepackage{graphicx}
\usepackage{amssymb}
\usepackage{amsthm}
\usepackage{tikz-cd}
\usepackage{mathrsfs}
\usepackage[colorinlistoftodos]{todonotes}
\usepackage{enumitem}
\usepackage{yfonts}
\usepackage{ dsfont }
\usepackage{MnSymbol}
\usepackage{slashed}

\title{The large mass limit of $G_2$ and Calabi-Yau monopoles}

\author{Yang Li}

\date{\today}
\newtheorem{thm}{Theorem}[section]
\newtheorem{lem}[thm]{Lemma}

\theoremstyle{definition}
\newtheorem{eg}[thm]{Example}

\newtheorem{cor}[thm]{Corollary}

\newtheorem{rmk}[thm]{Remark}
\newtheorem{prop}[thm]{Proposition}
\newtheorem{Def}[thm]{Definition}
\newtheorem{Question}{Question}
\newtheorem*{Notation}{Notation}
\newtheorem*{Main Setting}{Main Setting}

\newtheorem*{Acknowledgement}{Acknowledgement}

\newcommand{\cf}{\emph{cf.} }

\newcommand{\R}{\mathbb{R}}

\newcommand{\Z}{\mathbb{Z}}

\newcommand{\norm}[1]{\left\lVert#1\right\rVert}
\newcommand{\Lap}{\Delta}

\DeclareMathOperator{\Tr}{Tr}

\def\Xint#1{\mathchoice
	{\XXint\displaystyle\textstyle{#1}}%
	{\XXint\textstyle\scriptstyle{#1}}%
	{\XXint\scriptstyle\scriptscriptstyle{#1}}%
	{\XXint\scriptscriptstyle\scriptscriptstyle{#1}}%
	\!\int}
\def\XXint#1#2#3{{\setbox0=\hbox{$#1{#2#3}{\int}$ }
		\vcenter{\hbox{$#2#3$ }}\kern-.6\wd0}}

\def\dashint{\Xint-}

\begin{document}
	\maketitle

	\begin{abstract}
	We develop a structure theory for the limit of $SU(2)$ $G_2$-monopoles (resp. Calabi-Yau monopoles) on a principal $SU(2)$-bundle over an asymptotically conical $G_2$-manifolds (resp. Calabi-Yau 3-folds) as the mass parameter tends to infinity, while the topologial data for the bundle stays fixed. We show how to extract a singular abelian $G_2$-monopole (resp. Calabi-Yau monopole) with Dirac singularity along a calibrated cycle in the large mass limit, and we prove an energy identity for monopole bubbles.
	\end{abstract}

	\section{Introduction and background}

	Donaldson-Segal \cite[Section 6.3]{DonaldsonSegal}  proposed a programme relating gauge theory on a complete noncompact $G_2$-manifold $M$ (resp. Calabi-Yau 3-fold $M$), to calibrated geometry on $M$. %We shall make a parallel treatment of both the $G_2$ and the Calabi-Yau case. 
	We denote $n=\dim M$, so $n=7$ in the $G_2$-case, and $n=6$ in the Calabi-Yau case.

	\subsection{Curvature concentration for $G_2$/Calabi-Yau monopoles}

	Let $(M,\phi)$ be a $G_2$-manifold with the parallel positive 3-form $\phi$ and the 4-form $\psi=*\phi$, and $P$ be a principal $G$-bundle over $M$. In this paper we take $G=SU(2)$, and the Lie algebra $\mathfrak{g}=su(2)$. (The case $G=SO(3)$ works almost the same way).
	 We consider the pair $(A,\Phi)$ consisting of a connection $A$ on $P$ with curvature $F$, and a smooth section $\Phi$ of the adjoint bundle $\mathfrak{g}_P=P\times_{Ad} \mathfrak{g}$, known as the Higgs field. The pair is called a \emph{$G_2$-monopole}, if 
	\begin{equation}\label{eqn: G2monopole}
		*(F\wedge \psi)= \nabla\Phi.
	\end{equation}	
	In particular this satisfies the Euler-Lagrange equation of  both the \emph{Yang-Mills-Higgs functional}
	\begin{equation}\label{eqn:YMfunctional}
	\mathcal{E}(A,\Phi)= \frac{1}{2} \int_M |F|^2+ |\nabla \Phi|^2,
	\end{equation}
	and the \emph{intermediate energy functional}
	\begin{equation}\label{eqn:intermediateenergy}
	\mathcal{E}^\psi( A,\Phi  )= \frac{1}{2}\int_M  |F\wedge \psi|^2 + |\nabla \Phi|^2.
	\end{equation}
	An integration by parts argument shows that if $\Phi$ is not parallel, then $M$ is noncompact.

		There is a closely analogous equation on Calabi-Yau 3-folds $(M,\omega,\Omega)$, where $\omega$ is the Calabi-Yau metric, and $\Omega$ is the nowhere vanishing holomorphic volume form. A pair $(A,\Phi)$ of a connection on $P$ and a Higgs field is called a Calabi-Yau monopole, if
	\begin{equation}
		F\wedge \omega^2=0,\quad *\nabla \Phi= F\wedge \text{Re} \Omega.
	\end{equation}
	In particular this satisfies the Euler-Lagrange equation of  both the Yang-Mills-Higgs (YMH) functional (\ref{eqn:YMfunctional}), and the intermediate energy functional
	\begin{equation}\label{eqn:intermediateenergyCY}
		\mathcal{E}^{\Omega}(A,\Phi)=  \frac{1}{2}\int_M  |F\wedge \text{Re}\Omega|^2 + |\nabla \Phi|^2.         
	\end{equation}

	\begin{eg}\label{eg:flatmodel}
	The simplest example comes from dimensional reduction: $M=Q\times \R^3$ where $Q \simeq \R^4$ is a a \emph{coassociative} 4-plane in the $G_2$ case, and $Q\simeq \R^3$ is a \emph{special Lagrangian} 3-plane in the Calabi-Yau case. We  take $(A,\Phi)$ as the pullback of a monopole over $\R^3$. Suppose the monopole has mass $m\gg 1$ and fixed charge $k\in \Z$. At large distance  from  $Q$, the structure group asymptotically reduces to $U(1)$, with first Chern class $k$ over the large $S^2$ linking $Q$. The asymptotic far from $Q$ is modelled on a singular solution to the $U(1)$-monopole equation, known as \emph{Dirac singularity}:
	\[
	\Phi\sim (m- \frac{  k  }{2 \text{dist}(\cdot, Q)} )  \begin{pmatrix}
		i & 0\\
		0& -i 
	\end{pmatrix} ,
		\quad 
		F \sim - *_3 d(   \frac{  k  }{2 \text{dist}(\cdot, Q)} )  \begin{pmatrix}
			i & 0\\
			0& -i 
		\end{pmatrix} . 
	\]
	For $ \text{dist}(\cdot, Q)     \gtrsim m^{-1}$, the magnitude of the curvature and the Higgs field are of the order
	\[
	|\Phi|= m+ O( \frac{1}{\text{dist}(\cdot, Q)   }  ), \quad   |F|=  |\nabla \Phi|= O( \frac{1}{\text{dist}(\cdot, Q) ^2} ).
	\]
	For $\text{dist}(\cdot, Q)  \lesssim m^{-1}$, we have $|F|= |\nabla \Phi|= O(m^2)$. Inside any ball $B(p,r)$ with $p$ centred on $Q$ and $r\gg m^{-1}$, the $L^2$-curvature is of the order
	\[
	\int_{B(p,r)} |F|^2\sim  mr^{n-3}.
	\]
	Most of the $L^2$-curvature contribution is concentrated in the $O(m^{-1})$ neighbourhood of $Q$.

	\end{eg}

	More generally, 
	Donaldson-Segal proposed the following conjectural mechanism of curvature concentration for $G_2$-monopoles (resp. Calabi-Yau monopoles). Let $Q\subset M$ be a compact \emph{coassociative} submanifold (resp. \emph{special Lagrangian} submanifold), and take the mass parameter $m\gg 1$. On each normal 3-plane in the $O(m^{-1})$ tubular neighbourhood of $Q$,  the restriction of $(A,\Phi)$ is well approximated by a monopole of charge $k$ and mass $m$. The variation of the monopoles along $Q$ is governed by a section of the bundle $\underline{M}_{k}\to Q$ with fibres modelled on the hyperk\"ahler moduli space of centred charge $k$ monopoles, such that the section satisfies a Fueter type equation. Away from the $O(m^{-1})$ tubular neighbourhood, the structure group asymptotically reduces to $U(1)$, and $(A,\Phi)$ is then effectively approximated by an abelian $G_2$-monopole (resp. Calabi-Yau monopole) with Dirac singularity along $Q$. Taking a sequence of such $(A,\Phi)$ as $m\to +\infty$, the $L^2$-curvature in a fixed neighbourhood of $Q$ is of order $m$, and the normalised curvature density $m^{-1} |F|^2 dvol$ is expected to concentrate along $Q$, and converge weakly to the volume measure of $Q$ up to a constant factor. It is worth noting that this \emph{$L^2$-curvature goes to infinity} as $m\to +\infty$, which is a source of analytic difficulty from the viewpoint of Yang-Mills theory.

	Donaldson-Segal suggest that there is a one parameter family of enumerative problems for $G_2$-monopoles (resp. Calabi-Yau monopoles) parametrised by the mass $m$. As $m\to +\infty$, the limit turns into an enumerative problem for coassociative submanifolds (resp. special Lagrangians), with suitably defined weights counting some additional data such as Fueter sections.

	\begin{rmk}
	Donaldson-Segal are inspired by Taubes' relation between the Seiberg-Witten (SW) and Gromov-Witten (GW) invariants of a symplectic 4-manifold. Taubes considered a 1-parameter family of perturbed SW equations, with perturbation term $t\omega$ proportional to the symplectic form $\omega$ (analogous to the $G_2$-monopoles with a mass parameter). As $t\to +\infty$, the curvature of the SW solution concentrates near a pseudoholomorphic curve $\Sigma$ (analogous to the coassociative submanifold $Q$), and the local behaviour on the 2-planes normal to $\Sigma$ is modelled on the vortices in $\R^2$, and the variation of the vortices along $\Sigma$ is governed by a holomorphic section of a bundle over $\Sigma$ whose fibres are modelled on the moduli space of charge $k$ vortices (analogous to the monopole moduli space). 
	\end{rmk}

\begin{rmk}
	The Donaldson-Segal programme fits into the more general phenomenon that minimal surfaces emerge from the limits of a family of variational problems, as some parameter varies. A classical instance is that the critical points  $u_\epsilon: X\to \R$ of the  \emph{Allen-Cahn functional} in the $H^1$-Sobolev space,
	\[
	F_\epsilon(u) (v):= \int_X (\epsilon |dv|^2+ \frac{1}{4\epsilon^2}  (1-v^2)^2 ),
	\]
	can converge as $\epsilon\to 0$ in $L^1$ to a BV function $u_0: X\to \{ \pm 1 \}$, such that the interface is a two sided minimal hypersurface in $X$ \cite{PacardRitore}.  More recently, Stern et al. \cite{Stern}\cite{SternPigati} made a deep study of the self-dual Yang-Mills-Higgs energy
	\[
	E_\epsilon(u, \nabla):= \int_X |\nabla u|^2+ \epsilon^2 |F|^2+ \frac{1}{ 4 \epsilon^2} (1-|u|^2)^2,
	\]
	where $(u,\nabla)$ consists of a section $u$ of a given Hermitian line bundle $L\to X$, and a $U(1)$-connection $\nabla $ on $L$, and $F$ denotes the curvature of $\nabla$. Among other achievements, they extracted a stationary integral $(n-2)$-varifold inside $X$, from the weak limit of the energy measures as $\epsilon\to 0$.

	There are two major distinctive features of the Donaldson-Segal programme: first, the parameter $\epsilon$ is explicitly put into the functional in the Allen-Cahn or self-dual YMH setting, while the mass parameter $m:= \limsup_{x\to \infty} |\Phi|(x)$ is part of the \emph{asymptotic boundary condition}, rather than a parameter in the equation. Second, the Donaldson-Segal programme concerns $SU(2)$-gauge theory, which is  \emph{non-abelian}.

\end{rmk}

\begin{rmk}
(***)
After the completion of this paper, the author is made aware of the preprint by Parise-Pigati-Stern \cite{ParisePigatiStern} from one month ago, which studies a boundary value problem version of the Donaldson-Segal programme, and shows morally that the $\epsilon$-rescaled $SU(2)$ Yang-Mills-Higgs energy 
\[
\int_X \epsilon^{-1} |\nabla \Phi|^2+ \epsilon |F|^2
\]
$\Gamma$-converges to the mass functional for $(n-3)$-cycles in the $\epsilon\to 0$ limit. In particular, they show how to extract an $(n-3)$-dimensional integral current $Q$ from the large mass limit of Yang-Mills-Higgs solutions. In the setting of $G_2$/Calabi-Yau monopoles, they prove a calibration inequality for $Q$, which is saturated if and only if $Q$ is coassociative (resp. special Lagrangian).

\end{rmk}

\begin{rmk}
(***) After the completion of this paper, Oliveira and Fadel informed the author that they have some independent unpublished work in progress towards the Donaldson-Segal programme. We look forward to their future paper on the subject.
\end{rmk}

	\subsection{Asymptotic geometry of $G_2$/Calabi-Yau monopoles}\label{sect:Oliveira}

The main progress on the Donaldson-Segal programme is by Fadel-Nagy-Oliveira \cite{Oliveira1}. We will focus on the structure group $G=SU(2)$. The inner product on the Lie algebra $\mathfrak{g}=su(2)$ follows the convention $|a|^2= \frac{1}{2} \Tr (a^\dagger a)=-\frac{1}{2} \Tr(a^2)$.

\begin{Def}\label{Def:ACG2}
An irreducible  $G_2$-manifolds $(M,\phi)$ is called asymptotically conical (AC)  with rate $\nu_0<0$, if there is a $G_2$-cone $C(\Sigma)= (r_0,\infty)_{r}\times \Sigma$ with nearly K\"ahler cross section $(\Sigma, \omega_\Sigma, \Omega_\Sigma)$, 
\[
\phi_C= r^2 dr\wedge \omega_\Sigma+ r^3 \text{Re} \Omega_\Sigma, \quad \psi_C= r^4 \frac{\omega_\Sigma^2} {2}- r^3 dr\wedge \text{Im}\Omega_\Sigma,
\]
such that the complement of a compact set  is diffeomorphic to $C(\Sigma)$, and via this diffeomorphism
\[
|\nabla_C^j(  \phi- \phi_C  )|_{  g_C } =O( r^{\nu_0-j}) , \quad r(x)\to +\infty, \quad \forall j\geq 0.
\]
For ease of notation, we will extend $r(x)$ to a global smooth function on $M$ satisfying $\min_M r(x)>1$.
\end{Def}

\begin{Def}
An irreducible Calabi-Yau 3-fold $(M,\omega,\Omega)$ is asymptotically conical  (AC) with rate $\nu_0<0$, if there is a Calabi-Yau cone $C(\Sigma) = (r_0,\infty)_{r}\times \Sigma$ with Sasaki-Einstein cross section $(\Sigma, \lambda_\Sigma, \omega_{\Sigma, 1} ,\omega_{\Sigma, 2},\omega_{\Sigma, 3})$, 
\[
\omega_C=    rdr\wedge \lambda_\Sigma+ r^2\omega_{\Sigma,1}, \quad \Omega_C= (r^2dr+ r^3\sqrt{-1} \lambda_\Sigma)\wedge(\omega_{\Sigma,2}+ \sqrt{-1}\omega_{\Sigma,3})
\]
such that the complement of a compact set is diffeomorphic to $C(\Sigma)$, and  via this diffeomorphism
\[
|\nabla_C^j(   \omega- \omega_C  )|_{  g_C } =O( r^{\nu_0-j}) ,  \quad  |\nabla_C^j(   \Omega- \Omega_C  )|_{  g_C } =O( r^{\nu_0-j}) \quad r(x)\to +\infty,  \forall j\geq 0.
\]
\end{Def}

The main results in \cite{Oliveira1} concern the asymptotic geometry of a smooth $G_2$-monopole $(A,\Phi)$ on AC $G_2$-manifolds, but the arguments apply almost verbatim to Calabi-Yau monopoles on AC Calabi-Yau manifolds, so we will state them uniformly.

\begin{enumerate}
	\item  (Finite mass) If $(A,\Phi)$ has finite intermediate energy (\ref{eqn:intermediateenergy}),  and $F\in L^\infty(M)$, then the following limit exists:
	\[
	\lim_{ r(x)\to \infty} |\Phi(x)|=m.
	\]
	We call $m$ the \emph{mass} of the $G_2$-monopole (resp. Calabi-Yau monopole). Since the function $|\Phi|^2$ is subharmonic, $\Phi$ is identically zero if $m=0$. We shall focus on $m>0$.

	\item  (Asymptotic decay) Assume furthermore that $|F|\to 0$ as $r(x)\to +\infty$. Then along the asymptotically conical end, $|\nabla \Phi| \lesssim r(x)^{-(n-1)}$, and $|[ \Phi, \nabla \Phi  ]|+ |[\Phi, F]|$ decays exponentially in $r(x)$.

	\item (Pseudo-Hermitian-Yang-Mills limit) Furthermore, if $|F|=O(r(x)^{-2})$, then there is a principal $G$-bundle $P_\infty$ over $\Sigma$, together with a pair $(A_\infty, \Phi_\infty)$ such that $\Phi_\infty$ is a nonzero parallel section of the adjoint bundle for $P_\infty$ over $\Sigma$, and $A_\infty$ solves the pseudo-HYM equation in the $G_2$ setting,
	\[
	F_{A_\infty} \wedge \Omega_\Sigma =0,\quad F_{A_\infty}\wedge \omega_\Sigma^2=0,
	\]
	and respectively in the Calabi-Yau setting,
	\[
	F_{A_\infty} \wedge \omega_{\Sigma,2}=0,\quad F_{A_\infty} \wedge \omega_{\Sigma,3}=0,\quad  F_{A_\infty}\wedge \lambda_\Sigma\wedge \omega_{\Sigma,1}=0,
	\]
	and $(A, \Phi)|_{\{R\}\times \Sigma}\to (A_\infty, \Phi_\infty)$ uniformly as $R\to +\infty$ up to gauge.

	The parallel Higgs field reduces the limiting connection $A_\infty$ to the structure group $U(1)\subset SU(2)$, with first Chern class $\beta\in H^2(\Sigma)$, called the \emph{monopole class} of $(A,\Phi)$.

	\item (Bolgomolny trick)  The intermediate energy admits the topological formula in the $G_2$-case,
	\begin{equation}\label{eqn:intermediateenergy2}
\mathcal{E}^\psi(A,\Phi)= \int_M |\nabla \Phi|^2= \int_M |F\wedge \psi|^2= 2\pi m \langle \beta \cup \psi|_\Sigma, [\Sigma]\rangle,
	\end{equation}
	and respectively the following formula holds in the Calabi-Yau case,
	\begin{equation}
		\label{eqn:intermediateenergyCY2}
 \mathcal{E}^\Omega(A,\Phi)=\int_M |\nabla \Phi|^2= \int_M |F\wedge \text{Re}\Omega|^2= 2\pi m \langle \beta \cup \text{Re}\Omega|_\Sigma, [\Sigma]\rangle.
	\end{equation}
	Here $\psi|_\Sigma$ denotes the cohomology class defined by $\psi$ restricted to $\{R\}\times \Sigma$ for any large enough $R$.
	
	In the $G_2$ case, this formally follows from the Stokes formula and the fact that $d\Tr (\Phi F)= \Tr( \nabla \Phi\wedge F)$: 
	\[
	\begin{split}
		-2\int_M |\nabla \Phi|^2=	\int_M \Tr(  \nabla \Phi\wedge *\nabla \Phi  )=& \lim_{R\to \infty}\int_{ r(x)\leq R}  \Tr( \nabla \Phi\wedge F\wedge \psi)
		\\
		=& \lim_{R\to \infty} \int_{ r(x)=R}  \Tr (\Phi F\wedge \psi)
		\\
		=& \int_\Sigma \Tr (\Phi_\infty F_\infty\wedge \psi)
	\end{split}
	\]
	Locally on $\Sigma$, we can write the parallel section
	$\Phi_\infty=\text{diag}(im,-im)$ in  the matrix form , while $F_{A_\infty}= \text{diag}(F_L,-F_L)$, such that $\frac{-i}{2\pi }F_L$  represents the first Chern class $\beta$ of the $U(1)$-bundle  (which is $-c_1(L)$ in our convention), so the RHS evaluates to $-4\pi m \langle \beta \cup \psi|_\Sigma, [\Sigma]\rangle.$ The Calabi-Yau case is entirely analogous.

	\begin{rmk}\label{rmk:SO(3)1}
		For $G=SO(3)$, the Lie algebra is the same, so the same formula goes through. The caveat is that unless $P$ lifts to an $SU(2)$-bundle, in general only $L^{\otimes 2}$ is a well defined complex line bundle on $\Sigma$, and
	$2\beta:=[ -\frac{i}{\pi}F_L  ]\in H^2(\Sigma,\Z)$ is the first Chern class of a $U(1)$-subbundle of the $SO(3)$-bundle.
	\end{rmk}

	\item Based on these a priori estimates, \cite{Oliveira1} developed a Fredholm deformation theory for smooth $G_2$-monopoles with prescribed pseudo-HYM asymptote at infinity.
\end{enumerate}

		\begin{eg}
		Oliveira \cite{Oliveiraexample1}\cite{Oliveiraexample2}  constructed symmetric examples of a 1-parameter family of Calabi-Yau monopoles (resp. $G_2$-monopoles) $(A,\Phi)$  with $m\to +\infty$, on the Stenzel $T^*S^3$ (resp. Bryant-Salamon $\Lambda^2_-(\mathbb{CP}^2)$ and $\Lambda^2_-(S^4)$), such that 
		the $L^2$-curvature concentrates along the zero section $S^3$ (resp. $\mathbb{CP}^2$, $S^4$), and on the complement of the zero section the solutions $(A,\Phi)$ converges to an abelian Calabi-Yau monopole (resp. abelian $G_2$-monopole) with Dirac singularity along the zero section.

	\end{eg}

	\subsection{Large mass limit and calibrated cycles}\label{sect:largemasslimit}

	The goal of this paper is to make progress on the Donaldson-Segal programme, by relating the large mass limiting behaviour of the $G_2$-monopole (resp. Calabi-Yau monopole), to coassociative cycles (resp. special Lagrangian cycles) in the sense of Harvey-Lawson \cite{HarveyLawson}.

	\begin{Def}
		Inside a $G_2$-manifold $(M,\phi)$,
	a \emph{coassociative cycle} $Q$ is a 4-dimensional closed integral current, such that  $\norm{Q}$-a.e.
	the oriented tangent planes are calibrated by $\psi$, and the multiplicity function takes value in positive integers.
	\end{Def}

		\begin{Def}
		Inside a Calabi-Yau 3-fold $(M,\omega, \Omega)$,
		a \emph{special Lagrangian cycle} $Q$ is a 3-dimensional closed integral current, such that  $\norm{Q}$-a.e.
		the oriented tangent planes are calibrated by $\text{Re}(\Omega)$, and the multiplicity function takes value in positive integers.

	\end{Def}

	\begin{Main Setting}
			Let $(M,\phi)$ (resp. $(M,\omega, \Omega)$) be an asymptotically conical $G_2$-manifold (resp. Calabi-Yau 3-fold). Let $(A_i,\Phi_i)$ be a sequence of $G_2$-monopoles (resp. Calabi-Yau monopoles) satisfying the \emph{qualitative} hypotheses in \cite{Oliveira1}, namely the curvature is in $L^\infty$ and has quadratic asymptotic decay $|F_{A_i}|=O(r(x)^{-2})
			$	as $r(x)\to +\infty$, and the intermediate energy is finite. 
Consequently by \cite{Oliveira1}, 
$
		\lim_{r(x)\to +\infty }  |\Phi_i(x)|=m_i\geq 0,
$
		the asymptotic limit is pseudo-HYM,
		and the intermediate energy admits the topological formula (\ref{eqn:intermediateenergy2}) (resp. (\ref{eqn:intermediateenergyCY2})). We will assume that $m_i\to +\infty$ as $i\to +\infty$, and suppose the monopole class $\beta\in H^2(\Sigma)$ and the second Chern class of the underlying $SU(2)$ bundle $P$ are independent of $i$. We shall usually suppress the sequential index $i$ below. All constants are independent of large $m$.
	\end{Main Setting}

Our goal is to develop a structure theory of $G_2$-monopoles (resp. Calabi-Yau monopoles) when the mass $m$ is large, and describe the limit as $m\to +\infty$. %Therefore it is important to derive estimates which are uniform for large $m$.
	We can decompose the curvature $F\in \Omega^2(M, \mathfrak{g}_P)$ into the parts parallel to $\Phi$ and orthogonal to $\Phi$ in the Lie algebra factor: $F= F^\parallel+ F^\perp$. The parallel part is proportional to $\frac{\Tr(\Phi F)}{4\pi m}$ or $\frac{\Tr(\Phi F)}{4\pi |\Phi|}$, and its $L^1_{loc}$-convergence theory is summarized in the following theorem:

	\begin{thm}\label{mainthm:L1}
			There exists some $L^2$-harmonic 2-form $\sigma\in L^2\mathcal{H}^2(M)\simeq H^2_c(M)$ depending on $(A, \Phi)$ with $\norm{\sigma}_{L^2}\leq Cm$, such that the following convergence results hold:
	
	\begin{enumerate}
		\item (Strong convergence, see Prop. \ref{prop:weaklimitChernform}, Lem. \ref{lem:singularabelianmonopole}) After passing to subsequence, the following strong $L^1_{loc}$ limits exist as $i\to +\infty$:
		\[
		\begin{cases}
		\tilde{F}_\infty:= \lim_{i\to +\infty} (- \frac{\Tr(\Phi F)}{4\pi m}+ \sigma )= \lim_{i\to +\infty} (- \frac{\Tr(\Phi F)}{4\pi |\Phi| }+ \sigma ),
		\\
		\Phi_\infty:= \lim_{i\to +\infty} \frac{ |\Phi|^2-m^2}{2m}.
		\end{cases}
		\]
		We let $F_{\infty}= 2\pi \tilde{F}_\infty.$

		\item (Singular abelian $G_2$-monopole/Calabi-Yau monopole, see Lem. \ref{lem:singularabelianmonopole})  The limit $(F_\infty,\Phi_\infty)$ satisfies the following system in the distributional sense: in the $G_2$-monopole case, 
		\[
			\begin{cases}
				d F_\infty = 2\pi Q, 
				\\
				F_\infty \wedge \psi= *d\Phi_\infty,
			\end{cases}
\]
		while in the Calabi-Yau monopole case,
		\[
			\begin{cases}
				d F_\infty = 2\pi Q, 
				\\
				F_\infty \wedge \text{Re}\Omega= *d\Phi_\infty.
			\end{cases}
		\]

		\item (Calibrated cycle, see Thm. \ref{thm:weaklimitstructure}, Cor. \ref{cor:homology}) The $(n-3)$-current $Q$ is a coassociative cycle (resp. special Lagrangian cycle), contained inside a compact subset of $M$, and is represented by a smooth submanifold with integer multiplicity except on a subset of Hausdorff dimension at most $n-5$. 
		
		The homology class of  $[Q]\in H_{n-3}(M)$ is Poincar\'e dual to the image of the monopole class $\beta$ under the connecting map $H^2(\Sigma)\to H^3_c(M)$ in the 
		long exact sequence for cohomology with real coefficient for the pair $(M,\Sigma)$,
		\begin{equation*}
			\ldots \to H^{k-1}(\Sigma)\to H_c^k(M)\to H^k(M)\to H^k(\Sigma)\to \ldots. 
		\end{equation*}
		The total mass of $Q$ is finite, and admits the formula
		\[
		\text{Mass}(Q)= 
		\begin{cases}
			\int_Q \psi =	\langle  \beta \cup \psi|_\Sigma, [\Sigma]\rangle, \quad & \text{$G_2$ case},
			\\
			\int_Q \text{Re}\Omega =	\langle  \beta \cup \text{Re}(\Omega)|_\Sigma, [\Sigma]\rangle, \quad & \text{Calabi-Yau case}.
		\end{cases}
		\]

		\item (Smoothness and asymptotic decay, see Thm. \ref{thm:weaklimitstructure})  The pair $(F_\infty, \Phi_\infty)$ is smooth away from the support of $Q$. Near the infinity of $M$, then $\Phi_\infty= O(r(x)^{2-n} )$, and $\tilde{F}_\infty$ is asymptotic to  the harmonic 2-form on $\Sigma$ in the monopole class $\beta$.

		\item ($U(1)$-bundle and integrality, see Prop. \ref{prop:weaklimitChernform}, Cor. \ref{cor:homology}) We can require the $L^2$-harmonic forms $\sigma$ to represent integral classes in $H^2_c(M)$. Then 
		the closed form $\tilde{F}_\infty$ represents an integral class in $H^2(M\setminus \text{supp}(Q), \R)$. In particular, there exists a  connection $A_\infty$ on some $U(1)$-bundle over $M\setminus \text{supp}(Q)$ whose curvature 2-form is $F_\infty$.

	\end{enumerate}

	\end{thm}

	\begin{rmk}
	A conceptually significant aspect of this convergence theorem, is that assuming the existence of the sequence $(A,\Phi)$, then it produces a coassociative (resp. special Lagrangian) cycle, within the prescribed homology class. This reduces the highly  non-perturbative existence question for these calibrated cycles, to a question in gauge theory. This strategy is morally analogous to producing holomorphic curves by showing the nontriviality of some Seiberg-Witten invariant, using Taubes's famous work that $GW=SW$.
	\end{rmk}

	\begin{rmk}
	In the $G=SO(3)$ variant case, the same conclusions hold, except for those concerning integrality:  we have instead that $2\tilde{F}_\infty$ represents an integral class in $H^2(M\setminus \text{supp}(Q))$, and $2F_\infty$ is the curvature 2-form of some $U(1)$-connection.
	\end{rmk}

The perpendicular component $F^\perp$ turns out to be much smaller compared to $F^\parallel$ in the $L^1$-average sense; the intuition is that for very large $m$, the $G_2$-monopole (resp. Calabi-Yau monopole) $(A,\Phi)$ \emph{behaves as if the structure group is the abelian group} $U(1)$, in the \emph{generic region}. On the other hand, there is some non-generic curvature concentration locus $\mathcal{C}_{\Lambda_0}$, where $(A,\Phi)$ is genuinely non-abelian.   We show the \emph{concentration-decay dichotomy}: 
 the $s$-tubular neighbourhoods of $B_s(\mathcal{C}_{\Lambda_0}):=  \{   x\in M: \text{dist}(x, \mathcal{C}_{\Lambda_0}) \leq s  \}$  have small measure when $s$ is small, while the $L^2$ integral of $F^\perp$ \emph{decays exponentially} in the distance to $\mathcal{C}_{\Lambda_0}$.

 %; this effect is more visible in the $L^2$-norm than the $L^1$-norm.

\begin{thm}\label{mainthm:Fperp}
There is a subset $\mathcal{C}_{\Lambda_0}\subset M$, such that the perpendicular component $F^\perp$ satisfies the following when $m$ is large enough:

\begin{enumerate}
	\item  ($L^1_{loc}$-estimate, see Cor. \ref{cor:L1Fperp}) For any $r>1$, 
	\[
		\int_{ \{   r(x)<r  \} }   |F^\perp|\leq C m^{-1/2}(1+ m^{-1} r^{n-4} ).
	\]
In particular $F^\perp$ converges in $L^1_{loc}$-topology to zero as $i\to +\infty$.

	\item  (Measure bound, see Prop. \ref{prop:measurebound}) Whenever $m^{-1/2}\gg s\geq m^{-1}$, we have the Minkowski measure estimate	\[
	\text{Vol}( B_s( \mathcal{C}_{\Lambda_0}) \cap \{ r(x)<r \}    )
	\leq    C(ms^2 )^2  (1+ m^{-1} r^{n-4}).
	\]

	\item  ($L^2$-Exponential decay, see Prop. \ref{prop:expdecay2})   If $r>1$ and $s< r/C'$, then 
	\[
	\int_{  \{   r(x)<r \}  \setminus  B_s(\mathcal{C}_{\Lambda_0})      } |F^\perp|^2 dvol(x) \leq  Ce^{-ms} (m+ r^{n-4}),
	\]
	where $C,C'$ are independent of $m, s$. 
\end{enumerate}

\end{thm}

	While the abelianization effect dominates in the generic region, as measured by the local $L^1$-norm, the nonabelian nature of the $SU(2)$ structure group persists in the region of curvature concentration, as measured by the local $L^2$-norm. We assume the setting of Thm. \ref{mainthm:L1}. It turns out that as $i\to +\infty$, the $L^2$-curvature density is concentrated on the support of the calibrated cycle $Q$, and the weak limit of the $L^2$-curvature density is uniquely determined by $Q$.

	Let $Q_{sm}$ denote the smooth locus of the support of $Q$, and using the normal exponential map, we identify a tubular neighbourhood  $\mathcal{U}$ of $Q_{sm}$ with an open subset of the normal bundle. Inside $\mathcal{U}$, we can decompose the curvature form as
	\[
	F= F^{nn}+ F^{tn}+ F^{tt},
	\]
	where the superscript $t$ stands for `tangential factors', while $n$ stands for `normal factors'. Likewise we write $\nabla \Phi= \nabla^t \Phi+ \nabla^n \Phi$.

\begin{thm}\label{mainthm:L2}(See Prop. \ref{prop:measurelimit2})
The following holds for the limit of $L^2$-curvature densities:

\begin{enumerate}
	\item (Weak limit) We have the weak convergence of measures
	\[
	\begin{split}
	&	\lim_{i\to +\infty} \frac{1}{2\pi m} |\nabla \Phi|^2dvol = 	\lim_{i\to +\infty} \frac{1}{2\pi m} |F|^2dvol 
		\\= &
	\begin{cases}
		Q\wedge \psi,\quad &\text{$G_2$-monopole case},
		\\
		Q\wedge \text{Re}\Omega, \quad & \text{Calabi-Yau monopole case}.
	\end{cases}
	\end{split}
	\]
	
	\item  (Tangential contributions)  We have 
	\[
	\lim_{i\to +\infty} m^{-1} |F^{nt}|^2 dvol= \lim_{i\to +\infty} m^{-1} |F^{tt}|^2 dvol = \lim_{i\to +\infty} m^{-1} |\nabla^t \Phi|^2dvol=0.
	\]

	\item (Approximate monopoles on normal fibres)
	\[
	\lim_{i\to+\infty}   \int_\mathcal{U} m^{-1} |F^{nn}- *_3 \nabla^n \Phi|^2 dvol =0,
	\]
	where $*_3$ denotes the Hodge star on the 3-dimensional normal fibres.

\end{enumerate}

\end{thm}

Furthermore, as $m\to +\infty$, then along most of the normal fibres to $Q_{sm}$, the restriction of $(A,\Phi)$ is modelled on rescalings of monopoles over $\R^3$, and 
almost all the $L^2$-curvature in any fixed large open subset of $M$ can be accounted for by \emph{monopole bubbling} in the normal 3-planes along the support of $Q$. This intuition is captured by the more technical  `energy identity'  in Thm. \ref{thm:energyidentity}.

	\begin{rmk}
	It should be noted that the full Donaldson-Segal programme is much more ambitious than our main theorems. We make an extended discussion in Section \ref{sect:open} on some open analytic questions. 
	\end{rmk}

	%allow for admissible connections?

	\section{Some estimates on $G_2$/Calabi-Yau monopoles}

	\subsection{Yang-Mills-Higgs equation and monotonicity formula}

	Both $G_2$-monopoles and Calabi-Yau monopoles are special instances of critical points of the Yang-Mills-Higgs functional, namely they solve the Euler-Lagrange equation, called the Yang-Mills-Higgs equation.

	\begin{Notation}
	We will denote the rough Laplacian as $\Lap=\nabla^*\nabla =-\nabla_i\nabla_i$, with the same notation for functions, forms and covariant versions.  The structure group is chosen to be $G=SU(2)$, and the inner product on the Lie algebra $su(2)$ is $|a|^2=\frac{1}{2}\Tr (a^\dagger a)=- \frac{1}{2}\Tr(a^2)$ (beware of the slightly different convention in \cite{Oliveira1}). We will occasionally comment on the $G=SO(3)$ variant case.   For any adjoint valued differential forms $\alpha, \beta$, we write
	\[
	[\alpha\wedge \beta]= \sum_{I,J} [\alpha_I, \beta_J] dx^I \wedge dx^J.
	\]
	If either $\alpha$ or $\beta$ is an adjoint valued scalar function, then we drop the wedge notation and write $[\alpha, \beta]=[\alpha\wedge \beta]$. 
	\end{Notation}

	\begin{lem}\label{lem:YMH}\cite[Lem. 2.7]{Oliveira1}
	Let $(A,\Phi)$ be a $G_2$-monopole (resp. Calabi-Yau monopole). Then it solves the Yang-Mills-Higgs equation
	\begin{equation}
	\Lap \Phi=0,\quad d_A^* F=[ \nabla \Phi, \Phi ].
	\end{equation}
%	In particular $\Lap_A F= [[  F, \Phi ], \Phi]]-[  \nabla \Phi \wedge \nabla \Phi  ]$. 
	\end{lem}

	Solutions for the Yang-Mills-Higgs equation satisfy a well known \emph{monotonicity formula} for the normalised energy on balls. We work in the Main Setting, and apply this to the balls in our AC  $G_2$-manifolds (resp. Calabi-Yau 3-folds). Recall that $r(x)$ is a positively valued smooth function on $M$, which asymptotically gives the radial coordinate on the conical end $C(\Sigma)$.

	\begin{prop}\label{prop:monotonicity1}
		\cite[Thm. 4]{UhlenbeckSmith}  Let $(M,\phi)$ be an AC $G_2$-manifold (resp. Calabi-Yau 3-fold), and let $(A,\Phi)$ be a $G_2$-monopole (resp. Calabi-Yau monopole). Then for any $p\in M$, the normalised energy $e^{Cs/r(p)}s^{4-n} \int_{B(p,s)} |F|^2$ is monotone increasing for $0<s<r(p)/C'$. 		 In particular for radius $s\leq  r(p)/3$, we have
		\[
		s^{4-n}\int_{B(p, s)} ( |F|^2 + |\nabla \Phi|^2)  dvol\leq C r(p)^{4-n} \int_{B(p,r(p)/3) } ( |F|^2 + |\nabla \Phi|^2)  dvol.
		\]
	\end{prop}

	We need an effective bound in terms of the mass $m$ and monopole class $\beta\in H^2(\Sigma)$. This follows from a \emph{calibration argument}.  Recall that a $G_2$-monopole $(A,\Phi)$ has \emph{intermediate energy}
	\[
	E^\psi= \int_M |\nabla \Phi|^2= \int_M |F\wedge \psi|^2 = 3 \int_M |F^7|^2= 2\pi m \langle \beta\cup [\psi]|_\Sigma, \Sigma\rangle,
	\]
	where $F=F^7+ F^{14}$ is the type decomposition of the curvature according to $\Lambda^2=\Lambda^2_7\oplus \Lambda^2_{14}$.
	On the other hand, the \emph{second Chern form}\footnote{Here the trace for elements of $su(2)$ is zero, so $\Tr(F\wedge F)$ represents the second Chern class up to constant. In the $SO(3)$-variant, it represents the first Pontrjagin class up to constant.} is related to the energy via the formula
	\begin{equation}\label{eqn:Chernform}
	-\frac{1}{2}\Tr(F\wedge F)\wedge \phi =2 |F^7|^2- |F^{14}|^2.
	\end{equation}
	In the Calabi-Yau monopole analogue,
	\[
	E^\Omega=  \int_M |\nabla \Phi|^2= \int_M |F\wedge \text{Re}\Omega |^2 = 2 \int_M |F^6|^2= 2\pi m \langle \beta\cup [\text{Re}\Omega]|_\Sigma, [\Sigma]\rangle,
	\]
	where $F=F^6+F^8$ is the type decomposition of the curvature according to $\Lambda^2=\Lambda^2_6\oplus \Lambda^2_1\oplus \Lambda^2_8$ (here the 1-dimensional component vanishes by $F\wedge \omega^2=0$), and the second Chern form is related to the energy via the formula
	\[
	-\frac{1}{2}\Tr(F\wedge F)\wedge \omega= |F^6|^2- |F^8|^2.
	\]
	We will always assume a uniform bound on the second Chern class of the underlying $SU(2)$-bundle $P$.

	\begin{lem}\label{lem:monotonicity2}
		For $r\geq r_0$ sufficiently large, we have a uniform bound 
		\[
		\int_{r(x)\leq r } |F|^2  dvol(x)  \leq \begin{cases}
			E^\psi+ Cr^{n-4}, \quad  & \text{for $G_2$-monopoles},
			\\
 E^\Omega + Cr^{n-4}, \quad& \text{for Calabi-Yau monopoles},
		\end{cases}   
		\]
		where the constant $C$ depends only on the ambient manifold, the monopole class $\beta$, and the second Chern class of the bundle $P$, but not on the mass $m$. 
	\end{lem}

	\begin{proof}
	We focus on the AC $G_2$-manifolds with decay rate $\nu_0$ (\cf Def. \ref{Def:ACG2} for the notations).
There is an explicit primitive for the conical 3-form $\phi_C$:
	\[
	d( r\iota_{\partial_r} \phi_C)= \mathcal{L}_{r\partial_r} \phi_C= 3\phi_C.
	\]
	By the Poincar\'e lemma on the conical end, we can find a 2-form $\tau$ supported in the conical end region,  such that $d\tau=\phi$ on the complement of a fixed compact subset  $\{ r(x)<r_0 \}\subset M$, and $\tau-\frac{1}{3} r\iota_{\partial_r} \phi_C = O(r^{\nu_0+1})$ .

		Let $I(r)= 	\int_{r(x)\leq r } |F|^2  dvol(x) $, then by the coarea formula and the AC condition, for $r\geq r_0$ large enough, we have
		\[
		I'(r)= \int_{r(x)=r}   \frac{ |F|^2}{ |\nabla r|}  d\text{Area}_g\geq  \int_{r(x)=r} |F|^2 (1-Cr^{\nu_0})  d\text{Area}_{g_C}.
				\]
		By taking $r_0$ large enough, we may assume $Cr^{\nu_0}<\frac{1}{2}$.
		On the other hand, 
		\[
	\tilde{C}=	\int_{r(x)\leq r} 	\frac{1}{2}\Tr(F\wedge F)\wedge (\phi- d\tau )
		\]
		is the pairing between the second Chern class with a fixed class in $H^3_c(M)$, so is uniformly bounded by assumption, and independent of $r\geq r_0$. By the Stokes formula, 
	\[
	\begin{split}
			\int_{r(x)=r} 	\frac{1}{2}\Tr(F\wedge F)\wedge\tau
			=&  \int_{r(x)\leq r} \frac{1}{2}\Tr(F\wedge F)\wedge d\tau
			\\
			=&  \int_{ r(x)\leq r} 	\frac{1}{2}\Tr(F\wedge F)\wedge \phi -\tilde{C}
			\\
			=&  I(r)- 3 \int_{r(x)\leq r} |F^7|^2 -\tilde{C}\geq I(r)- E^\psi- \tilde{C}.
	\end{split}
	\]
		where the final line uses the Chern form formula (\ref{eqn:Chernform}) and the intermediate energy formula above. We now apply the pointwise \emph{calibration type inequality} on the link $\Sigma$,
		\[
		\frac{1}{2}\Tr(F\wedge F)\wedge  \iota_{\partial_r} \phi_C \leq |F|^2 d\text{Area}_{g_C},
		\]
		to deduce for $r\geq r_0$ the following differential inequality:
		\begin{equation}
		\begin{split}
		I(r)- E^\psi-\tilde{C}\leq \frac{r}{3} \int_{r(x)=r} |F|^2 (1+ Cr^{\nu_0}) d\text{Area}_{g_C} 
		\leq  \frac{r}{3} I'(r) (1+Cr^{\nu_0}).
		\end{split}
		\end{equation}
		Thus whenever $I(r)> E^\psi+\tilde{C}$, we have
		\[
		\frac{d}{dr} \log (I(r)- E^\psi-\tilde{C}) \geq \frac{3}{r} - Cr^{-1+\nu_0}.
		\]
		Suppose $I(r_1)\geq E^\psi+ \tilde{C}+ A r_1^3$ for some $r_1\geq r_0$, and $A>0$. Then by ODE comparison, 
		\[
		I(r)- E^\psi-\tilde{C} \geq A r^3 \exp (- C\int_{r_1}^r t^{-1+\nu_0} dt    )\geq A r^3 \exp(-C r_1^{\nu_0}), \quad \forall r\geq r_1,
		\]
		and in particular the LHS is always positive for $r\geq r_1$. In contrast, the asymptotic geometry of $G_2$-monopoles is modelled on an abelian monopole in the class $\beta\in H^2(\Sigma)$, so 
		\[
		\limsup_{r\to +\infty}  r^{-3} I(r) \leq C(\beta)
		\]
		depending only on the monopole class on the nearly K\"ahler cross section, but not on the mass $m$. Combining these two bounds shows that $A$ is a priori bounded, so we conclude 
		\[
		I(r)-E^\psi-\tilde{C}\leq C r^3,\quad \forall r\geq r_0.
		\]
The $G_2$ case of the Lemma follows by absorbing the constant $\tilde{C}$ into $C$. The Calabi-Yau case is completely analogous.
	\end{proof}

	Combining Lemma \ref{lem:monotonicity2} and Prop. \ref{prop:monotonicity1}, and the intermediate energy formula, we arrive at the following:
	
	\begin{cor}\label{cor:monotonicity}
For any $p\in M$ and radius $s\lesssim r(p)$, we have the Morrey norm estimate:
		\[
		s^{4-n}\int_{B(p, s)} ( |F|^2 + |\nabla \Phi|^2)  dvol\leq C(m r(p)^{4-n} +1).
		\]
		The constant depends only on the ambient geometry, the monopole class $\beta$, and the second Chern class of $P$.
	\end{cor}

We will  frequently encounter the following Green integrals, which can be bounded by applying the monotonicity formula estimate to the dyadic annuli.

\begin{cor}\label{cor:monotonicity2}
For any $p\in M$, we have the bound
\[
\int_M   \frac{ |F|^2+ |\nabla \Phi|^2}{  \max( d(x,p), s)^{n-2}    } dvol(x)  \leq 
\begin{cases}
C(m r(p)^{4-n}+1) s^{-2}  ,\quad & 0<s\lesssim r(p),
\\
C(ms^{4-n}+1) s^{-2},\quad & s\gtrsim r(p).
\end{cases}
\]
%and
%\[
%\int_M   \frac{ |\nabla \Phi|^2}{  \max( d(x,p), s)^{n-2}    } dvol(x)  \leq 
%\begin{cases}
%	C(m r(p)^{4-n}+1) s^{-2}  ,\quad & 0<s\lesssim r(p),
%	\\
%	C(ms^{4-n}+1) s^{-2},\quad & s\gtrsim r(p).
%\end{cases}
%\]
\end{cor}

	\subsection{$\epsilon$-regularity for Yang-Mills-Higgs equation}

	The Yang-Mills-Higgs equation satisfies a well known $\epsilon$-regularity theorem.

	\begin{prop}\label{prop:epsilonregularity}
		\cite[Prop. 3.5]{Oliveira1} Let $(A,\Phi)$ be a solution of the Yang-Mills-Higgs equation on a geodesic ball $B(p,s)$ which is $C^3$-uniformly equivalent to the Euclidean ball. There is a constant $\epsilon_1>0$ depending only on the ambient geometry, such that if
		\[
		s^{-(n-4)} \int_{B(p,s)} |F|^2+ |\nabla \Phi|^2 \leq \epsilon_1,
		\]
		then  on  the smaller ball $B(p,s/2)$,
		\[
		\sup_{B(p,s/2)}   (s^2|F| + s^2 |\nabla \Phi|+ s|\nabla F| ) \leq Cs^{-(n-4)} \int_{B(p,s)} |F|^2+ |\nabla \Phi|^2.
		\]
	\end{prop}

	\subsection{Higgs field estimate}

	Let $(A,\Phi)$ be a $G_2$-monopole (resp. Calabi-Yau monopole) on an AC $G_2$-manifold (resp. Calabi-Yau 3-fold) as in the Main Setting.
	The Yang-Mills-Higgs equation implies the Poisson equation 
		\begin{equation}\label{eqn:LapPhisquare}
		\Lap (|\Phi|^2-m^2) = -2|\nabla \Phi|^2.
	\end{equation}
	We now quantify how  $|\Phi|$ clusters around the mass $m$.

\begin{lem}\label{lem:smallPhi}
The following estimates hold for the Higgs field $\Phi$:
\begin{enumerate}
	\item  There is an upper bound $|\Phi|\leq m$. 
	
	\item The magnitude for $|\Phi|$ satisfies
	\begin{equation*}
		||\Phi|^2(p)-m^2| \leq    C\int_M \frac{1}{ d(x,p)^{n-2}  } |\nabla \Phi|^2dvol(x) ,\quad \forall p\in M.
	\end{equation*}
	
	\item Moreover, for balls $B(p,s)$ with radius $s<r(p)/3$, we have the $L^1$-gradient estimate
	\begin{equation}\label{eqn:smallPhi2}
		\int_{B(p,s)} | |\Phi|^2-m^2|+   s|\nabla |\Phi|^2| 
		\leq C s^n\int_{M} \frac{ |\nabla \Phi|^2   }{ \max(d(x,p), s)^{n-2}  } dvol(x) .
	\end{equation}
	and the weak $L^1$ Hessian estimate
	\[
	\norm{ \nabla^2( |\Phi|^2) }_{L^{1,\infty}(B(p,s)) } \leq C s^{n-2}\int_{M} \frac{ |\nabla \Phi|^2   }{ \max(d(x,p), s)^{n-2}  } dvol(x) .
	\]
\end{enumerate}

\end{lem}

\begin{proof}
\textbf{Item 1}. Since $m=\lim_{r(x)\to +\infty}|\Phi|$, and $\Lap (|\Phi|^2-m^2)\leq 0$, the maximum principle  implies $|\Phi|\leq m$. 

\textbf{Item 2}. 	In the Poisson equation (\ref{eqn:LapPhisquare}),
the forcing term satisfies the $L^1$ bound $\int_M  |\nabla \Phi|^2\leq Cm$ by the intermediate energy formula (\ref{eqn:intermediateenergy2})(\ref{eqn:intermediateenergyCY2}), and asymptotically $|\Phi|^2-m^2\to 0$ as $r(x)\to +\infty$.
These imply the Green representation formula
\[
(|\Phi|^2-m^2)(p)= -2\int_M G(p,x) |\nabla \Phi|^2(x) dvol(x).
\]
On asymptotically conical manifolds with $n>2$, the Green function is uniformly equivalent to $d(p,x)^{2-n}$, hence item 2 follows.

%We apply the mean value inequality to the Poisson equation (\ref{eqn:LapPhisquare}), to see
%\[
%||\Phi|^2-m^2 | (p)  \leq C\dashint_{ B(p,r)  } ||\Phi|^2-m^2| dvol +C \int_{B(p,r)} \frac{1}{ d(x,p)^{n-2}  } |\nabla \Phi|^2dvol(x) .
%\]
%By Lemma \ref{lem:|Phi|closetom}, $\dashint_{ B(p,r)  } ||\Phi|^2-m^2| dvol \leq Cm r^\nu$, so (\ref{eqn:smallPhi1}) follows. 
	
\textbf{Item 3}. 
The ball $B(p,3s/2)$ is contained either in a fixed compact region, or in an annulus $\{    r\lesssim r(x)\lesssim 2r\}$.  Applying Fubini theorem to item 2, we see
\[
\begin{split}
& \int_{B(p,3s/2)}  ||\Phi|^2-m^2 |(y) dvol(y)  
\\
& \leq C\int_{B(p,3s/2)} \int_M \frac{1}{ d(x,y)^{n-2}  } |\nabla \Phi|^2(x) dvol(x) dvol(y)
\\
& \leq   C s^n\int_M \frac{1}{ \max(d(x,p), s)^{n-2}  } |\nabla \Phi|^2dvol(x) .
\end{split}
\]
Elliptic regularity for the Poisson equation (\ref{eqn:LapPhisquare}) on $B(p,3s/2)$ implies
\[
\begin{split}
&	s^2\norm{ \nabla^2( |\Phi|^2) }_{L^{1,\infty}(B(p,s)) } +s \int_{B(p,s)} |\nabla |\Phi|^2| 
	\\
\leq & C\int_{B(p,3s/2)} ||\Phi|^2-m^2 | + Cs^2 \int_{B(p,3s/2)} |\Lap |\Phi|^2|
\\
\leq & C s^n\int_M \frac{1}{ \max(d(x,p), s)^{n-2}  } |\nabla \Phi|^2dvol(x) .
\end{split}
\]
\end{proof}

	\begin{rmk}\label{rmk:smallPhi}
	In particular, we can cover the dyadic annuli by a bounded number of balls with radius comparable to $r(p)$, to deduce
		\[
\int_{ r/2\leq  r(x)\leq r} |m^2-|\Phi|^2 | + r  |   \nabla |\Phi|^2|     |  dvol(x) \leq Cm r^2 ,\quad \forall r>1.
	\]
	Summing over dyadic scales, we see
	\[
	\int_{   r(x)\leq r} |m^2-|\Phi|^2 | + r  |   \nabla |\Phi|^2|   dvol(x) \leq Cm r^2 .
	\]
Since $|m^2- |\Phi|^2| = (m-|\Phi|)(m+ |\Phi|)\geq m |m-|\Phi||$, this implies
	\[
	\sup_{r>1} 	r^{n-2}\dashint_{  r(x)\leq r} |m-|\Phi| | dvol(x) \leq C .
	\]
For fixed $m$, assuming $|F|$ has quadratic decay, then elliptic regularity recovers the result in \cite[Page 31-42]{Oliveira1} that $m-|\Phi|(x)=O(r(x)^{2-n})$. 
	\end{rmk}

	\subsubsection{An integral Hessian estimate}

	Here we prove an integral Hessian estimate, which is not directly used in the main theorems, but seems interesting in its own right.

\begin{lem}
	For any solution of the Yang-Mills-Higgs equation, we have
	\[
	\Lap( \nabla \Phi)= [ [\nabla \Phi,\Phi], \Phi   ]- 2*[*F\wedge \nabla \Phi].
	\]
	In particular,
	\[
	\frac{1}{2} \Lap|\nabla \Phi|^2 + |\nabla^2\Phi|^2= \langle \nabla \Phi, \Lap (\nabla \Phi)\rangle= -2\langle \nabla \Phi, *[*F\wedge \nabla \Phi]\rangle - |[\Phi, \nabla \Phi]|^2.
	\]
\end{lem}

\begin{proof}
	This follows from the Yang-Mills-Higgs equation, the Bochner formula and Lie algebra identities, see \cite[Lem. 2.8]{Oliveira1}.
\end{proof}

\begin{cor}\label{cor:LapgradPhi}
	For any solution of the Yang-Mills-Higgs equation,
	\begin{equation}\label{eqn:LapgradPhi}
		\Lap |\nabla \Phi| + \frac{ |\nabla^2 \Phi|^2   }{  n  |\nabla \Phi| }	+  \frac{ |[\Phi, \nabla \Phi]|^2 }{  |\nabla \Phi|   }\leq C|F||\nabla \Phi|.
	\end{equation}
\end{cor}

\begin{proof}
	The above Lemma implies
	\[
	\frac{1}{2} \Lap|\nabla \Phi|^2 + |\nabla^2\Phi|^2 +   |[\Phi, \nabla \Phi]|^2 \leq C|F| |\nabla \Phi|^2.
	\]
	By computing the Laplacian of $\sqrt{|\nabla \Phi|^2+ \delta}$ for $\delta>0$, we deduce
	\[
	\begin{split}
		& 	\Lap \sqrt{|\nabla \Phi|^2+ \delta} + \frac{1}{  \sqrt{|\nabla \Phi|^2+ \delta}  }(|\nabla^2 \Phi|^2- \sum_i | \langle \nabla_i \nabla\Phi,  \frac{ \nabla \Phi}{  \sqrt{|\nabla \Phi|^2+ \delta}   }\rangle |^2   )	+  \frac{ |[\Phi, \nabla \Phi]|^2 }{   \sqrt{|\nabla \Phi|^2+ \delta}    }
		\\
		\leq &  C|F| \frac{  |\nabla \Phi|^2}{   \sqrt{|\nabla \Phi|^2+ \delta}}\leq C|F| |\nabla \Phi|.
	\end{split}
	\]

	Now we claim the Kato inequality
	\[
	\sum_i  | \langle \nabla_i \nabla\Phi,  \frac{ \nabla \Phi}{  \sqrt{|\nabla \Phi|^2+ \delta}   }\rangle |^2  
	\leq \sum_i  | \langle \nabla_i \nabla\Phi,  \frac{ \nabla \Phi}{  |\nabla \Phi|   }\rangle |^2 \leq \frac{n-1}{n}|\nabla^2 \Phi|^2.
	\]
	Here $\Phi$ is valued in the adjoint $su(2)$, but we can separately treat the 3 components of $su(2)$, so the problem reduces to real valued $\Phi$. Without loss $e_1= \frac{ \nabla \Phi}{  |\nabla \Phi|   }$, and we observe 
	\[
	\sum_{i\neq 1}|\nabla_1 \nabla_i \Phi|^2 \leq \frac{1}{2} \sum_{i\neq j} |\nabla_i \nabla_j \Phi|^2,
	\]	
	while the harmonicity $\Lap \Phi=0$ implies by Cauchy-Schwarz that
	\[
	|\nabla_1\nabla_1 \Phi|^2 =|\sum_{i>1} \nabla_i\nabla_i \Phi|^2 \leq (n-1) \sum_{i>1}|\nabla_i \nabla_i \Phi|^2, 
	\]
	whence $  \frac{n}{n-1} |\nabla_1\nabla_1 \Phi|^2\leq \sum_1^n |\nabla_i \nabla_i \Phi|^2$. Summing up the two contributions,
	\[
	\sum_i  | \langle \nabla_i \nabla\Phi,  \frac{ \nabla \Phi}{  |\nabla \Phi|   }\rangle |^2=	|\nabla_1 \nabla \Phi|^2 \leq \frac{n-1}{n}  \sum_{i,j} |\nabla_i \nabla_j \Phi|^2 ,
	\]
	so the Kato inequality follows.

	It follows that 
	\[
	\Lap \sqrt{|\nabla \Phi|^2+ \delta} + \frac{ |\nabla^2 \Phi|^2   }{  n\sqrt{|\nabla \Phi|^2+ \delta}  }	+  \frac{ |[\Phi, \nabla \Phi]|^2 }{   \sqrt{|\nabla \Phi|^2+ \delta}    }\leq C|F||\nabla \Phi|.
	\]
	Taking $\delta\to 0$ gives (\ref{eqn:LapgradPhi}). 
\end{proof}

We apply this to the $G_2$-monopoles (resp. Calabi-Yau monopoles) on AC manifolds.
The coercivity term in (\ref{eqn:LapgradPhi}) leads to an integral Hessian estimate.

\begin{prop}\label{prop:L1Hessian}
	For any $p\in M$ and $s<\frac{1}{3}r(p)$, we have
	\begin{equation}\label{eqn:HessianL2}
	\begin{split}
		\int_{B(p,s) }  \frac{ |\nabla^2 \Phi|^2 + |[F,\Phi]|^2   }{  |\nabla \Phi| }	
		\leq   Cs^{n-2} \int_M \frac{  |\nabla \Phi||F|   }{\max(d(x,p),s)^{n-2}}dvol(x).
	\end{split}
	\end{equation}
	In particular for any $r>1$,
	\[
		\int_{r<r(x)<2r }  \frac{ |\nabla^2 \Phi|^2 + |[F,\Phi]|^2  }{  |\nabla \Phi| }	
	\leq   C(m+ r^{n-4}).
	\]
\end{prop}

\begin{proof}

	By Cor. \ref{cor:LapgradPhi}, we have $\Lap |\nabla \Phi|\leq C|F||\nabla \Phi|$. By the growth bound in Cor. \ref{cor:monotonicity2}, the Green representation
	\[
	w(x)= \int_M G(x,y) |\nabla \Phi| |F| (y) dvol(y) \leq  C\int_M d(x,y)^{2-n} |\nabla \Phi| |F| (y) dvol(y)
	\]
	is well defined. Comparing $w$ with $|\nabla \Phi|$, we deduce
	\[
	|\nabla\Phi|(x)\leq C\int_M d(x,y)^{2-n} |\nabla \Phi| |F| (y) dvol(y). 
	\]
	For  any $s< r(p)/3$, the ball $B(p,2s)$ is either contained in a fixed compact subset of $M$, or in some annulus region. 
	Using Fubini theorem, 
	\begin{equation*}
		\int_{B(p,2s) } |\nabla \Phi|  \leq 
		Cs^n\int_M \frac{1}{\max(d(x,p),s)^{n-2}} |\nabla \Phi||F| dvol(x).
	\end{equation*}

	Now let $\eta$ be a nonnegative cutoff function supported on $B(p,2s)$, which is equal to one on $B(p,s)$ and satisfies $s|d\eta|+ s^2|\Lap \eta|\leq C$. Upon integration by parts, we estimate
	\[
|	\int_{B(p,2s)} \eta \Lap |\nabla \Phi| |=| \int_{B(p,2s)} |\nabla \Phi| \Lap \eta| \leq Cs^{-2}\int_{B(p,2s)} |\nabla \Phi| .
	\]
	By  (\ref{eqn:LapgradPhi}), we get an integral estimate on the coercivity terms,
	\begin{equation*}
		\begin{split}
			\int_{B(p,s) }  \frac{ |\nabla^2 \Phi|^2   }{  n  |\nabla \Phi| }	+  \frac{ |[\Phi, \nabla \Phi]|^2 }{  |\nabla \Phi|   }\leq  &    C\int_{B(p,2s)} |F||\nabla \Phi| +    Cs^{-2} \int_{B(p,2s)} |\nabla \Phi| 
			\\
			\leq  &  Cs^{n-2} \int_M \frac{  |\nabla \Phi||F|   }{\max(d(x,p),s)^{n-2}}dvol(x).
		\end{split}
	\end{equation*}
Now $[F_{ij}, \Phi]= (\nabla_i\nabla_j-\nabla_j\nabla_i)\Phi$, so $|[F,\Phi]|^2\leq  2 |\nabla^2\Phi|^2$, hence (\ref{eqn:HessianL2}) follows. The `in particular' statement follows by covering $\{  r<r(x)<2r \}$ with finitely many balls of radius $\sim r$, and applying Cor. \ref{cor:monotonicity2}. 
\end{proof}

	\subsection{Differential inequalities and exponential decay}

	On $M\setminus \Phi^{-1}(0)$, we can decompose the adjoint bundle $\mathfrak{g}_P= \mathfrak{g}_P^\parallel +  \mathfrak{g}_P^\perp$, by setting
	\[
	\mathfrak{g}_P^\parallel= \ker( ad_{{\Phi(x)}}   : \mathfrak{g}_P\to \mathfrak{g}_P  )
	\]
	and $\mathfrak{g}_P^\perp$ as its orthogonal complement. For $\mathfrak{g}=su(2)$, the parallel subbundle $\mathfrak{g}_P^\parallel$ is generated by $\Phi$. We decompose sections of $\mathfrak{g}_P$ as $\chi= \chi^\parallel+ \chi^\perp$. For the Lie algebra $su(2)$ with inner product $- \frac{1}{2}\Tr(a^2)$, we have
	\begin{equation}\label{eqn:Liealgebranorm}
		|[\chi, \Phi]|= 2 |\Phi| |\chi^\perp|.
	\end{equation}
The goal of this Section is to estimate the perpendicular component of the curvature $F$. 
We will  use the following pointwise inequalities.

\begin{lem}\label{lem:Liealgebra}
For $G_2$-monopoles, 
\[
3|F|^2\geq  |F\wedge \psi|^2= |\nabla \Phi|^2, \quad 3|[F,\Phi]|^2\geq |[\nabla \Phi,\Phi]|^2,
\]
while for Calabi-Yau monopoles
\[
2|F|^2\geq  |F\wedge \text{Re}\Omega|^2= |\nabla \Phi|^2, \quad 2|[F,\Phi]|^2\geq |[\nabla \Phi,\Phi]|^2. 
\]
\end{lem}

	\subsubsection{Exponential decay of perpendicular curvature}
	
The following Lemma is based on Bochner formula computations, and is crucial for the exponential decay of the perpendicular curvature components \cite{Oliveira1}.

	\begin{lem}\label{lem:commutator1}
	Let $(A,\Phi)$ solve the Yang-Mills-Higgs equation. Then
%	\begin{equation}\label{eqn:commutator1}
%	\frac{1}{2} \Lap |[\Phi, \nabla \Phi] |^2 + |\nabla [\Phi, \nabla \Phi]|^2+ |  [\Phi, [\Phi, \nabla \Phi]]|^2 \leq C|F| |[ \Phi,\nabla \Phi  ]|^2 + |\nabla \Phi| |[  F,\Phi  ]||[  \Phi, \nabla \Phi ]|,
%	\end{equation}
	\begin{equation}\label{eqn:commutator1}
	\begin{split}
		& \frac{1}{2} \Lap |[\Phi, F] |^2 + |\nabla [\Phi, F]|^2+  4|\Phi|^2  |[\Phi, F] |^2
		\\
		\leq & C(|Rm|+ |F|) |[F,\Phi]|^2+ C |\nabla \Phi|   |[\nabla \Phi,\Phi] | |[F,\Phi]|    +C |\sum_i [\nabla_iF, \nabla_i\Phi]  | |[F,\Phi]|.
	\end{split}
	\end{equation}

	\end{lem}

	\begin{proof}
As in the proof of \cite[Lemma 5.3]{Oliveira1}, 
	\[
	\begin{split}
			\langle [F,\Phi], \Lap [F,\Phi]\rangle& = - |[[F,\Phi],\Phi]|^2- \langle [F,\Phi],[[\nabla\Phi\wedge \nabla \Phi], \Phi] \rangle + \langle [F,\Phi], Rm([F,\Phi])\rangle
			\\
			& + \langle [F,\Phi], (F\cdot [F,\Phi])\rangle -2 \langle [F,\Phi], \sum_i [\nabla_iF, \nabla_i\Phi] \rangle,
	\end{split}
	\]
so 
\[
	\begin{split}
	& \frac{1}{2} \Lap |[\Phi, F] |^2 + |\nabla [\Phi, F]|^2+  |[[F,\Phi],\Phi]|^2
	\\
	\leq & C(|Rm|+ |F|) |[F,\Phi]|^2+ C |\nabla \Phi|   |[\nabla \Phi,\Phi] | |[F,\Phi]|    +C |\sum_i [\nabla_iF, \nabla_i\Phi]  | |[F,\Phi]|.
\end{split}
\]
By the Lie algebra identity (\ref{eqn:Liealgebranorm}), we obtain $|[[F,\Phi],\Phi]|^2=4|[F,\Phi]|^2|\Phi|^2$ (beware the Lie algebra norm convention is different from \cite{Oliveira1}). 
	\end{proof}

	\begin{cor}\label{cor:commutator1}
		There is a small enough constant $\epsilon_2>0$ such that the following holds.
		For a $G_2$-monopole (resp. Calabi-Yau monopole), suppose on an open subset we have $m\geq |\Phi|\geq \frac{m}{2}$ and $m^{-2}|Rm| + m^{-2}|F|+ m^{-3} |\nabla F| \leq \epsilon_2$,  then
		\begin{equation}\label{eqn:[Phi,F]modulus1}
		\frac{1}{2}	\Lap   |[\Phi,F]|^2  +(1-C\epsilon_2) |\nabla [F,\Phi]|^2+   (1-C\epsilon_2) m^2   |[\Phi,F]|^2 \leq 0.
			\end{equation}
	\end{cor}

	\begin{proof}
		(\cf \cite[Lem. 5.2, Cor. 4.4]{Oliveira1}) We estimate the bad term $|\sum_i [\nabla_iF, \nabla_i\Phi]  | $ by the parallel-perpendicular decomposition and Lemma \ref{lem:Liealgebra},
		\[
		\begin{split}
			& |\sum [\nabla_iF, \nabla_i\Phi]  | \leq C|(\nabla F)^\perp| |\nabla \Phi|+ C |\nabla F| |\nabla \Phi)^\perp| 
			\\
			\leq & C( |\Phi|^{-1}  |F|   | [\nabla F,\Phi]| + |\nabla F| |F^\perp| )
			\\
			\leq & C( |\Phi|^{-1}  |F|   (| \nabla[ F,\Phi]| +   | [ F,\nabla\Phi]| ) +   |\nabla F|  |F^\perp|)
			\\
			\leq &  C(     |\Phi|^{-1}  |F|   (| \nabla[ F,\Phi]| +   | F| |(\nabla \Phi)^\perp|+ |(\nabla \Phi)^\perp )| |F|) +   |\nabla F|  |F^\perp|) 
			\\
			\leq &  C(     |\Phi|^{-1}  |F|   (| \nabla[ F,\Phi]| +   | F| |F^\perp| ) +   |\nabla F|  |F^\perp|) 
			\\
			\leq & C     |\Phi|^{-1}  |F|  | \nabla[ F,\Phi]| + C ( |\Phi|^{-1}| F|^2  +   |\nabla F|)  |F^\perp| 
			\\
			\leq  & C     |\Phi|^{-1}  |F|  | \nabla[ F,\Phi]| + C ( |\Phi|^{-2}| F|^2  + |\Phi|^{-1}  |\nabla F|)  |[F,\Phi]| .
		\end{split}
		\]
Thus for $m\geq |\Phi|\geq \frac{m}{2}$ and $m^{-2}|Rm| + m^{-2}|F|+ m^{-3} |\nabla F| \leq \epsilon_2$ small enough, the RHS of (\ref{eqn:commutator1}) is bounded by
$
C\epsilon_2  ( m^2|[F,\Phi]|^2 + m |\nabla [\Phi,F]| |[F,\Phi]| ),
$
which can be absorbed by the good terms on the LHS of (\ref{eqn:commutator1}).
	\end{proof}

	\begin{prop}\label{prop:expdecay}
	(Exponential decay) There is some $\epsilon_3>0$ such that the following holds for all sufficiently large $m$. Suppose that  $(A,\Phi)$ is a $G_2$-monopole (resp. Calabi-Yau monopole), such that on a geodesic ball $B(p,s)$ with radius $s>2m^{-1}$, where the metric is $C^3$-uniformly equivalent to the Euclidean metric, $ 	\frac{m}{2}\leq |\Phi|\leq m, $ and 
	\[
  \int_{B(y, 2m^{-1}) } |F|^2dvol(x) \leq \epsilon_3 m^{4-n} , \quad \forall B(y,2m^{-1})\subset B(p,s)
	\]
  then there is an estimate uniform in large $m$,
	\begin{equation}\label{eqn:expdecay}
		|F^\perp|^2 (p) \leq \frac{C}{ s^n } e^{-ms} \int_{B(p,s)} |F^\perp|^2.
	\end{equation}
	\end{prop}

	\begin{proof}
	By the $\epsilon$-regularity in Prop. \ref{prop:epsilonregularity}, when $\epsilon_3\leq \epsilon_1$, then the following holds for any $B(y,2m^{-1})\subset B(p,s)$,
	\[
		\sup_{B(y,m^{-1})}   (m^{-2}|F| + m^{-2} |\nabla \Phi|+ m^{-1}|\nabla F| ) \leq Cm^{n-4} \int_{B(y,2m^{-1})} |F|^2+ |\nabla \Phi|^2\leq C\epsilon_3.
	\]
We assume that $C\epsilon_3\leq \epsilon_2\ll 1$, then
Cor. \ref{cor:commutator1} implies $	\Lap   |[\Phi,F]|^2  +   \frac{3}{2} m^2   |[\Phi,F]|^2 \leq 0$, hence
	\[
		|[\Phi,F]|^2 (p) \leq \frac{C}{ s^n } e^{-ms} \int_{B(p,s)} |[\Phi,F]|^2.
	\]
	But since $ 	\frac{m}{2}\leq |\Phi|\leq m $, the quantity $|[\Phi, F]|=2 |\Phi||F^\perp|$ is uniformly equivalent to $m|F^\perp|$. 
	\end{proof}

%	\begin{proof}
%			This is a variant of the mean value property for subharmonic functions. We let
%	\[
%	I(t)= \int_{B(x,t)} |[\Phi,F]|^2 dvol,
%	\]
%	then
%	\[
	%	I'(t)=\int_{\partial B(x,t)}  |[\Phi,F]|^2 d\text{Area} 
%	\]
%	 and
%	\[
%	\begin{split}
%			I''(t)& =\int_{\partial B(x,t)} \nabla_\nu  |[\Phi,F]|^2 d\text{Area} + \int_{\partial B(x,t)} H |[\Phi,F]|^2   d\text{Area} 
%			\\
%			& = -\int_{B(x,t)} \Lap |[\Phi,F]|^2 d\text{Area} + \int_{\partial B(x,t)} (\frac{n-1}{t}+O(t)) |[\Phi,F]|^2   d\text{Area} 
%			\\
%			& \geq m^2I(t) +   (\frac{n-1}{t}-Ct) I'(t),
%	\end{split}
%	\]
%	where $H$ is the mean curvature of the geodesic sphere. Then apply some comparison argument (for $t\gtrsim m^{-1}$).
%	\end{proof}

\subsubsection{Curvature concentration locus}\label{sect:curvatureconcentrationlocus}

We design a notion of curvature concentration locus, such that on its complement we can apply Prop. \ref{prop:expdecay} to infer \emph{exponential decay} of $F^\perp$, and at the same time its tubular neighbourhoods enjoy  \emph{measure estimates}.

\begin{Def}\label{Def:curvatureconcentrationlocus}
	Given a $G_2$-monopole (resp. Calabi-Yau monopole) as in the  Main Setting, and assume the mass $m\gg \Lambda\geq 1$ is  sufficiently large. 
	 For  any $p\in M$, we define its \emph{characteristic radius} as
	\begin{equation}
		r_\Lambda(p)= \min \{   s\in [m^{-1},+\infty)  :    \int_M       \frac{|F|^2}{  \max(d(x,p), s)^{n-2} }  dvol(x) \leq \Lambda^{-1}m^2       \}.
	\end{equation}
The \emph{curvature concentration locus} with parameter $\Lambda$ is defined as the subset
	\begin{equation*}
		\mathcal{C}_\Lambda:=  \{   p\in M: r_\Lambda(p)>m^{-1}    \}  =   \{ p\in M: 	\int_M  \frac{|F|^2}{  \max(d(x,p), m^{-1})^{n-2} }  dvol(x) > \frac{m^2}{\Lambda}     \}. 
	\end{equation*}
	Clearly, if $\Lambda_1>\Lambda_2$, then $\mathcal{C}_{\Lambda_1}\supset \mathcal{C}_{\Lambda_2}$.  We denote  its $s$-tubular neighbourhood as
$
	B_s(\mathcal{C}_\Lambda):= \{   x\in M:  \text{dist}(x,\mathcal{C}_\Lambda) \leq s  \}.
$

\end{Def}

	\begin{lem}\label{lem:curvatureconcentrationlocus1}
	There exists some $\Lambda_0\geq  1$, such that the following holds for all sufficiently large $m$:
if either
$
\int_{B(p, 2m^{-1}) } |F|^2dvol(x) \geq \epsilon_3 m^{4-n} ,
$
%where $\epsilon_3>0$ is the constant in Prop. \ref{prop:epsilonregularity},
or $\min_{B(p,m^{-1})  }|\Phi| \leq m/2$, then $p\in \mathcal{C}_{\Lambda_0}.$ 

	\end{lem}

	\begin{proof}
	If $
	\int_{B(p, 2m^{-1}) } |F|^2dvol(x) \geq \epsilon_3 m^{4-n} ,
	$
	then 
	\[
		\int_M  \frac{|F|^2}{  \max(d(x,p), m^{-1})^{n-2} }  dvol(x) \geq  (2m^{-1})^{2-n} \int_{B(p,2m^{-1})} |F|^2 \geq 2^{2-n}\epsilon_3 m^2,
	\]
	so $p\in \mathcal{C}_\Lambda$ as long as $\Lambda> 2^{n-2}\epsilon_3^{-1} $.

	Alternatively, $
	\int_{B(p, 2m^{-1}) } |F|^2dvol(x) \leq \epsilon_3 m^{4-n} ,
	$
	but $\min_{B(p,m^{-1})  }|\Phi| \leq m/2$.
The $\epsilon$-regularity Prop. \ref{prop:epsilonregularity} applied to the ball $B(p,2m^{-1})$ shows
	\[
	\sup_{B(p,m^{-1})} (|F|+|\nabla \Phi| ) \leq C\epsilon_3 m^2.
	\]
	In particular $|\Phi| \leq (\frac{1}{2}+ C\epsilon_3)m\leq \frac{2}{3} m  $ on $B(p,m^{-1})$, after possibly shrinking $\epsilon_3$ a little. Contrasting this with (\ref{eqn:smallPhi2}), we deduce
	\[
		\int_M  \frac{|F|^2}{  \max(d(x,p), m^{-1})^{n-2} }  dvol(x)  \geq C^{-1} \dashint_{B(p,m^{-1})} ||\Phi|^2-m^2| \geq C^{-1}m^2. 
	\]
	so $p\in \mathcal{C}_{\Lambda}$ as long as $\Lambda>C$. 
	\end{proof}

	\begin{prop} \label{prop:expdecay2}
		The following exponential decay estimates holds away from the curvature concentration locus $\mathcal{C}_{\Lambda_0}$:
		\begin{enumerate}
			\item 	If $B(p,s)\subset M\setminus \mathcal{C}_{\Lambda_0}$, and $s\leq r(p)/C'  $, then 
			\[
			|F^\perp|^2 (p) \leq C e^{-ms} \dashint_{B(p,s)} |F^\perp|^2,
			\]
			where $C, C'$ are independent of $m, s$. 
			
			\item If $r>1$ and $s< r/C'$, then 
			\[
			\int_{  \{   r(x)<r \}  \setminus  B_s(\mathcal{C}_{\Lambda_0})      } |F^\perp|^2 dvol(x) \leq  Ce^{-ms} (m+ r^{n-4}),
			\]
			where $C,C'$ are independent of $m, s$. 
		\end{enumerate}

	\end{prop}

	\begin{proof}
\textbf{Item 1} follows from
 Lemma \ref{lem:curvatureconcentrationlocus1} and Prop. \ref{prop:expdecay}, applied inside a geodesic ball of radius $r(p)/C'$. For \textbf{item 2}, we apply item 1 and Fubini theorem,
 \[
 \begin{split}
 & 	\int_{   B(p,r(p)/C')\setminus  B_s(\mathcal{C}_{\Lambda_0}) }|F^\perp|^2(x) dvol(x) 
 	\\
 	\leq & 	 e^{-ms} \int_{   B(p,r(p)/C')\cap  \{  dist(x, \mathcal{C}_{\Lambda_0})>s\}  } \dashint_{B(x,s)} |F^\perp|^2 (y)dvol (y) dvol(x)
 	\\
 	\leq & Ce^{-ms} \int_{   B(p, 2r(p)/C')  } |F^\perp|^2 dvol(y). 
 \end{split}
 \]
 We can cover $\{   r<r(x)<2r \}$ with a bounded number of geodesic balls where the metric is $C^3$ equivalent to the Euclidean metric, to deduce
 	\[
 \int_{  \{   r<r(x)<2r \}  \setminus  B_s(\mathcal{C}_{\Lambda_0})      } |F^\perp|^2 dvol(x) \leq  Ce^{-ms}\int_{  \{   r/2<r(x)<4r \} } |F^\perp|^2 ,
 \]
Summing over the dyadic annuli up to scale $r$,
\[
	\int_{  \{   r(x)<r \}  \setminus  B_s(\mathcal{C}_{\Lambda_0})      } |F^\perp|^2 dvol(x) \leq  Ce^{-ms}\int_{  \{   r(x)<2r \} } |F^\perp|^2 \leq Ce^{-ms} (m+ r^{n-4}),
\]
where the second inequality follows from the monotonicity formula Lemma \ref{lem:monotonicity2}.
	\end{proof}

	\begin{lem}\label{lem:characteristicradius}
Let $p\in \mathcal{C }_{\Lambda}$ with $1<r<r(p)<2r$. 
	The characteristic radius satisfies the following properties:
	\begin{enumerate}
		\item We have  $
		\int_{ B(p, r_{2\Lambda}(p) )}  \frac{|F|^2}{  \max(d(x,p), r_{\Lambda}(p) )^{n-2} }  dvol(x) \geq (2\Lambda)^{-1}m^2       .
		$

		\item Suppose $m\gg \Lambda\geq 1$, then 
	$
		r_{2\Lambda}^2(p)\leq C\Lambda m^{-1} (r^{4-n}+ m^{-1}) \ll 1.
$
		
		\item  (Characteristic radius for nearby points)   Suppose $d(p,x) \leq s \leq r/3$ for some $x\in \mathcal{C}_{\Lambda}$. Then $r_{  \tilde{\Lambda}    } (p) \geq r_{\Lambda}(x)$ for some 
	 $  \tilde{\Lambda} \leq C\Lambda \max(s^2 r_{\Lambda}(x)^{-2},1) $, and in particular $p\in \mathcal{C}_{  \tilde{\Lambda} }$.

		\item (Tubular neighbourhood containment)   If $m^{-1}\leq s<r/3$, then
		\[
		B_s(\mathcal{C}_{\Lambda}) \cap \{  r\leq r(x)\leq 2r  \}  \subset  \mathcal{C}_{C\Lambda s^2m^2}.
		\]

	\end{enumerate}
	\end{lem}

\begin{proof}
\textbf{Item 1}. For $p\in \mathcal{C}_\Lambda$, 
\[
\begin{split}
	& 	\int_{ B(p, r_{2\Lambda}(p) )}  \frac{|F|^2}{  \max(d(x,p),m^{-1} )^{n-2} }  dvol(x) 
		\\
		\geq  &  	\int_{ M}  \frac{|F|^2}{  \max(d(x,p), m^{-1}   )^{n-2} }  dvol(x) 
		- \int_{ M}  \frac{|F|^2}{  \max(d(x,p), r_{2\Lambda}(p) )  ^{n-2} }  dvol(x)
		\\
		\geq &  \Lambda^{-1} m^2-(2 \Lambda)^{-1} m^2 =(2 \Lambda)^{-1} m^2 .
\end{split}
\]
%Here the third lines uses that $\int_{ M}  \frac{|F|^2}{  \max(d(x,p), r_\Lambda(p)   )^{n-2} }  dvol(x) = \Lambda^{-1}m^2$ since $r_\Lambda(p)>m^{-1}$. 

\textbf{Item 2}. Since $m\gg \Lambda\geq 1$,  item 1 implies that
\[
(mr^{4-n}+1)r^{-2} \ll  (2\Lambda)^{-1}m^2 \leq 	\int_{ B(p, r_{2\Lambda}(p))}  \frac{|F|^2}{  \max(d(x,p), m^{-1})^{n-2} }  dvol(x) ,
\]
so  Cor. \ref{cor:monotonicity2} gives  $r_{2\Lambda}(x)\leq r$, and
\[
C(mr^{4-n} +1)r_{2\Lambda}(x)^{-2} \geq 	\int_{ B(p, r_{2\Lambda}(p))}  \frac{|F|^2}{  \max(d(x,p), m^{-1})^{n-2} }  dvol(x) \geq (2\Lambda)^{-1}m^2.
\]
This is equivalent to item 2.

\textbf{Item 3}. Suppose $r<r(p)<2r$, and $d(p,x) \leq s$ for some $x\in \mathcal{C}_{\Lambda}$. Since $s\leq r/3$, we know $r/2 \leq r(x) \leq 3r$. As $x\in \mathcal{C}_{\Lambda}$, we have $r_{\Lambda}(x)> m^{-1}$, 
\[
\int_M \frac{ |F|^2(y)}{\max(d(x,y), r_{\Lambda}(x) )^{n-2}} dvol(y)> \Lambda^{-1} m^2. 
\]
If $s\lesssim r_{\Lambda} (x)$, then $\max( d(p,y), r_\Lambda(x)) $ is uniformly equivalent to 
$\max( d(x,y), r_\Lambda(x)) $ for all $y\in M$, so 
\[
%\int_M \frac{ |F|^2(y)}{\max(d(p,y), m^{-1})^{n-2}} dvol(y)\geq
 \int_M \frac{ |F|^2(y)}{\max(d(p,y), r_{\Lambda}(x) )^{n-2}} dvol(y)> (C\Lambda)^{-1} m^2,
\]
whence the characteristic radius $r_{C\Lambda}(p)\geq r_{\Lambda}(x)$, and in particular $p\in \mathcal{C}_{C\Lambda}$. Otherwise we have $s> r_{\Lambda(x)}$, and then we can use the monotonicity formula Prop. \ref{prop:monotonicity1} to see
\[
\begin{split}
 & \int_{B(x,2s)} \frac{ |F|^2(y)}{\max(d(x,y), r_{\Lambda}(x) )^{n-2}} dvol(y) 
\\
\leq &   Cs^{4-n} r_{\Lambda}(x)^{-2}  \int_{B(x,2s)}  |F|^2(y) dvol(y) 
\\
\leq & Cs^2 r_{\Lambda}(x)^{-2}  \int_{B(x,2s)} \frac{ |F|^2(y)}{\max(d(p,y), r_{\Lambda}(x) )^{n-2}} dvol(y) 
\\
\leq & Cs^2 r_{\Lambda}(x)^{-2}  \int_M \frac{ |F|^2(y)}{\max(d(p,y), r_{\Lambda}(x) )^{n-2}} dvol(y) ,
\end{split}
\]
hence
\[
\begin{split}
\Lambda^{-1} m^2& <	\int_M \frac{ |F|^2(y)}{\max(d(x,y), r_{\Lambda}(x) )^{n-2}} dvol(y) 
\\
	\leq 
	& \int_{B(x,2s)} \frac{ |F|^2(y)}{\max(d(x,y), r_{\Lambda}(x) )^{n-2}} dvol(y) + 	\int_M \frac{ |F|^2(y)}{\max(d(x,y), 2s)^{n-2}} dvol(y) 
	\\
	\leq  &  Cs^2 r_{\Lambda}(x)^{-2}  \int_M \frac{ |F|^2(y)}{\max(d(p,y), r_{\Lambda}(x) )^{n-2}} dvol(y) .
\end{split}
\]
This shows $r_{  \tilde{\Lambda}    } (p) \geq r_{\Lambda}(x)$ for $  \tilde{\Lambda} \leq Cs^2 r_{\Lambda}(x)^{-2}\Lambda $, so $p\in \mathcal{C}_{  \tilde{\Lambda}  }$.

\textbf{Item 4} is a corollary of item 3: each point in the $s$-tubular neighbourhood is contained in $\mathcal{C}_{\tilde{\Lambda}} $ with $\tilde{\Lambda}\leq C\Lambda s^2m^{2}$.
\end{proof}

We then control the measure on the neighbourhoods of the  curvature centration locus
$
B_s(\mathcal{C}_\Lambda). 
$

\begin{prop}\label{prop:measurebound}
	 Let $r>1$ and  $m\gg \Lambda\geq 1$. The curvature concentration locus satisfies the following measure estimates with constants independent of $m,\Lambda$:

\begin{enumerate}

	\item We have $
		\text{Vol}(  \mathcal{C}_\Lambda\cap \{ r(x)<r \}    )
	\leq C\Lambda ^2m^{-2}(1+ m^{-1} r^{n-4} ).
	$

	\item Whenever $(\Lambda m)^{-1/2}\gg s\geq m^{-1}$, we have the Minkowski measure estimate	\[
	\text{Vol}( B_s( \mathcal{C}_{\Lambda}) \cap \{ r(x)<r \}    )
	\leq    C(ms^2 \Lambda)^2  (1+ m^{-1} r^{n-4}).
	\]

\end{enumerate}

\end{prop}

\begin{proof}
\textbf{Item 1}. For any $r>1$, and $m\gg \Lambda\geq 1$, we have
\[
\begin{split}
	& (2\Lambda)^{-1} m^2 \text{Vol}(  \mathcal{C}_\Lambda\cap \{ r<r(x)<2r \}    )
	\\
	& \leq 	\int_{\{x\in \mathcal{C}_\Lambda: r<r(x)<2r \}    } \int_{B(x, r_{2\Lambda}(x)) } \frac{ |F|^2(y)} { \max( d(x,y), m^{-1})^{n-2} } dvol(x)dvol(y) 
	\\
	&  \leq \int_{   r/2<r(y)< 4r }   \int_{ \{ x\in \mathcal{C}_\Lambda, d(x,y)\leq r_{2\Lambda}(x) \}   }\frac{ |F|^2(y)} { \max( d(x,y), m^{-1})^{n-2} } dvol(x)dvol(y) 
	\\
	&
	\leq    C\int_{   r/2<r(y)< 4r } |F|^2(y)   \sup_{x\in  \mathcal{C}_\Lambda, d(x,y)\leq r_{2\Lambda}(x)    } r_{2\Lambda}^2(x)   dvol(y)
	\\
	&  \leq     C\Lambda m^{-1} (r^{4-n}+m^{-1})    \int_{   r/2<r(y)< 4r } |F|^2(y)      dvol(y)
	\\
	& \leq  C\Lambda m^{-1}     \int_{   r/2<r(y)< 4r } |F|^2(y)      dvol(y)
\end{split}
\]
Here the second line uses item 1 in Lemma  \ref{lem:characteristicradius}, the third line applies Fubini theorem, the 4th line integrates over the $x$-variable, and the 5th line applies item 2 in Lemma \ref{lem:characteristicradius}. 
Summing over  all dyadic scales up to $r$, we see
\[
\begin{split}
 \text{Vol}(  \mathcal{C}_\Lambda\cap \{ r(x)<r \}    )
\leq  & C\Lambda^2 m^{-3}  \int_{ r(y)\leq 2r}  |F|^2(y)dvol(y) 
\\
\leq  & C\Lambda^2 m^{-2}(1+ m^{-1} r^{n-4} ),
\end{split}
\]
where the second inequality uses the monotonicity formula Lemma \ref{lem:monotonicity2}.

\textbf{Item 2}. 
Here $m^{-1}\leq s\ll (\Lambda m)^{-1/2} \ll r$. By item 4 in Lemma \ref{lem:characteristicradius},
\[
\text{Vol}(  B_s( \mathcal{C}_{\Lambda}) \cap \{ r(x)<r \}    ) \leq  \text{Vol}(  \mathcal{C}_{Cm^2s^2 \Lambda}\cap \{ r(x)<r \}    ).
\]
Since $Cm^2 s^2 \Lambda\ll m$, the RHS $\leq C(m^2s^2 \Lambda)^2 m^{-2} (1+ m^{-1} r^{n-4})$ by item 1.
\end{proof}

\begin{cor}\label{cor:L1Fperp}
Let $r>1$ and $m\gg \Lambda_0>1$ for $\Lambda_0$ is the constant fixed in Lemma  \ref{lem:curvatureconcentrationlocus1}.
\begin{enumerate}
	\item  The $L^1$-curvature in the curvature concentration locus is bounded by
	\[
	\int_{ \mathcal{C}_{\Lambda_0}\cap \{   r(x)<r  \} }   |F|\leq C m^{-1/2} (1+ m^{-1} r^{n-4} ).
	\]
	\item  We have the $L^1$-estimate on $F^\perp$,
	\[
		\int_{ \{   r(x)<r  \} }   |F^\perp|\leq C m^{-1/2}(1+ m^{-1} r^{n-4} ).
	\]

\end{enumerate}
\end{cor}

\begin{proof}
\textbf{Item 1} follows from Cauchy-Schwarz, the $L^2$-energy bound in Lemma \ref{lem:monotonicity2}, and the measure bound  in Prop. \ref{prop:measurebound}.

\textbf{Item 2}. We use item 1 to control the $L^1$ integral within $\mathcal{C}_{\Lambda_0}$. On the other hand, one can show
\[
\int_{ \{   r(x)<r  \}\setminus \mathcal{C}_{\Lambda_0}   }   |F^\perp|\leq C m^{-1/2}(1+ m^{-1} r^{n-4} )
\]
by decomposing 
$\{   r(x)<r  \}\setminus B_{m^{-1}}(\mathcal{C}_{\Lambda_0} )  $ into a union of thin shells
\[
\{   r(x)<r  \}\cap B_{s+m^{-1}}(  \mathcal{C}_{\Lambda_0}   )\setminus B_{s}(  \mathcal{C}_{\Lambda_0}   ), \quad s=0, m^{-1}, 2m^{-1},\ldots ,
\]
and control the $L^1$-norm on each shell 
using Cauchy-Schwarz,  the exponential decay of $L^2$ energy in item 2 of Prop. \ref{prop:expdecay2}, and the measure bound on the tubular neighbourhoods in item 2 of Prop. \ref{prop:measurebound}.
\end{proof}

	\begin{rmk}
		In our estimates, 
		the dependence on $m\gg \Lambda\geq  1$ is probably not sharp. On the unit ball in 
		the local model in Example \ref{eg:flatmodel}, for $\Lambda\sim 1$,
		\[
		\begin{cases}
				\text{Vol}(  \mathcal{C}_{\Lambda}\cap B(1)   )\sim m^{-3},
				\\
				\int_{B(1)\cap   \mathcal{C}_{\Lambda }} |F|  \sim m^{-1},
				\\
					\int_{B(1)\cap   \mathcal{C}_{\Lambda }} |F^\perp|  \sim m^{-1},
					\\
					\text{Vol}(  B_s( \mathcal{C}_{\Lambda}) \cap B(1)    )\sim s^3,\quad   m^{-1}\leq s\leq 1.
		\end{cases}
		\]
	%	If instead $1\leq \Lambda \ll m$, the characteristic radius suggested by Example \ref{eg:flatmodel} is $r_\Lambda(x)\sim \Lambda^{-1}$. 
		The place where our estimates fail to be optimal, is when we applied the non-sharp upper bound on $r_\Lambda(x)$ using Cor. \ref{cor:monotonicity2}. The root issue is that the energy monotonicity formula captures  $(n-4)$-dimensional energy growth, but  $\int_{B(p,s)} |F|^2$ really has $\sim ms^{n-3}$ growth in Example \ref{eg:flatmodel}. Can one prove an effective estimate for large $m$ that captures this $(n-3)$-dimensional energy growth (\cf Remark \ref{rmk:monotonicityminsurface} below)?

	\end{rmk}

	\subsection{Parallel curvature component}

	We now turn to the parallel component $F^\parallel$ of the curvature along the $\Phi$ direction. This is proportional to $\Tr(F\Phi)=\Tr(F^\parallel \Phi)$.

	\begin{lem}\label{lem:traceeqn}
		 In the $G_2$-monopole case, $\Tr(F\Phi)$ satisfies the following linear differential equations
	\begin{equation}
	\begin{cases}
		d\Tr(F\Phi)= \Tr(F\wedge \nabla \Phi), 
		\\
	\Tr(F\Phi)\wedge \psi= *\frac{1}{2} d\Tr(\Phi^2),
		\\
		d* \Tr(F\Phi)=- d 	\Tr(F\Phi) \wedge \phi= -  \Tr(F\wedge \nabla \Phi)\wedge \phi,
	\end{cases}
	\end{equation}
	while in the Calabi-Yau monopole case,
	\begin{equation}
		\begin{cases}
		d\Tr(F\Phi)= \Tr(F\wedge \nabla \Phi), 
		\\
		\Tr(F\Phi)\wedge \text{Re}\Omega= *\frac{1}{2} d\Tr(\Phi^2),
		\\
		d * \Tr(F\Phi)= - d\Tr(F\Phi) \wedge \omega=   -  \Tr(F\wedge \nabla \Phi)\wedge \omega  .
	\end{cases}
	\end{equation}
	\end{lem}

	\begin{proof}
	In both case, the Bianchi identity implies 
	\[
		d\Tr(F\Phi)= \Tr(d_A F\wedge \Phi)+\Tr(F\wedge \nabla \Phi)= \Tr(F\wedge \nabla \Phi).
		\]
	In the $G_2$-monopole case, 
	\[
		* (\Tr(F\Phi)\wedge \psi)= \Tr(\Phi\nabla \Phi)= \frac{1}{2} d\Tr(\Phi^2),
	\]
 so $d*(\Tr(F\Phi)\wedge \psi)= 0.$ Now for any 2-form $v$, in general the 6-form $d * v$ is a linear combination of $	dv\wedge \phi$ and $	\psi\wedge d*(v\wedge \psi)$. 
	The Calabi-Yau monopole case is very similar, with $\psi$ replaced by $\text{Re}\Omega$. 
	\end{proof}

		\subsubsection{$L^1$ type estimate on the parallel curvature component}

	\begin{cor}\label{cor:traceeqn}
For any $r>1$, we have the $L^1$-estimate
	\[
	\int_{ r(x)<r}  |d\Tr(\Phi F)|+ |d^* \Tr (\Phi F)| \leq   Cm+ Cm^{1/2} r^{(n-4)/2}.
	\]
	\end{cor}

	\begin{proof}
	We combine Lemma \ref{lem:traceeqn} with the monotonicity formula Lemma \ref{lem:monotonicity2}, and the bound $\int_M |\nabla \Phi|^2\leq Cm$ from the intermediate energy formula:
	\[
	\begin{split}
			\int_{  r(x)<r}  |d\Tr(\Phi F)|+ |d^* \Tr (\Phi F)| & \leq  C	\int_{ r(x)<r} |F| |\nabla \Phi| 
			\\
			&   \leq  C	(\int_{ r(x)<r} |F|^2  )^{1/2}  ( \int_M |\nabla \Phi|^2)^{1/2} 
			\\
			&  \leq  C	(m+ r^{n-4} )^{1/2}  m^{1/2} .
	\end{split}
	\]
	\end{proof}

	By \cite{Oliveira1}, for fixed $m$,  the connection $A$ converges to a $U(1)$ pseudo-HYM connection $A_\infty$, while $\Phi$ converges to a parallel adjoint section $\Phi_\infty$ in the asymptotic geometry at $r(x)\to +\infty$. By the Bolgomolny trick computation in the Section \ref{sect:Oliveira}, 
in particular $-\frac{ \Tr (F\Phi)}{4\pi m}$ converges to the harmonic representative of the monopole class $\beta\in H^2(\Sigma)$. We wish to improve this qualitative result to a quantitative estimate uniform in large $m$. This invovles a slightly tricky Lemma in Hodge theory on the AC manifold $M$, presented in the Appendix.

We fix a reference 2-form $\sigma_0$ on $M$ which is asymptotic to  the harmonic representative of the monopole class $\beta\in H^2(\Sigma)$ as $r(x)\to +\infty$, and which satisfies $d\sigma_0= d^*\sigma_0=0$ outside a fixed compact set. Thus $d\sigma_0$ is compactly supported, and in the  long exact sequence (\ref{eqn:LES}), it represents the image of the class $\beta$ in the connecting map $H^2(\Sigma)\to H^3_c(M)$, which is not necessarily zero, so in general $d\sigma_0$ and $d^*\sigma_0$ do not vanish globally.

	\begin{prop}\label{prop:parallelcurvatureL1}
We fix any $1-n<\nu< -1-n/2$, and suppose $m\geq 1$.
	There exists some $L^2$ harmonic  2-form $\sigma\in L^2\mathcal{H}^2(M)\simeq H^2_c(M)$ depending on $(A,\Phi)$, such that the following holds:
	\begin{enumerate}
		\item  $	\dashint_{r<r(x)<2r} | \frac{\Tr(\Phi F)}{4\pi m} +\sigma_0-\sigma|  \leq C \max(  r^{\nu},  m^{-1/2} r^{- 1-n/2}      ). $

		%	\item  We have a weak $L^1$-estimate
		%	\[
		%	r^{-n} \norm{  \nabla (  \frac{\Tr(\Phi F)}{4\pi m} +\sigma_0-\sigma    ) }_{L^{1,\infty}(  \{r<r(x)<2r\}  )} \leq C \max(  r^{\nu-1},  m^{-1/2} r^{- 2-n/2}      ).
		%	\]
		
		\item $	\norm{\sigma}_{L^2}^2 \leq Cm.$

		\item The class $[\sigma]\in H^2_c$ is integral.
		
		\item  We have a pointwise estimate at any $p$ with  $1<r<r(p)< 2r$,
		\[
		\begin{split}
		 |\frac{\Tr(\Phi F)}{4\pi m} +\sigma_0-\sigma|(p)  \leq & C \max(  r^{\nu},  m^{-1/2} r^{- 1-n/2}      )
		 \\
		 + & \int_{ r/2< r(x)<  4r  }    \frac{  m^{-1} |F| |\nabla \Phi| }{ \text{dist}(x,p)^{n-1}  }dvol(x)    .
		\end{split}
		\]
	
	\end{enumerate}

	\end{prop}

	\begin{proof}
		\textbf{Item 1}.
	By the asymptotic geometry, $\frac{\Tr(\Phi F)}{4\pi m}+  \sigma_0=o(r(x)^{-2})$ at infinity. By Cor. \ref{cor:traceeqn}, for any $r>1$, 
		\[
		\begin{split}
				\dashint_{ r(x)<r}  |d   (\frac{\Tr(\Phi F)}{4\pi m}+  \sigma_0) |+ |d^*( \frac{\Tr(\Phi F)}{4\pi m}+  \sigma_0) |  \leq & C(1+ m^{-1/2} r^{(n-4)/2})r^{-n}
				\\
				\leq & C \max(  r^{\nu-1},  m^{-1/2} r^{- 2-n/2}      ).
		\end{split}
	\]
	where we absorbed the error term $d\sigma_0$ and $d^*\sigma$ in a compact region into $C$. For $n=6,7$ we note that $-n< - 2-n/2<-3$. We can then apply Lemma \ref{lem:2form2} (with $\nu$ replaced by $\nu-1$) to obtain an $L^2$ harmonic 2-form $\sigma$ such that
	\[
		\dashint_{ r< r(x)<2r}  |\frac{\Tr(\Phi F)}{4\pi m}+  \sigma_0-\sigma | \leq C \max(  r^{\nu},  m^{-1/2} r^{- 1-n/2}      ). 
	\]

%	\textbf{Item 2}. We apply elliptic regularity for the operator $d+d^*$ to $\frac{\Tr(\Phi F)}{4\pi m}+  \sigma_0-\sigma$, on the annuli regions. 

	%	\textbf{Item 3}. Cor. \ref{cor:monotonicity} gives an $L^1$-type Morrey norm estimate
	%	\[
	%	\int_{B(p,s)}| (d+d^*)( \frac{\Tr(\Phi F)}{4\pi m} +\sigma_0-\sigma ) | \leq C(1+ m^{-1/2} r^{(n-4)/2}) 
	%	\]
 %so together with item 1, elliptic regularity gives a Morrey estimate 

	\textbf{Item 2}. On the compact subset $\{   r(x)\lesssim 10 \}$, 
		\[
		\int_{  \{   r(x)\lesssim 10 \} } |\frac{\Tr(\Phi F)}{4\pi m} |^2 \leq 	C\int_{  \{   r(x)\lesssim 10 \} } |F|^2 \leq Cm,
	\]
	so by Cauchy-Schwarz,
	\[
		\int_{  \{   r(x)\lesssim 10 \} } |\frac{\Tr(\Phi F)}{4\pi m} | \leq Cm^{1/2}. 
	\]
	By item 1,  we see $	\int_{  \{   r(x)\lesssim 10 \} } |\sigma | \leq 	Cm^{1/2}$.  
	But $\sigma$ is harmonic, so $ \norm{  \sigma}_{C^0( \{   r(x)\lesssim 10 \}   )}\leq Cm^{1/2}$. Now $L^2\mathcal{H}^2(M)$ is a finite dimensional space, so this bounds also the $L^2$-norm.

	\textbf{Item 3}. We can adjust $\sigma$ by a bounded element in $L^2\mathcal{H}^2(M)\simeq H^2_c(M)$, to make it lie on the integral lattice in $H^2_c(M)$. The other estimates are not affected.

	\textbf{Item 4}. We apply the Green representation formula estimate for the $d+d^*$ operator on $\{ r/2< r(x)<4r \}$: for any $p$ with $r<r(p)<2r$, 
	\[
	|v|(p) \leq C\dashint_{  r/2<r(x)< 4r} |v| + C\int_{ r/2<r(x)< 4r} \frac{ |dv|+ |d^* v|}{ \text{dist}(x,p)^{n-1}  }dvol(x),
	\]
	to the differential form $v= \frac{\Tr(\Phi F)}{4\pi m}+\sigma_0-\sigma$.
	\end{proof}

We make a few remarks:
	\begin{enumerate}
		\item    Geometrically,  the $L^2$-harmonic 2-form $\sigma\in L^2\mathcal{H}^2(M)\simeq H^2_c(M)$ appears because the first Chern class of a $U(1)$ connection on $M$ is not determined by its restriction to the cross section $\Sigma$ at infinity, but  only up to a class in the compactly supported cohomology $H^2_c(M)$. We did not rule out the possibility that the local $L^1$-norm of $\sigma$ on balls may be much bigger compared to $\frac{\Tr(\Phi F)}{4\pi m}-\sigma$, neither do we know any example for this phenomenon when $c_2(P)$ is fixed. The point is however that modulo this finite dimensional space $L^2\mathcal{H}^2(M)$, then $\frac{\Tr(\Phi F)}{4\pi m}$ has a \emph{uniform local $L^1$ bound} which is independent of large $m$.

		\item  The usual intuition in analysis is that on a fixed ball, $L^2$-bounds are stronger than $L^1$-bounds. Here however one should view $L^1$ and $L^2$ as capturing very different kinds of information, and $L^1$-bounds are much better behaved in its dependence on $m$.

	As an illustration,	in Example \ref{eg:flatmodel} when $m$ is very large, the $L^p$ curvature integral
		\[
		\int_{B(1)} |F|^p  dvol \sim \begin{cases}
			m^{2p-3},\quad  &p> 3/2,
			\\
			1, \quad &p< 3/2.
		\end{cases}
		\]
		In the $p<3/2$ case, the density $|F|^p dvol$ is dispersed over $B(1)$, while for $p>3/2$, the density is concentrated near the $O(m^{-1})$-neighbourhood of the calibrated plane $Q$.

		As another illustration, in Cor. \ref{cor:L1Fperp} we proved the (likely non-optimal) $L^1$-estimate on the perpendicular curvature $F^\perp$, and in particular $\int_{r(x)\lesssim 10} |F^\perp| \leq Cm^{-1/2}\ll 1$.  The upshot is that  in the $L^1$-norm, $F^\perp$ is much smaller than $F^\parallel$, which can be viewed as an \emph{abelianization effect} in the generic region for large mass $m$. On the other hand, both curvature components have comparable $L^2$-integral in Example \ref{eg:flatmodel},
		\[
		\int_{B(1)} |F^\perp|^2\sim  	\int_{B(1)} |F^\parallel |^2\sim m,
		\]
		so the $L^2$-norm captures the nonabelian effect of $G_2$-monopoles (resp. Calabi-Yau monopoles) in the curvature concentration locus.

%	\item 	improve to Morrey norm?
	\end{enumerate}

	\section{Chern form of the singular $U(1)$ connection}\label{sect:Chernform}

	Let $(A,\Phi)$ be a $G_2$-monopole (resp. Calabi-Yau monopole) on the AC manifold $M$ as in the Main Setting.
	 Over $M\setminus \Phi^{-1}(0)$, the unit adjoint section $\frac{\Phi}{|\Phi|}$ specifies a $U(1)$-subbundle of the $SU(2)$-bundle $P$, with an induced $U(1)$-connection. 
	Locally, we can take the associated rank 2 complex vector bundle $E$ for the $SU(2)$-bundle $P$ via the fundamental representation. The unit adjoint section $\frac{\Phi}{   |\Phi|}$ decomposes $E$ into two orthogonal complex line bundles, defined by the eigenspace condition $ \frac{\Phi}{  \sqrt{-1} |\Phi|} s= \pm  s $. Let $\pi_{\pm}=\pm \frac{1}{2\sqrt{-1}} \frac{\Phi}{|\Phi|}+\frac{1}{2}$ denote the orthogonal projections into the two line subbundles. The induced curvature on the $\pm 1$ eigen line subbundle is respectively
	$
	\pi_\pm (F + \frac{1}{2}[ \nabla \pi_\pm \wedge \nabla \pi_\pm]) \pi_\pm
	$.
	This data is equivalent to the \emph{Chern form} of the $U(1)$-bundle on $M\setminus \Phi^{-1}(0)$,
	\begin{equation}
		\begin{split}
				F_{U(1)}=& \pm \frac{1}{2\pi\sqrt{-1}  }\Tr (    \pi_\pm	(F +\frac{1}{2}[ \nabla \pi_\pm\wedge \nabla \pi_\pm])    )
				\\
				=& \frac{-1}{4\pi  } \Tr \left(    \frac{\Phi}{|\Phi|}	(F - \frac{1}{8} [ \nabla (  \frac{\Phi}{|\Phi|}  )\wedge \nabla  ( \frac{\Phi}{|\Phi|})   ])   \right).
		\end{split}
	\end{equation}
	This is a closed real 2-form representing the first Chern class of the $U(1)$-bundle.
	%We observe that $\Tr (  \frac{\Phi F}{|\Phi|}	 )= \Tr (  \frac{\Phi}{|\Phi|}	F^\parallel  )$
	%extracts the \emph{parallel curvature component} $F^\parallel$ up to a normalisation factor. 

	\begin{rmk}\label{rmk:SO(3)2}
	In the $G=SO(3)$ variant case, the same formula for $F_{U(1)}$ defines a closed 2-form on $M\setminus \Phi^{-1}(0)$, since locally the $SO(3)$-connection can be lifted to an $SU(2)$-connection, and the formula only requires the Lie algebra structure. The caveat is that instead $2F_{U(1)}$ is the representative of the first Chern class of a $U(1)$-bundle on $M\setminus \Phi^{-1}(0)$, while $[F_{U(1)}]\in H^2(M\setminus \Phi^{-1}(0),\R)$ may be only a half integral calss, \cf Remark \ref{rmk:SO(3)1}.
	\end{rmk}

	\begin{Notation}
		We consider a sequence $(A_i, \Phi_i)$ of $G_2$-monopoles (resp. Calabi-Yau monopoles) with fixed monopole class $\beta$ and second Chern class $c_2(P)$, but the mass $m_i\to +\infty$. For ease of notation, we will mostly suppress the index $i$, and we sometimes refer to the $i\to +\infty$ limit suggestively as the $m\to +\infty$ limit. We say a quantity satisfies uniform $L^1_{loc}$ estimates, if  on any fixed compact subset in $M$,  its $L^1$-norm is uniformly bounded as $i\to +\infty$. 
	\end{Notation}

We morally want to understand the  $m\to +\infty$ limit of $F_{U(1)}$ and the $(n-3)$-current $dF_{U(1)}$, which is supported on the zero locus of the Higgs field. In the grand scheme, we hope to take the weak limit of $dF_{U(1)}$ as $m\to +\infty$ to obtain the coassociative cycle (resp. special Lagrangian cycle). As a first step, the main task of this Section is to  extract some \emph{rectifiable current} in the $m\to +\infty$ limit.

The main difficulties are the following:
\begin{enumerate}
	\item  We do not have direct control over $\Phi^{-1}(0)$, such as $\mathcal{H}^{n-3}$-measure bounds, or estimates on the mass of the $(n-3)$-current $dF_{U(1)}$ inside balls, so in particular we cannot apply Federer-Fleming compactness theory to $F_{U(1)}$ directly.

	\item 
	We do not know if there is some open dense subset of $M$ such that the Chern form converges in $C^\infty_{loc}$ topology. (See Section \ref{sect:bubbling} for a conjectural mechanism that $C^\infty_{loc}$ convergence fails on every open subset, in the local analogue of the problem; this is due to the curvature concentration locus becoming arbitrarily dense in the $m\to +\infty$ limit).
	
	\item 
As $m\to +\infty$, assuming one can take a subsequential weak limit of $dF_{U(1)}$, it is a priori possible that  its support has Hausdorff dimension greater than $n-3$, and in the extreme situation may conceivably `spread over the whole manifold'.  To prove that the weak limit is a \emph{rectifiable} $(n-3)$-current, one has to control the size of the support.

	\item Suppose one can prove the weak limit $Q$ is a rectifiable $(n-3)$-current with \emph{real coefficient}, meaning that it admits the representation
	\[
Q(\alpha)=	\int_S \Theta(x) \langle \alpha(x), \xi(x) \rangle d\mathcal{H}^{ n-3 }(x),\quad \forall \alpha\in \Omega^3_c(M),
	\]
	where $S$ is a countably rectifiable set, $\xi(x)$ is the unit $(n-3)$-vector field given by $\Lambda^k T_x S$ for $\mathcal{H}^{n-3}$-a.e. $x$, and $\Theta(x)$ is a \emph{real valued} locally integrable measurable function on $S$. Then one still needs to show its multiplicity function $\Theta(x)$ is \emph{integer valued}.
	
\end{enumerate}

To address these subtleties, we construct some delicate approximations of $dF_{U(1)}$. We first construct a mollified Chern form $\tilde{F}$, which is supported near the curvature concentration locus, and $d\tilde{F}$ admits $L^1_{loc}$-bounds. We then use a Whitney type decomposition and a deformation lemma construction to perturb $d\tilde{F}$ to a codimension three polyhedral chain with real coefficients. The key is to control the Hausdorff measure of the support of this polyhedral approximation, so we can apply a result of Ambrosio-Kirchheim \cite{Ambrosio} to extract a rectifiable current $Q$ in the $m\to +\infty$ limit.  Due to the mollification procedure, a priori the multiplicity function of $Q$ is only real valued. We will exploit the integrality of the Chern class to derive some effective estimates, which will lead to integer multiplicity by a further blow up limit argument in Thm. \ref{thm:weaklimitstructure}.

	\begin{rmk}
Parise-Pigati-Stern \cite[Section 3]{SternPigati} have recently found a different way to extract an $(n-3)$-dimensional \emph{integral cycle} from the limit of the 3-forms $-\frac{1}{4\pi m}\Tr(F\nabla \Phi)$, viewed as $(n-3)$-currents. 
	Their method is to reduce the problem to ambient dimension three, by considering the $0$-current slices of the $(n-3)$-currents along any projection to an $(n-3)$-plane. In dimension three, they used Uhlenbeck gauge fixing and local averaging technique to construct a H\"older replacement for the Higgs field, such that $\Tr(F\nabla \Phi)$ changes by only a small amount. Using a Bolgomolny trick style argument, one can then approximate the $0$-currents modulo small error,  by a finite sum of delta masses with \emph{integral coefficients} that encode local topological degree information.

	Their method has the advantage of proving rectifiability and integrality simultaneously, and avoiding the direct use of PDE. In comparison, our method is technically a little more complicated, but it also gives some extra information, including a quantitative polyhedral approximation of the $(n-3)$ currents, and proves not just the integer multiplicity of the limit $(n-3$)-current $Q$, but also the integrality of some cohomology class.

	\end{rmk}

	\begin{rmk}
	We apologize that the word `mass' can refer to both the parameter $m$, or the mass of a current. Both terminologies are well established, and we hope this causes no confusion.
	\end{rmk}

	\subsection{Mollification of the Chern form}

		We will consider an auxiliary \emph{mollification} of $F_{U(1)}$.  We take a smooth function $\tilde{\eta}_1: [0,1]\to [\frac{1}{3},1]$, such that 
			\[
		\tilde{\eta}_1(x) =\begin{cases}
			x, \quad & x\geq 1/2,
			\\
			1, \quad  &x\leq 1/3.
		\end{cases}
		\]
 Let
	 $\eta_1(x)= m\tilde{\eta}_1(  x /m  ).$
		Thus $\eta_1(|\Phi|)$  agrees with $|\Phi|$ as long as  $|\Phi|\geq m/2$ (which holds in particular on the complement of the curvature concentration locus $\mathcal{C}_{\Lambda_0}$ by Lemma \ref{lem:curvatureconcentrationlocus1}), but $\eta_1$  has a global lower bound $m/3$ since $\tilde{\eta}_1\in [\frac{1}{3},1]$. Let
		$\tilde{\eta}_2:\R_{\geq 0}\to [0,1]$ be a 
		 smooth cutoff function  such that 
		\[
		\tilde{\eta}_2(x) =\begin{cases}
			0, \quad & x=0,
			\\
			1, \quad  &x\geq 1.
		\end{cases}
		\]
		We take (a smoothing of) the Lipschitz function $\eta_2= \tilde{\eta}_2(m \text{dist}(x, \mathcal{C}_{\Lambda_0}) )$, which is zero on the curvature concentration locus $\mathcal{C}_{\Lambda_0}$, equal to one outside its $m^{-1}$-tubular neighbourhood, while the gradient satisfies $|d \eta_2|\leq Cm^{-1}$ on its support.

	We define the \emph{mollified Chern form} as
		\[
		\tilde{F}=- \frac{1}{4\pi} \Tr (    \frac{\Phi F}{\eta_1(|\Phi|)}	) +    \frac{1}{32\pi} \eta_2 \Tr \left(  \frac{\Phi}{|\Phi|}  [ \nabla (  \frac{\Phi}{|\Phi|}  )\wedge \nabla  ( \frac{\Phi}{|\Phi|})   ]    \right).
		\]
The main advantage of our choice of mollification is that the support of $d\tilde{F}$ is contained in a small neighbourhood of the curvature concentration locus. %(See Section \ref{sect:curvatureconcentrationlocus}). 

		\begin{lem}\label{lem:mollifiedcurvaturedF}
		The mollified Chern form satisfies the following:
		\begin{enumerate}
			\item The 3-form $d\tilde{F}$ is supported in $B_{m^{-1}}(\mathcal{C}_{\Lambda_0})$.

			\item On $\mathcal{C}_{\Lambda_0}$, we have $|d\tilde{F}|\leq Cm^{-1} |\nabla \Phi||F|$.
			
			\item On $B_{m^{-1}}(  \mathcal{C}_{\Lambda_0})\setminus  \mathcal{C}_{\Lambda_0}$, we again have $|d\tilde{F}|\leq Cm^{-1} |\nabla \Phi||F|$.

		\end{enumerate}
		\end{lem}

		\begin{proof}
		\textbf{Item 1}. On the complement of $B_{m^{-1}}(\mathcal{C}_{\Lambda_0})$,  $\tilde{F}=F_{U(1)}$ holds by the design of the mollification, and $F_{U(1)}$ is closed because it is the Chern  form for a $U(1)$-connection.

		\textbf{Item 2}. In this region $\eta_2=0$, so $\tilde{F}=-\frac{1}{4\pi} \Tr(     \frac{\Phi F}{\eta_1(|\Phi|)} )$, and by the Bianchi identity $d_A F=0$, we obtain
		\[
		d\tilde{F} =  -\frac{1}{4\pi} \Tr(  F\wedge \nabla (   \frac{\Phi }{\eta_1(|\Phi|)  } ) ), 
		\]
		whence $|d\tilde{F}|\leq Cm^{-1} |F||\nabla \Phi|$ using that $\eta_1$ is uniformly equivalent to $m$.

		\textbf{Item 3}. In this region $\eta_1(|\Phi|)= |\Phi|$, so 
		\[
		\tilde{F}= F_{U(1)}+  \frac{1}{32\pi}(\eta_2-1)  \Tr (  \frac{\Phi}{|\Phi|}  [ \nabla (  \frac{\Phi}{|\Phi|}  )\wedge \nabla  ( \frac{\Phi}{|\Phi|})   ]    ),
		\]
		hence by the closedness of $F_{U(1)}$,
		\[
		\begin{split}
	32\pi d\tilde{F}= &  d\eta_2 \wedge \Tr (  \frac{\Phi}{|\Phi|}  [ \nabla (  \frac{\Phi}{|\Phi|}  )\wedge \nabla  ( \frac{\Phi}{|\Phi|})   ]    )+  (\eta_2-1) d\Tr (  \frac{\Phi}{|\Phi|}  [ \nabla (  \frac{\Phi}{|\Phi|}  )\wedge \nabla  ( \frac{\Phi}{|\Phi|})   ])    
		\\
		=&  d\eta_2 \wedge \Tr (  \frac{\Phi}{|\Phi|}  [ \nabla (  \frac{\Phi}{|\Phi|}  )\wedge \nabla  ( \frac{\Phi}{|\Phi|})   ]   )+  8(\eta_2-1) d  \Tr (    \frac{\Phi F}{|\Phi|}	)   
		\\
		= &  d\eta_2 \wedge \Tr (  \frac{\Phi}{|\Phi|}  [ \nabla (  \frac{\Phi}{|\Phi|}  )\wedge \nabla  ( \frac{\Phi}{|\Phi|})   ]    )+  8(\eta_2-1) \Tr ( F\wedge \nabla(   \frac{\Phi }{|\Phi|})	)   
		\end{split}
		\]
		Since $|d\eta_2|\leq Cm$, and $|\Phi|\geq m/2$, we obtain 
		\[
		|d\tilde{F}|\leq Cm^{-1}|\nabla \Phi|^2 + Cm^{-1} |F||\nabla \Phi| \leq Cm^{-1}|F||\nabla \Phi|.
		\]
		using $|\nabla \Phi|\leq C|F|$.
		\end{proof}

Combining Lemma \ref{lem:mollifiedcurvaturedF} with the monotonicity formula Lemma \ref{lem:monotonicity2} and $\int_M |\nabla \Phi|^2\leq Cm$, we deduce a local $L^1$-bound on $d\tilde{F}$, for any $r>1$:
\begin{equation}\label{eqn:massdtildeF}
	\begin{split}
	\int_{ r(x)<r} |d\tilde{F}|\leq Cm^{-1} \int_{ r(x)<r} |\nabla \Phi| |F|
	 \leq C(1+   m^{-1} r^{n-4})^{1/2} .
	\end{split}
\end{equation}

	\subsubsection{Comparison between several mollifications}

By the following technical Lemma, our choice of the mollified Chern form $\tilde{F}$ is  $L^1_{loc}$-close to two other  approximations $ -\frac{1}{4\pi} \Tr (    \frac{\Phi F}{|\Phi|})$ and $-\frac{1}{4\pi m} \Tr (  \Phi F)  $ of the Chern form, for sufficiently large $m$.

\begin{lem}\label{lem:mollifiedcurvaturecomparison}
Let $m\gg 1$ be sufficiently large. Then for any  $r>1$, 
\[
\begin{cases}
\int_{ r(x)<r}  |\tilde{F}+   \frac{1}{4\pi} \Tr (    \frac{\Phi F}{|\Phi| }	) | \leq Cm^{-1/2} (1+ m^{-1} r^{n-4}),
\\
\int_{ r(x)<r}  |\tilde{F}+   \frac{1}{4\pi} \Tr (    \frac{\Phi F}{m}	) | \leq  C (m+ r^{n-4} )^{1/2} (m^{-1}r)^{   \frac{n}{2(n-1)} } .  
\end{cases}
\]
\end{lem}

\begin{proof}
\textbf{Item 1}.
We start by observing
\[
\begin{split}
\int_{ r(x)<r}  |\tilde{F}+   \frac{1}{4\pi} \Tr (    \frac{\Phi F}{\eta_1(|\Phi|)}	) |
= &  \int_{ r(x)<r}    |\frac{1}{32\pi} \eta_2 \Tr (  \frac{\Phi}{|\Phi|}  [ \nabla (  \frac{\Phi}{|\Phi|}  )\wedge \nabla  ( \frac{\Phi}{|\Phi|})   ]   )|
\\
\leq & C\int_{\{ r(x)<r \} \setminus \mathcal{C}_{\Lambda_0} }   |\nabla (  \frac{\Phi}{|\Phi|}  )\wedge \nabla  ( \frac{\Phi}{|\Phi|})  |
\\
\leq &  C m^{-2} \int_{\{ r(x)<r \} \setminus \mathcal{C}_{\Lambda_0} }   |\nabla \Phi  |^2 \leq Cm^{-1},
\end{split}
\]
where the last line uses that $|\Phi|\geq m/2$ on the complement of $\mathcal{C}_{\Lambda_0}$ by Lemma \ref{lem:curvatureconcentrationlocus1}.

Now $ -\frac{1}{4\pi} \Tr (    \frac{\Phi F}{\eta_1(|\Phi| )}) $ agrees with $  \frac{1}{4\pi} \Tr (    \frac{\Phi F}{|\Phi| })$ except on $\mathcal{C}_{\Lambda_0}$, and their difference is bounded using item 1 in Cor. \ref{cor:L1Fperp},
	\[
	\int_{\{ r(x)<r \}\cap \mathcal{C}_{\Lambda_0} }  |  \Tr (    \frac{\Phi F}{|\Phi| })-    \Tr (    \frac{\Phi F}{\eta_1(|\Phi|)}	) | \leq C	\int_{\{ r(x)<r \}\cap \mathcal{C}_{\Lambda_0} }  |  F| \leq Cm^{-1/2} (1+ m^{-1} r^{n-4}).
	\]
This shows item 1.

\textbf{Item 2}.
Next, we estimate
\[
\begin{split}
&\int_{ r(x)<r}    |  \Tr (    \frac{\Phi F}{\eta_1(|\Phi|)  })- \Tr (    \frac{\Phi F}{m }) |  
\\
\leq & C(\int_{ r(x)<r } |F|^2 )^{1/2}  ( \int_{  r(x)<r    }  |  m^{-1} \eta_1(|\Phi|)-1|^2        )^{1/2}
\\
\leq & C (m+ r^{n-4} )^{1/2}   (\int_{  r(x)<r    }  |  m^{-1} \eta_1( |\Phi|)  -1|^{ \frac{n}{n-1}   }  )^{1/2}
\end{split}
\]
Here second line uses  Cauchy-Schwarz and the fact $\eta_1(|\Phi|)$ is uniform equivalent to $m$, while the third line applies the monotonicity formula Lemma \ref{lem:monotonicity2}, and the trivial bound $|1- m^{-1}\eta_1(|\Phi|)|\leq 1$.

By the Poincar\'e inequality and Sobolev inequality on $M$,
\[
\begin{split}
	&  (\int_{  r(x)<r    }  |  m^{-1} \eta_1( |\Phi|)  -1|^{ \frac{n}{n-1}   }  )^{\frac{n-1}{n}}
	\\
	 \leq & C r^{-1}\int_{r(x)<C'r} | m^{-1}\eta_1(|\Phi|)-1|+ C  \int_{  r(x)<C'r    } m^{-1} | \nabla   \eta_1(|  \Phi|) | 
	 \\
	 \leq & C r^{-1} \int_{r(x)<C'r} | m^{-1}|  \Phi|-1|   + C \int_{  r(x)<C'r    } m^{-2} | \nabla   |  \Phi|^2 |  
	 \\
	 \leq & Cm^{-1} r.
\end{split}
\]
Here the third line uses the definition of $\eta_1$, and in particular that $\nabla \eta_1=0$ for $|\Phi|< m/3$, so that $|\Phi|$ is uniformly equivalent to $m$  on the support of $\nabla \eta_1(|\Phi|)$. The fourth line follows from Remark \ref{rmk:smallPhi}. Consequently,
\[
\int_{ r(x)<r}    |  \Tr (    \frac{\Phi F}{\eta_1(|\Phi|)  })- \Tr (    \frac{\Phi F}{m }) |  \leq C (m+ r^{n-4} )^{1/2} (m^{-1}r)^{   \frac{n}{2(n-1)} } .
\]
which implies item 2 after absorbing the $Cm^{-1}$ error term. 

We comment that item 2 is likely very non-optimal, but the key point is that for fixed $r$, the RHS is bounded by a negative power of $m$.
\end{proof}

\subsection{Whitney decomposition and deformation lemma}

\subsubsection{Stopping time algorithm}

Let $r>1$, and let $\mathcal{Q}\simeq [-s_0/2, s_0/2 ]^n$ be a coordinate cube inside $\{   r< r(x)< 2r      \}$, whose length size $s_0$ is comparable to $r$. We shall perform a \emph{Whitney type decomposition} of $\mathcal{Q}$ into generations of dyadic subcubes of size $s=s_0/2^j$, $j=0,1, 2, \ldots$, by the following stopping time algorithm:
\begin{itemize}
	\item  In the zeroth generation $j=0$, we start with $\mathcal{Q}$. In each generation $j=0,1,\ldots$, we have a collection of subcubes $ \{  \mathcal{Q}_i \}$ which form a partition of $\mathcal{Q}$. The size of these subcubes are $s_0/2^k$ for $k\in \{ 0,1,\ldots j \}$. 
	
	\item 
	To obtain the $(j+1)$-th generation from the $j$-th generation, for each subcube $\mathcal{Q}_i$ of size $s=s_0/2^j$, we either stop subdividing it for all future generations, or we replace $\mathcal{Q}_i$ by the collection of all its half cubes (there are $2^n$ half cubes, each of size $s/2$). The subcubes of size $s_0/2^k$ for $k<j$ have already been stopped in some previous generation.

	\item  The \emph{stopping criterion} is that on the given subcube $\mathcal{Q}_i$ of size $s$, one of the the following alternatives hold:
	\begin{enumerate}
		\item 	Either $\mathcal{Q}_i$ is disjoint from $B_{m^{-1} } 	(  \mathcal{C}_{\Lambda_0} )$ (See Def. \ref{Def:curvatureconcentrationlocus});
	
		\item  Or
		\[
		\max_{x\in \mathcal{Q}_i \cap B_{m^{-1} }( \mathcal{C}_{\Lambda_0} ) } \int_{B(x,s/2) } |F|^2 dvol   \geq  (a\Lambda_0)^{-1}  s^{n-3+\kappa}  m^{1+\kappa} .  
		\]
		Here $\Lambda_0$ is the constant fixed in Lemma \ref{lem:curvatureconcentrationlocus1}, while $a>1$ and $0<\kappa<1$ are parameter to be determined later. 
		
	\end{enumerate}
	Otherwise we keep subdividing.

\item  If the algorithm stops, we collect together all the subcubes $\{ \mathcal{Q}_i  \}$ from the last generation which intersect $B_{m^{-1} } 	(  \mathcal{C}_{\Lambda_0} )\cap \mathcal{Q}$, and we discard all the other subcubes. We denote this final collection as $\{  Q_i \}$, and this provides a cover of $B_{m^{-1}} (  \mathcal{C}_{\Lambda_0} ) \cap \mathcal{Q}$. We denote the size of $Q_i$ as $s_i$.
\end{itemize}

\begin{lem}(Whitney type decomposition)\label{lem:Whitney}
Suppose $m\gg 1$ is sufficiently large.
For any fixed $0<\kappa<1$, we can pick the constant $a>1$ independent of $r, m$, such that the algorithm stops in at most $j\leq \log_2(  s_0m  )$ generations, and 
\[
\begin{cases}
	\max_{Q_i} s_i \leq    Cm^{ -\kappa /(\kappa+1) }. 
\\
	\sum_{Q_i} s_i^{n-3} \leq Cm^{-1} \int_{ r/2< r(x)< 4r} |F|^2 dvol. 
\end{cases}
\]
\end{lem}

\begin{proof}
	\textbf{Item 1}. 
Without loss $\mathcal{Q}\cap B_{m^{-1} }( \mathcal{C}_{\Lambda_0} )$ is nonempty. Suppose the algorithm continues for $j> \log_2(s_0m)$ generations, then there is some cube $\mathcal{Q}_i$ containing some $p\in  \mathcal{Q}\cap B_{m^{-1} }( \mathcal{C}_{\Lambda_0} )$, such that the second alternative of the stopping criterion fails for all previous generations, hence
\[
\int_{B(p,s/2) } |F|^2 dvol   <  (a\Lambda_0)^{-1}  s^{n-3+\kappa}  m^{1+\kappa} , \quad s= s_0/2^k, k=0,1,\ldots j-1.
\]
Hence by summing over all dyadic scales, we deduce that for $s=s_0/2^j< m^{-1}$, 
\[
\begin{split}
& \int_M   \frac{ |F|^2}{   \max(d(x,p), s)^{n-2}    } dvol(x) 
\\
\leq &  C(a\Lambda_0)^{-1} \sum_{0\leq k\leq j} ( {s_0/2^k})^{n-3+\kappa-(n-2)}   m^{1+\kappa} + C(mr^{4-n}+1) r^{-2}
\\
\leq &  C(a\Lambda_0)^{-1} s^{\kappa-1} m^{1+\kappa}+ C(mr^{4-n}+1) r^{-2}
\\
\leq &  C(a\Lambda_0)^{-1}  m^2.
\end{split}
\]
Here the second line used the monotonicity formula Lemma \ref{lem:monotonicity2} to estimate the contributions outside $B(p,s_0/2)$ by $C(mr^{4-n}+1) r^{-2}$. The third line summed over a geometric series whose exponent $\kappa-1<0$. The 4th line uses $s=s_0/2^j<m^{-1}$, and absorbs the $O(m)$ term into the $m^2$ term for large enough $m$, as $a\Lambda_0$ is a constant independent of $m$.

On the other hand, by item 4 in Lemma \ref{lem:characteristicradius}, for any  $p\in \mathcal{Q}\cap B_{m^{-1} }( \mathcal{C}_{\Lambda_0} )$, we have $p\in \mathcal{C}_{C\Lambda_0}$, namely
\[
\int_M  \frac{ |F|^2}{   \max(d(x,p), m^{-1})^{n-2}    } dvol(x) >(C\Lambda_0)^{-1} m^2.
\]
We can choose $a$ large enough but independent of $m$, such that the two bounds are incompatible, which shows that the stopping time $j\leq  \log_2(s_0m)$. We fix $a$ henceforth.

\textbf{Item 2}.
For any $p\in \mathcal{Q}\cap B_{m^{-1} }( \mathcal{C}_{\Lambda_0} )$, the monotonicity formula Cor. \ref{cor:monotonicity} gives
\[
 \int_{B(x,s/2) } |F|^2 dvol  \leq  C(mr^{4-n}  + 1) s^{n-4} \leq Cm s^{n-4} ,
\]
so as long as $Cm s^{n-4}< (a\Lambda_0)^{-1}  s^{n-3+\kappa}  m^{1+\kappa} $ on the subcube $\mathcal{Q}_i$ containing $p$, namely $s> C(a,\Lambda_0) m^{ -\kappa /(\kappa+1) }$, then the cube will be further subdivided. Thus at the stopping time
$
s_i \leq Cm^{ -\kappa /(\kappa+1) }.
$

\textbf{Item 3}.  For each $j=0,1,\ldots $, 
We collect together the subcubes $Q_i$ with size $s=s_0/2^j$. The second alternative in the stopping criterion implies that
\[
s^{n-3+\kappa} m^{1+\kappa}
\leq C	\max_{x\in Q_i \cap B_{m^{-1} }( \mathcal{C}_{\Lambda_0} ) } \int_{B(x,s/2) } |F|^2 dvol 
\leq C\int_{ B_{s/2 }(Q_i  )  } |F|^2.
\]
Now for each fixed $j$, the neighbourhoods of cubes $B_{s/2 }(Q_i  )$ only overlap with each other in a  bounded way so 
\[
m\sum_{Q_i: s_i=s} s_i^{n-3} \leq C(sm)^{-\kappa} \sum_{ Q_i: s_i=s   } \int_{ B_{s/2 }(Q_i  )  } |F|^2 \leq C(sm)^{-\kappa}  \int_{ r/2<r(x)<4r  } |F|^2 .
\]
We can sum over all $j=0,1,\ldots $ to get
\[
m\sum_{Q_i} s_i^{n-3}  \leq C\sum_j  (ms_0/2^j)^{-\kappa}  \int_{ r/2<r(x)<4r  } |F|^2 \leq C  \int_{ r/2<r(x)<4r  } |F|^2.
\]
Here we used $\kappa>0$, and $j\leq \log_2(s_0 m)$ by item 1.
\end{proof}

The following Vitali cover argument uses very similar ideas and gives a little more flexibility in later applications, notably the proof of Thm. \ref{thm:energyidentity}.

\begin{lem}\label{lem:Vitali}
Let $0<\kappa< 1$ be fixed, and suppose $m^\kappa \gg \Lambda\geq \Lambda_0$. 
Given a bounded open subset $U\subset M$, we can find geodesic balls $B(p_i, s_i)$ with the following properties:
\begin{enumerate}
	\item   The union of $B(p_i, 5s_i)$ covers $\mathcal{C}_\Lambda\cap U$, and $s_i\geq m^{-1}$.

	\item   The integral $\int_{ B(p_i, s_i) }|F|^2 \geq C^{-1}\Lambda^{-1} s_i^{n-3+ \kappa} m^{1+ \kappa}$, and $s_i$ is the largest number such that this holds.

	\item We have $\max_i s_i\leq C  (\Lambda m^{- \kappa})^{\frac{1}{\kappa+1}} $, and for each value of $k$, 
	\[
	\Lambda^{-1}	m^{1+\kappa}  \sum  s_i^{ n-3+\kappa  }\leq  C\int_{U' } |F|^2, \quad U'= B_{   C  (\Lambda m^{- \kappa})^{\frac{1}{\kappa+1}}  } (U).
	\]

	\item For $s\geq s_i$, the Higgs field satisfies the average bound
	\[
	\dashint_{B(p,s)}  | 1-  \frac{|\Phi|^2}{m^2} | \leq C\Lambda^{-1} (ms)^{\kappa-1 }.
	\]

	%\item Failure criterion ????????
\end{enumerate}
\end{lem}

\begin{proof}
For each $p \in  \mathcal{C}_{\Lambda} \cap U$, we have 
\[
\int_M \frac{ |F|^2}{   \max(d(x,p), m^{-1})^{n-2}    }dvol(x) \geq \Lambda^{-1}m^2,
\]
By the same argument as in item 1 of Lemma \ref{lem:Whitney}, by choosing $C$ large enough, we can find some geodesic ball $B(p, s)$ with a dyadic radius $s\geq m^{-1}$ such that 
\begin{equation}\label{eqn:Vitali1}
\int_{ B(p, s) }|F|^2 \geq C^{-1}\Lambda^{-1} s^{n-3+ \kappa} m^{1+ \kappa}.
\end{equation}
Then
the monotonicity formula Cor. \ref{cor:monotonicity} implies  the upper bound
\[
C^{-1} \Lambda^{-1}s^{n-3+\kappa} m^{1+\kappa} \leq \int_{B(p,s)} |F|^2\leq C(mr(p)^{4-n}+1) s^{n-4} \leq C m s^{n-4}
\]
whence $
s\leq  C  (\Lambda m^{- \kappa})^{\frac{1}{\kappa+1}}\ll 1. 
$
We let $s(p)$ be the largest $s$ such that (\ref{eqn:Vitali1}) holds.

By the Vitali covering argument, we can find a subcover $\{  B(p_i, 5s_i) \}$ for $\mathcal{C}_{\Lambda} \cap U$ (so \textbf{items 1,2}   are automatic), and we can demand $B(p_i, s_i)$ are  all disjoint, hence
\[
\Lambda^{-1}	m^{1+\kappa}  \sum  s_i^{n-3+ \kappa}  \leq C   \sum_i \int_{B(p_i,s_i)} |F|^2 \leq  C\int_{U' } |F|^2.
\]
where we notice that all the balls $B(p_i,s_i)$ are contained in $U'= B_{   C  (\Lambda m^{- \kappa})^{\frac{1}{\kappa+1}}  } (U)$.
This proves \textbf{item 3}.

By the maximality of $s_i$ in item 2 above, for any $s\geq s_i$, 
\[
\int_{ B(p_i, s) }|F|^2 \leq C^{-1}\Lambda^{-1} s^{n-3+ \kappa} m^{1+ \kappa},
\]
so by summing up the contributions from the dyadic annuli, 
\[
\int_M \frac{ |F |^2}{ \max(d(x,p_i),s )^{n-2}  } dvol(x) \leq C \Lambda^{-1} m^2 (ms)^{\kappa-1}. 
\]
\textbf{Item 4} now follows from  item 3 in Lemma \ref{lem:smallPhi} and $|\nabla \Phi|\leq C|F|$.
\end{proof}

\subsubsection{Deformation lemma}

We can now locally replace the mollified Chern form  by a rectifiable current with real coefficients.

\begin{lem}(Deformation lemma)\label{lem:deformation}
	Let $0<\kappa<1$ be fixed. 
There is some rectifiable $(n-3)$-current $P$ supported on the $(n-3)$-dimensional skeleton of the unions of $\{  Q_i \}$,  and an $(n-2)$-current $R$ supported in $\cup_i Q_i$, such that on the interior of  $\mathcal{Q}$, the 3-current $d \tilde{F}$ satisfies
\[
\begin{cases}
	d\tilde{F}  - P= \partial R,
	\\
\text{Mass}(  P ) \leq C\text{Mass}( d\tilde{F}  \lfloor_{  \text{Int}(\mathcal{Q})  } )   , \quad   \text{Mass}(  R ) \leq C\text{Mass}( d\tilde{F}  \lfloor_{  \text{Int}(\mathcal{Q})  } )    m^{-\kappa/(\kappa+1)},
\\
     \mathcal{H}^{n-3}( \text{supp}(P) ) \leq C   m^{-1}  \int_{r/2<r(x)<4r} |F|^2.
\end{cases}
\]

\end{lem}

\begin{proof}
The idea is the same as the \emph{deformation theorem} in geometric measure theory \cite[Chapter 5]{LeonSimon}, which is to push the current into the skeleton of some cube complex. To begin with, we observe that on the interior of $\mathcal{Q}$, the $(n-3)$-current  $d\tilde{F}$  is supported on $B_{m^{-1}} ( \mathcal{C}_{\Lambda_0} )\cap \mathcal{Q}\subset \cup_i  \{ Q_i\}$ by Lemma \ref{lem:mollifiedcurvaturedF}, and is a closed current with finite mass, so in particular is a normal current.

In the first step, on each $Q_i$, we find some $(n-2)$-normal current $R_i$ supported on $Q_i$, and some $(n-3)$-normal current $P_i$ supported on the boundary of $Q_i$, such that
\[
\begin{cases}
	d\tilde{F}\lfloor_{Q_i }= \partial R_i+ P_i.
	\\
\text{Mass}(R_i) \leq Cs_i \text{Mass}( d\tilde{F}  \lfloor_{  Q_i } ), \quad 
\text{Mass}(P_i) \leq C\text{Mass}( d\tilde{F}  \lfloor_{  Q_i } )
\end{cases}
\]
The appearance of the cube size $s_i$ is due to rescaling, and Lemma \ref{lem:Whitney} gives a uniform upper bound $s_i\leq Cm^{-\kappa/(\kappa+1)}$.

Next, we replace $d\tilde{F} $ by $d\tilde{F} - \sum_{Q_i} R_i$, which is still a closed $(n-3)$-current on $\text{Int}(\mathcal{Q})$, and its mass only increased by a constant factor, but the support is pushed into the $(n-1)$-skeleton. We repeat this process on the $(n-1)$-skeleta of $\cup \{  Q_i \}$ etc, and eventually the current is pushed into the $(n-3)$-skeleton. We thus obtain $	d\tilde{F}  - P= \partial R$ with the claimed mass bound. 
Finally, since the support of $P$ is contained in the $(n-3)$-skeleton of $\cup Q_i$, the Hausdorff measure
\[
\mathcal{H}^{n-3}(P)\leq C \sum_{Q_i} s_i^{n-3} \leq Cm^{-1} \int_{ r/2< r(x)< 4r} |F|^2 dvol,
\]
by Lemma \ref{lem:Whitney}.
\end{proof}

\subsection{Weak limit of the Chern form}\label{sect:weaklimitChernform}

We fix any $1-n< \nu< -1-n/2$. By Prop. \ref{prop:parallelcurvatureL1},
there exist $L^2$-harmonic 2-forms $\sigma\in L^2\mathcal{H}^2(M)$ depending on $(A,\Phi)$, satisfying the uniform $L^1_{loc}$ estimate
 \[
 	\dashint_{r<r(x)<2r} | \frac{\Tr(\Phi F)}{4\pi m} +\sigma_0-\sigma|  \leq C \max(  r^{\nu},  m^{-1/2} r^{- 1-n/2}      ).
 \]
 Moreover, $\sigma$ represents an \emph{integral class} in $L^2\mathcal{H}^2\simeq H^2_c(M)$. 
 Now $\sigma_0$ is a fixed smooth 2-form independent of $m$, so after passing to subsequence, the following $L^1_{loc}$ weak limit exists:
 \begin{equation}\label{eqn: weaklimitChernform}
 \tilde{F}_\infty := \lim_{i\to +\infty} (-\frac{\Tr(\Phi F)}{4\pi m} +\sigma).
 \end{equation}
We denote $Q:= d\tilde{F}_\infty$ as the closed $(n-3)$-current measuring the failure of the Bianchi identity for $\tilde{F}_\infty$.

\begin{prop}\label{prop:weaklimitChernform}

The following convergence results hold: 

\begin{enumerate}
		\item   The limit satisfies
	$
	\tilde{F}_\infty= \lim_{i\to +\infty} (  \tilde{F}+\sigma  )=  \lim_{i\to +\infty}  (  -\frac{\Tr(\Phi F)}{4\pi |\Phi|}   +\sigma  )  ,
	$
	where $\tilde{F}$ is the mollified Chern form. Moreover, for any fixed $1-n<\nu<-1-n/2$,  the asymptotic estimate holds:
	\[
		\dashint_{r<r(x)<2r} | \tilde{F}_\infty-\sigma_0|  \leq C r^\nu,\quad \forall r>1.
	\]

	\item    The convergence  (\ref{eqn: weaklimitChernform}) is  strong  in $L^1$ on any fixed compact subset of $M$.

	\item The $(n-3)$-current $Q$ is a rectifiable current with real coefficient, and has finite total mass on $M$.

\end{enumerate}
\end{prop}

\begin{proof}
\textbf{Item 1}. By Lemma \ref{lem:mollifiedcurvaturecomparison}, the mutual difference 
between $\tilde{F}+\sigma$, $- \frac{\Tr(\Phi F)}{4\pi m}+\sigma$, and $- \frac{\Tr(\Phi F)}{4\pi |\Phi|}+\sigma$ converge in $L^1_{loc}$ to zero as $m\to +\infty$, hence they share the same weak limit. The asymptotic estimate is obtained by passing the uniform $L^1_{loc}$-bound to the $m\to +\infty$ limit on each fixed compact subset.

\textbf{Item 2}.  By Cor. \ref{cor:traceeqn} and the fact that $d\sigma=d^*\sigma=0$, we obtain a uniform $L^1_{loc} $-estimate for $(d+d^*)$ of $ \frac{\Tr(\Phi F)}{4\pi m}-\sigma$,
\[
\int_{ r(x)<r}  |d (\frac{\Tr(\Phi F)}{4\pi m} -\sigma)|+ |d^*  (\frac{\Tr(\Phi F)}{4\pi m} -\sigma)  | \leq   C+ Cm^{-1/2} r^{(n-4)/2}.
\]

In general, if $(d+d^*)u=f$ on $\{  r/2< r(x)<4r \}$, then there is a Calderon-Zygmund operator $\mathcal{P}$ of order $-1$ whose leading order symbol is inverse to that of $d+d^*$, and cutoff functions $\phi,\phi'$ supported on $\{  r/2< r(x)<4r \}$ and equal to one on $\{  r< r(x)<2r \}$
with the following smoothing effect:
\[
\norm{ u- \phi' \mathcal{P}(\phi f) }_{C^{k,\alpha}( \{  r < r(x)<2r \}  )}  \leq C \dashint_{  r/2< r(x)<4r  } |u|+ Cr \dashint_{ r/2<r(x)<4r} |f|.
\]
Using the Fréchet–Kolmogorov compactness theorem, the Calderon-Zygmund operator $\phi' \mathcal{P} \phi: L^1\to L^1$ is compact. On the other hand, the embedding of $C^{k,\alpha}$ into $L^1$ is compact. The upshot is that given a uniform bound on the local $L^1$-norms on $u$ and $(d+d^*)u$ on $\{  r/2< r(x)<4r \}$,  then the sequence is precompact in $L^1$ on $\{  r< r(x)<2r \}$. Applying this to $u= \frac{\Tr(\Phi F)}{4\pi m} -\sigma $, the convergence (\ref{eqn: weaklimitChernform}) is strong in $L^1$ on $\{  r< r(x)<2r \}$, for any given $r>1$.

\textbf{Item 3}. The distributional derivative commutes with limit, so
\[
Q=d\tilde{F}_\infty= \lim d(-\frac{\Tr(\Phi F)}{4\pi m} +\sigma)= \lim d(-\frac{\Tr(\Phi F)}{4\pi m} ).
\]
For any fixed $r>1$, we pass the bound in Cor. \ref{cor:traceeqn} to the $m\to +\infty$ limit, to bound the mass of the $(n-3)$-current $Q$ over $\{  r(x)<r \}$,
\[
\begin{split}
& \text{Mass}(Q\lfloor_{  \{  r(x)<r\}  }) 
\leq \text{Mass}   (   d (\frac{\Tr(\Phi F)}{4\pi m})    \lfloor_{  \{  r(x)<r\}     }  )               
\\
\leq & C\liminf  \int_{r(x)<r}   | d (\frac{\Tr(\Phi F)}{4\pi m})| 
\\
\leq & C  \liminf_{m\to +\infty}  (  C+ Cm^{-1/2} r^{(n-4)/2}) \leq C.
\end{split}
\]
The second line uses that for $k$-currents represented by $(n-k)$-forms with $L^1_{loc}$-coefficients, the mass is uniformly equivalent to the $L^1$-norm.
Taking $r\to +\infty$, the total mass satisfies $\text{Mass} (Q)\leq C$.

To prove \emph{rectifiability}, we focus on the interior of a coordinate cube $\mathcal{Q}$. By item 1,
we have an alternative limiting characterisation of $Q$,
\[
Q=d\tilde{F}_\infty=  \lim_{i\to +\infty} d(  \tilde{F}+\sigma  )=  \lim_{i\to +\infty}  d\tilde{F}  .
\]
Now the mass of an $(n-3)$-current represented by a 3-form with $L^1_{loc}$-coefficients, is uniformly comparable to the $L^1$-norm, so by (\ref{eqn:massdtildeF}),
\[
\text{Mass}(   d\tilde{F}    \lfloor_{  \text{Int}(\mathcal{Q})  }    ) \leq C\int_{ r(x)<2r} |d\tilde{F}|\leq  C(1+   m^{-1} r^{n-4})^{1/2} .
\]
For any fixed $r$, this bound is uniform for $m\to +\infty$. We apply the Deformation Lemma \ref{lem:deformation}, to find a rectifiable $(n-3)$-current $P$ supported on the $(n-3)$-dimensional skeleton of the cube complex, with
\[
\begin{cases}
	d\tilde{F}  - P= \partial R,
	\\
	\text{Mass}(  P ) \leq C\text{Mass}( d\tilde{F}  \lfloor_{  \text{Int}(\mathcal{Q})  } )   , \quad   \text{Mass}(  R ) \leq C\text{Mass}( d\tilde{F}  \lfloor_{  \text{Int}(\mathcal{Q})  } )    m^{-\kappa/(\kappa+1)},
	\\
	\mathcal{H}^{n-3}( \text{supp}(P) ) \leq C   m^{-1}  \int_{r/2<r(x)<4r} |F|^2\leq C(1+ m^{-1}r^{n-4}).
\end{cases}
\]
Thus  $d\tilde{F}$ and $P$ have the same weak limit $Q$, because the flat norm of their difference is given by the mass of $R$, which converges to zero as $O(m^{-\kappa/(\kappa+1)})$.

Now by the \emph{closure theorem} of Ambrosio-Kirchheim \cite[Thm. 8.5]{Ambrosio}, given a sequence of $k$-dimensional rectifiable normal currents $T_i$ (in general with real coefficients) weakly converging to a current $T$, if there is uniform upper bound on  $\text{Mass}(T_i)$, $\text{Mass}(\partial T_i)$ and $\mathcal{H}^k(\text{supp}(T_i))$, then the weak limit $T$ is also a $k$-dimensional rectifiable normal current.

In our case $P$ is a closed rectifiable current with real coefficients, and have a uniform local mass bound and Hausdorff measure bound, hence on any proper open subset of $\mathcal{Q}$, its weak limit $Q=\lim_{i\to +\infty} P$ must also be rectifiable with real coefficients. 
\end{proof}

%\begin{rmk}
%In the rectifiability statement, it is important that we work with normal currents, rather than sets. Given $k<n$, a general sequence of measures supported on $k$-rectifiable sets with a uniform bound on the $\mathcal{H}^k$-measure,  may converge to a measure with full support on an open set. 
	
%\end{rmk}
	
	\begin{rmk}\label{rmk:denseconcentration}
The above argument \emph{does not} rule out that $\Phi^{-1}(0)$ or the curvature concentration locus may become arbitrarily dense on some open subset of $M$ in the $m\to +\infty$ limit. Given $k<n$, a general closed $k$-rectifiable currrent with a uniform bound on the mass and the $\mathcal{H}^k$-measure,  may be $\epsilon$-dense on an open subset for any $\epsilon>0$.

%It only shows that the measure theoretic limit of $m^{-1}|\nabla \Phi|^2 dvol$ is contained in the support of $Q$, which has Hausdorff dimension $\leq n-3$.
	\end{rmk}

\subsection{Energy measure limit I}

The intermediate energy formula gives a uniform $L^1$ bound on $|\nabla \Phi|^2$,
while the monotonicity formula Lemma \ref{lem:monotonicity2} gives a uniform $L^1_{loc}$ bound on $|F|^2$. Hence after passing to subsequence, we can take the weak limit of these density measures on any fixed subset $\{  r(x)<r   \}$,
\begin{equation}\label{eqn:weaklimitmeasure}
\begin{cases}
\mu_{|\nabla \Phi|^2} := \lim_{i\to +\infty}  \frac{1}{2\pi m} |\nabla \Phi|^2 dvol,
\\
\mu_{|F|^2} :=  \lim_{i\to +\infty}  \frac{1}{2\pi m}   |F|^2 dvol.
\end{cases}
\end{equation}
(It will turn out that these weak limits are determined by $Q$, so taking this further subsequence is in fact unnecessary, see Prop. \ref{prop:measurelimit2} below.)

\begin{lem}\label{lem:measurelimit1}
The weak limits of the measures satisfy
\[
\mu_{|\nabla \Phi|^2}=	
\begin{cases}
	Q\wedge \psi,\quad &\text{$G_2$-monopole case},
	\\
 Q\wedge \text{Re}\Omega, \quad & \text{Calabi-Yau monopole case},
\end{cases}
\]
and the total measure bound
\[
\max(\int_M d\mu_{|\nabla \Phi|^2} ,   \int_M d\mu_{|F|^2}) \leq  
\begin{cases}
	\langle  \beta \cup \psi|_\Sigma, [\Sigma]\rangle, \quad & \text{$G_2$ case},
	\\
	\langle  \beta \cup \text{Re}(\Omega)|_\Sigma, [\Sigma]\rangle, \quad & \text{Calabi-Yau case}.
\end{cases}
\]
\end{lem}

\begin{proof}
We observe the following identity for $G_2$-monopoles:
\begin{equation}
	d\Tr (F\Phi) \wedge \psi= \Tr ( \nabla \Phi \wedge F)\wedge \psi=  \Tr ( \nabla \Phi \wedge *\nabla \Phi) = -2 |\nabla \Phi|^2 dvol.
\end{equation}
Passing this identity to the $m\to +\infty$ limit, for any function $f\in C^0_c(M)$,
\[
Q\wedge \psi (f):= Q( f\psi )= \lim  \int_M d( - \frac{\Tr (\Phi F)}{4\pi m}  )\wedge f\psi=  \frac{1}{2\pi } \lim \int_M f m^{-1} |\nabla \Phi|^2 .
\]
This shows the measure convergence $2\pi Q\wedge \psi= \lim m^{-1}|\nabla \Phi|^2 dvol $.
The Calabi-Yau version of the identity is
\[
d\Tr (F\Phi) \wedge \text{Re}\Omega= \Tr ( \nabla \Phi \wedge F)\wedge  \text{Re}\Omega =  \Tr ( \nabla \Phi \wedge *\nabla \Phi) = -2 |\nabla \Phi|^2 dvol,
\]
whence $2\pi Q\wedge \text{Re}\Omega= \lim m^{-1}|\nabla \Phi|^2 dvol $. This shows the formula for $\mu_{|\nabla \Phi|^2}$.  %In particular, on any compact subset properly contained in the complement of the $\text{supp}(Q)$, the weak limit of $m^{-1}|\nabla \Phi|^2 dvol$ is zero.

Now on any fixed compact subset $\{  r(x)\leq r  \}$, we can pass the monotonicity formula Lemma \ref{lem:monotonicity2} to the $m\to +\infty$ limit, to deduce
\[
\int_{ r(x)<r}  d\mu_{|F|^2} \leq  \begin{cases}
	\lim \frac{1}{2\pi m} E_\psi = 	\langle  \beta \cup \psi|_\Sigma, [\Sigma]\rangle, \quad & \text{$G_2$ case},
	\\
	\lim \frac{1}{2\pi m} E_\Omega= \langle  \beta \cup \text{Re}(\Omega)|_\Sigma, [\Sigma]\rangle, \quad & \text{Calabi-Yau case}.
\end{cases}
\]
Taking $r\to +\infty$ shows the total measure bound on $\mu_{|F|^2}$. The same argument using the intermediate energy formula (\ref{eqn:intermediateenergy2})(\ref{eqn:intermediateenergyCY2}) proves the total measure bound on $\mu_{|\nabla \Phi|^2}$.
\end{proof}

The limiting measure provides a pointwise bound on $|\tilde{F}_\infty|$.

\begin{lem}
Let $1-n <\nu< -1-n/2$, then for $1<r<r(p)<2r$, we have
\[
|\tilde{F}_\infty-\sigma_0|(p) \leq  C  r^{\nu} + C\int_{ r/2< r(x)<  4r  }    \frac{  1 }{ \text{dist}(x,p)^{n-1}  }d\mu_{|F|^2} (x) .
\]
\end{lem}

\begin{proof}
Using item 4 in Prop. \ref{prop:parallelcurvatureL1} and $|\nabla \Phi|\leq C|F|$, for $1<r<r(p)<2r$ and $s< r/3$, we deduce by Fubini theorem
\[
\begin{split}
	\dashint_{B(p,s)} |-\frac{\Tr(\Phi F)}{4\pi m} +\sigma-\sigma_0| \leq & C \max(  r^{\nu},  m^{-1/2} r^{- 1-n/2}      )
\\
+ & C \int_{ r/2< r(x)<  4r  }    \frac{  m^{-1} |F|^2 }{ \max(\text{dist}(x,p),s)^{n-1}  }dvol(x) . 
\end{split}
\]
We fix $r$ and pass to the $m\to +\infty$ limit, to obtain
\begin{equation}\label{eqn:Chernformaveragebound}
\begin{split}
	\dashint_{B(p,s)} |\tilde{F}_\infty-\sigma_0| \leq & C  r^{\nu} + C\int_{ r/2< r(x)<  4r  }    \frac{  1 }{ \max(\text{dist}(x,p),s)^{n-1}  }d\mu_{|F|^2}  .
\end{split}
\end{equation}
The pointwise estimate on $|\tilde{F}_\infty|$ follows from the Lebesgue differentiation theorem, as we take $s\to 0$.
\end{proof}

\subsection{Integrality of the Chern class}\label{sect:integrality}

Heuristically, we would like to pass the integrality of the Chern class of the $U(1)$-bundle to the $m\to +\infty$ limit; this is very closely related to the integer multiplicity of the $(n-3)$-rectifiable current $Q$. The technical difficulties are the following:

\begin{enumerate}
	\item  We only have $L^1_{loc}$-convergence, instead of $C^\infty_{loc}$ convergence away from a codimension three subset. %In fact, a general closed rectifiable current  may still have dense support. 

	\item A priori $\tilde{F}_\infty\in L^1_{loc}$, so the integral $\int_{\Sigma^2} \tilde{F}_\infty$ 
	on a given 2-cycle $\Sigma^2\subset M$ is not a priori well defined. One can only speak of  some average quantity of such integrals.

	\item  The integrality is obscured by the  mollification procedure.
	
\end{enumerate}

Let $\Sigma^2\subset M$ be a closed 2-submanifold inside $\{  r(x)<r \}$. We will perform the following averaging construction. Let $V\subset C^\infty(M, T M)$ be a $d$-dimensional space of vector fields, where $d$ is large enough to satisfy the following transversality requirement: at each point in $\{  r(x)<r \}$, we can find $v_1,\ldots v_n \in V$ with $C^1$-norms $\leq C(r)$, such that $v_1(x),\ldots v_n(x)$ is an orthonormal basis of $T_xM$. We fix an inner product on $V\simeq \R^d$, and take $\mathcal{P}$ as the unit ball in $V$.  We fix a probability measure $\rho(v) dv$ on $\mathcal{P}$ whose density function is smooth and compactly supported.

Let $0<\tau\ll 1$, then the exponential map defines a family of 2-cycles $\exp(\tau v)_* (\Sigma^2)$ parametrised by $v\in \mathcal{P}$, all contained in some open subset $U=B_\tau(\Sigma^2)\subset \{  r(x)<2r \}$. We view the union of all these 2-submanifolds as a $(d+2)$-current $\tilde{\Sigma}$ in $M\times \mathcal{P}$, with projection maps $\pi_M: \tilde{\Sigma}\to M$ and $\pi_{\mathcal{P}}: \tilde{\Sigma}\to  \mathcal{P}$. We then define the \emph{average Chern number} as 
\begin{equation}
\int_{ \tilde{\Sigma} }  (\tilde{F} +\sigma)\wedge   \rho(v)dv
	= \int_{\mathcal{P} }   \left( \int_{ \exp(\tau v)_*(\Sigma^2) }(\tilde{F}+\sigma)  \right)    \rho(v)dv  .
\end{equation}
This is a more robust quantity to pass to the $m\to +\infty$ limit, but unlike usual Chern numbers, it is only real valued.

\begin{lem}\label{lem:averageChernnumber1}
The average Chern number converges in the limit:
\[
\lim_{i\to +\infty}  \int_{ \tilde{\Sigma} } ( \tilde{F} +\sigma)\wedge \rho(v)dv=\int_{ \tilde{\Sigma} }  \tilde{F}_\infty \wedge \rho(v)dv.
\]
\end{lem}

\begin{proof}
%The soft argument is that by transversality, the fibres $\tilde{\Sigma}\cap \pi_M^{-1}(x )  $ are smooth manifolds with boundary for each $x\in U$, and by the smooth cutoff of the density $\rho(v)$ near the boundary of $\mathcal{P}$,  the pushforward current $\pi_{M*} (\tilde{\Sigma} ) $ is represented by a smooth form on $M$. Thus the result follows from weak convergence $\lim_{i\to +\infty} (\tilde{F}+\sigma) =\tilde{F}_\infty$. 

%Here is a more quantitative argument which works even if $\Sigma^2$ is an integral 2-cycle.	
By the Jacobian formula, and the quantitative transversality requirement above, %$\mathcal{H}^{d+2}(\tilde{\Sigma}) \leq C\tau^d \text{Vol}(\Sigma^2)$, 
for any 2-form $\alpha\in \Omega^2(U)$, 
\[
\int_{ \tilde{\Sigma}  } \pi_M^* \alpha \wedge \rho(v)dv =\int_{x\in U} \int_{ \tilde{\Sigma}\cap \pi_M^{-1}(x )  } \pi_M^* \alpha \wedge \rho(v)dv\leq C\tau^{-n} \text{Vol}(\Sigma^2) \norm{\alpha}_{L^1(U)} .
\]
The factor $\tau^{-n}$ comes from the Jacobian factor when we invert $n$ variables in $\mathcal{P}$ in terms of the $n$ local coordinates on $M$ in the $\tau$-rescaled exponential map.
We take $\alpha= \tilde{F}+\sigma$, which strongly converge to $\tilde{F}_\infty$ in $L^1(U)$ by item 1 in Prop. \ref{prop:weaklimitChernform}, so the average Chern number converges in the limit.
\end{proof}

The following Lemma says that under mild conditions, most members of the parameter space $\mathcal{P}$ yield integer numbers when $m$ is very large.

\begin{lem}\label{lem:noninteger}
Let $0<\kappa<1$ be a fixed number, and suppose $m\gg \Lambda_0$ is sufficiently large. We assume for any $s\leq Cm^{-\kappa/(1+\kappa)}$, the 2-submanifold $\Sigma^2$ can be covered by $C\text{Vol}(\Sigma^2) / s^2$ geodesic balls of radius $s$.  %(which holds in particular if the $C^1$-regularity scale of $\Sigma^2$ is at least $m^{-\kappa/(\kappa+1)})$). 

Then
there is a subset $E\subset \mathcal{P}$ with probability less than
\[
C\text{Vol}(\Sigma^2) m^{ - \frac{ 3\kappa+1 }{1+\kappa  }   }    \tau^{-n}  \int_{U' } |F|^2, 
\quad U'=   B_{ Cm^{- \kappa /(1+\kappa) }  } (U) ,
\]
such that  for every $v\in \mathcal{P}\setminus E$, 	
the integral $  \int_{ \exp(\tau v)_*(\Sigma^2) }(\tilde{F}+\sigma)  \in \Z  $. Here the constants can depend on $\kappa$ but are independent of the large $m$ and the details of $\Sigma^2$. 
\end{lem}

\begin{proof}
Since $\tilde{F}$ agrees with the Chern form $F_{U(1)}$ outside of $B_{m^{-1/2}} (\mathcal{C}_{\Lambda_0})$, as long as $\exp(\tau v)_*(\Sigma^2) $ stays away from $B_{m^{-1/2}}( \mathcal{C}_{\Lambda_0}   )$, we will have $ \int_{ \exp(\tau v)_*(\Sigma^2) }(\tilde{F}+\sigma)  \in \Z  $. The following argument is a quantitative version of the classical fact that a generic perturbation of a $k$-dimensional submanifold does not intersect a subset with Hausdorff dimension $<n-k$. The constants will be allowed to depend on $r$ in this argument.

By item 4 of Lemma \ref{lem:characteristicradius}, $B_{m^{-1}}(\mathcal{C}_{\Lambda_0})$ is contained in some $\mathcal{C}_\Lambda$ with $\Lambda\leq C\Lambda_0$. We apply Lemma \ref{lem:Vitali} to find a Vitali cover of $\mathcal{C}_\Lambda\cap U$ by balls $B(p_i, s_i)$, with $s_i\geq m^{-1}$, $ \max s_i\leq C m^{-\kappa/(\kappa+1)},  $ and
	\[
	m^{1+\kappa}  \sum  s_i^{ n-3+\kappa  }\leq  C\int_{U' } |F|^2, \quad U'= B_{   C  m^{- \kappa/ (\kappa+1)}  } (U).
\]

By the transversality requirement, given a ball $B(q, s_i)$, the probability of $v\in \mathcal{P}$ such that $\exp(\tau v)_* (   B(q, s_i))$ intersects $B(p_i, s_i)$, is bounded by $C(s_i/\tau)^n$. For each $s_i$, we can take a covering of $\Sigma^2$ by balls $B(q, s_i)$, and then apply the union bound on the total probability. Using the hypothesis about the covering of $\Sigma^2$, we see the probability of $v\in \mathcal{P}$ such that $\exp(\tau v)_*(\Sigma^2)$ intersects $\cup_i B(p_i, s_i)$, is bounded by
\[
\begin{split}
& C \sum_i \text{Vol}(\Sigma^2) s_i^{-2} (s_i/\tau)^n\leq C\text{Vol}(\Sigma^2)(\max {s_i})^{1-\kappa} \sum_i  s_i^{n-3+\kappa} \tau^{-n} 
\\
\leq & C\text{Vol}(\Sigma^2) m^{ -\frac{\kappa(1-\kappa) }{1+\kappa  }   }    \tau^{-n} m^{-\kappa-1} \int_{U' } |F|^2
\\
= &  C\text{Vol}(\Sigma^2) m^{ - \frac{ 3\kappa+1 }{1+\kappa  }   }    \tau^{-n}  \int_{U' } |F|^2.
\end{split}
\]
This is an upper bound of the probability that $ \int_{ \exp(\tau v)_*(\Sigma^2) }(\tilde{F}+\sigma)$  fails to be an integer. 
\end{proof}

Next we estimate the $L^1$-mean oscillation  $\int_{ \exp(\tau v)_*(\Sigma^2) }(\tilde{F}+\sigma) $ within the parameter space $\mathcal{P}$.

\begin{lem} \label{lem:Chernnumberoscillation}
We have
\[
\begin{split}
& \int_{\mathcal{P}\times \mathcal{P}} |  \int_{ \exp(\tau v)_*(\Sigma^2) }(\tilde{F}+\sigma)- \int_{ \exp(\tau v')_*(\Sigma^2) }(\tilde{F}+\sigma)| \rho(v)dv\rho(v')dv'
\\
\leq &  C\text{Vol}(\Sigma^2)\tau^{1-n}\int_U  m^{-1} | \nabla \Phi | |F|
\end{split}
\]
\end{lem}

\begin{proof}
For each pair $v, v'\in \mathcal{P}$, we can join the two surfaces $\exp(\tau v)_*(\Sigma^2)$ and $\exp(\tau v')_*(\Sigma^2)$ by a 3-dimensional manifold $X_{v,v'}$ with boundary, obtained by taking the union of the short geodesic arcs between $\exp_p(\tau v)$ and $\exp_p(\tau v')$ for all $p\in \Sigma^2$. Then by Stokes formula,
\[
 \int_{ \exp(\tau v)_*(\Sigma^2) }(\tilde{F}+\sigma)- \int_{ \exp(\tau v')_*(\Sigma^2) }(\tilde{F}+\sigma)= \int_{X_{v,v'}} d\tilde{F}.
\]
since $d\sigma=0$. Thus 
\[
\begin{split}
	& \int_{\mathcal{P}\times \mathcal{P}} |  \int_{ \exp(\tau v)_*(\Sigma^2) }(\tilde{F}+\sigma)- \int_{ \exp(\tau v')_*(\Sigma^2) }(\tilde{F}+\sigma)| \rho(v)dv\rho(v')dv'
	\\
	\leq &  \int_{\mathcal{P}\times \mathcal{P}} |  \int_{X_{v,v'}} d\tilde{F} |\rho(v) dv \rho(v') dv' .
\end{split}
\]
As $(v,v')$ varies among $\mathcal{P}\times \mathcal{P}$, the 3-manifolds with boundary $X_{v,v'}$ traces out a $(2d+3)$-dimensional manifold with boundaries and corners inside $M\times \mathcal{P}\times \mathcal{P}$. This has a projection map into $U=B_\tau (\Sigma^2)\subset M$, as $X_{v,v'}$ lies inside $U$.  Using the Jacobian formula, the quantitative transversality hypothesis ensures that the fibres of this projection are again smooth manifolds with boundaries and corners, and the integral is then bounded by
\[
C\text{Vol}(\Sigma^2)\tau^{1-n}\int_U  | d\tilde{F}|  \leq C\text{Vol}(\Sigma^2)\tau^{1-n}\int_U  m^{-1} | \nabla \Phi | |F|
\]
where we applied Lemma \ref{lem:mollifiedcurvaturedF} to estimate $|d\tilde{F}|$. The factor $\tau^{1-n}$ comes from the Jacobian factor when we invert $(n-1)$ variables in $\mathcal{P}$ together with the geodesic arclength variable in terms of the $n$ local coordinates on $M$, in the rescaled exponential map.
\end{proof}

\subsubsection{Integrality estimate in the limit}

We can now estimate how much the average Chern numbers for $\tilde{F}_\infty$ fails to be an integer, in terms of the limiting measures on the local neighbourhood of $\Sigma^2$.

\begin{prop}\label{prop:integralitylimitingChernnumber}
The distance of the  limiting average Chern number to the nearest integer is estimated by
\[
	\text{dist}(   	\int_{ \tilde{\Sigma} }  \tilde{F}_\infty \wedge   \rho(v)dv, \Z      )\leq C\text{Vol}(\Sigma^2)\tau^{1-n} \mu_{|\nabla \Phi|^2}(\bar{U})   ^{1/2}     \mu_{|F|^2}(\bar{U})  ^{1/2},
\]
where $\bar{U}=\overline{ B_\tau (\Sigma^2)}$. The constants do not depend on the details of $\Sigma^2$. 
\end{prop}

\begin{proof}
Let $\Sigma^2, U, \tau$ be fixed, and let $m$ be sufficiently large. By Lemma \ref{lem:noninteger}, except on a subset $E\subset \mathcal{P}$ with probability
\[
\int_E \rho(v)dv\leq C\text{Vol}(\Sigma^2) m^{ - \frac{ 3\kappa+1 }{1+\kappa  }   }    \tau^{-n}  \int_{U' } |F|^2
= O( m^{ - \frac{ 3\kappa+1 }{1+\kappa  }+1  }  )\to 0, \quad m\to +\infty,
\]
the number $ \int_{ \exp(\tau v)_*(\Sigma^2) }(\tilde{F}+\sigma)$ is an integer. We claim that 	
the deviation  of the  average Chern number from the integers satisfies the following bound,
\begin{equation*}
	\begin{split}
		& \text{dist}(   	\int_{ \tilde{\Sigma} }  (\tilde{F} +\sigma)\wedge   \rho(v)dv, \Z      )
		\\
		\leq & C\int_{\mathcal{P}\times \mathcal{P}} |  \int_{ \exp(\tau v)_*(\Sigma^2) }(\tilde{F}+\sigma)- \int_{ \exp(\tau v')_*(\Sigma^2) }(\tilde{F}+\sigma)| \rho(v)dv\rho(v')dv'.
		\end{split}
\end{equation*}
Since the LHS is bounded by one, without loss we can assume the RHS quantity is very small. Then there is some subset of $E'\subset \mathcal{P}$ with probability measure at least $50\%$, such that the integers $ \int_{ \exp(\tau v)_*(\Sigma^2) }(\tilde{F}+\sigma)$ share the same value $k$ for all $v\in E$. Then
\[
\begin{split}
& |	\int_{ \tilde{\Sigma} }  (\tilde{F} +\sigma)\wedge   \rho(v)dv-  k   |
\\
\leq &  \frac{1}{  \int_{E'} \rho(v')dv'  }  \int_{\mathcal{P}\times E'} |  \int_{ \exp(\tau v)_*(\Sigma^2) }(\tilde{F}+\sigma)- k| \rho(v)dv\rho(v')dv'
\\
\leq &  2 \int_{\mathcal{P}\times E'} |  \int_{ \exp(\tau v)_*(\Sigma^2) }(\tilde{F}+\sigma)- k| \rho(v)dv\rho(v')dv',
\end{split}
\]
which implies the claim.

Applying Lemma \ref{lem:Chernnumberoscillation}, we deduce for large enough $m$, 
\begin{equation*}
	\begin{split}
	\text{dist}(   	\int_{ \tilde{\Sigma} }  (\tilde{F} +\sigma)\wedge   \rho(v)dv, \Z      )
	\leq & C\text{Vol}(\Sigma^2)\tau^{1-n}\int_U  m^{-1} | \nabla \Phi | |F|
	\\
	\leq & C\text{Vol}(\Sigma^2)\tau^{1-n}  m^{-1}   ( \int_U |\nabla \Phi|^2    )^{1/2}  ( \int_U |F|^2    )^{1/2} .
\end{split}
\end{equation*}
By the convergence of the average Chern number in Lemma \ref{lem:averageChernnumber1}, we can take the $m\to +\infty$ limit to obtain
\[
	\begin{split}
&	\text{dist}(   	\int_{ \tilde{\Sigma} }  \tilde{F}_\infty \wedge   \rho(v)dv, \Z      )
	\\
	\leq & C\text{Vol}(\Sigma^2)\tau^{1-n}  \limsup_{i\to +\infty}  ( m^{-1}\int_U |\nabla \Phi|^2    )^{1/2} \limsup_{i\to +\infty } (m^{-1} \int_U |F|^2    )^{1/2} 
	\\
	\leq &  C\text{Vol}(\Sigma^2)\tau^{1-n} \mu_{|\nabla \Phi|^2}(\bar{U})   ^{1/2}     \mu_{|F|^2}(\bar{U})  ^{1/2} ,
	\end{split}
\]
where $\bar{U}$ is the closure of the open set $U$. 
\end{proof}

\section{Abelian monopole and calibrated cycle}

\subsection{Singular abelian $G_2$/Calabi-Yau monopole}

Given a sequence $(A_i,\Phi_i)$ as in the Main Setting, 
we now extract a singular abelian $G_2$-monopole (resp. Calabi-Yau monopole).

In Section \ref{sect:weaklimitChernform} we extracted a weak limit of the (mollified) Chern form $\tilde{F}_\infty= \lim_{i\to+\infty}  (- \frac{ \Tr(F\Phi ) }  {4\pi m}  +\sigma)$. We write $F_\infty=2\pi \tilde{F}_\infty$.
On the other hand, by Remark \ref{rmk:smallPhi}, we have a uniform $L^1_{loc}$ bound on $m^{-1}(m^2-|\Phi|^2)$ and its gradient, so after passing to subsequence, we can take a strong $L^1_{loc}$-limit,
\[
\Phi_\infty:= \lim_{i\to +\infty}  \frac{1}{2m}  ( |\Phi|^2 -m^2   )=  \lim_{i\to +\infty}  (-\frac{1}{4m}\Tr (\Phi^2) -      \frac{m}{2}  ).
\]
The pair $(F_\infty, \Phi_\infty)$ satisfies the following abelian $G_2$-monopole equation (resp. Calabi-Yau monopole equation) with a 3-current source term, in the distributional sense.

\begin{lem}\label{lem:singularabelianmonopole}
		 In the $G_2$-monopole case, $(F_\infty, \Phi_\infty)$ satisfies the following linear differential system in the distributional sense, 
	\begin{equation}
		\begin{cases}
			d F_\infty = 2\pi Q, 
			\\
		F_\infty \wedge \psi= *d\Phi_\infty,
		\end{cases}
	\end{equation}
	while in the Calabi-Yau monopole case,
	\begin{equation}
		\begin{cases}
			d F_\infty = 2\pi Q, 
		\\
		F_\infty \wedge \text{Re}\Omega= *d\Phi_\infty.
		\end{cases}
	\end{equation}

\end{lem}

\begin{proof}
The differential system is obtained by passing the system in Lemma \ref{lem:traceeqn} to the limit.
\end{proof}

\begin{rmk}
The source term $Q$ plays a role similar to a free boundary, and is not known a priori. The real content of the  equation $dF_\infty=Q$ lies in the rectifiability and integrality properties established in Section  \ref{sect:Chernform} (\cf also \cite[Thm. 1.2]{ParisePigatiStern}).
\end{rmk}

\subsection{Local blow up and Dirac singularity}

We now address the regularity and geometry of the triple $(F_\infty, \Phi_\infty, Q)$, using the standard technique of \emph{blow up}. Let $p\in M$. We denote the local rescaling map as $\lambda_{p,s}: T_p M\to M ,
x\mapsto  \exp_p(s x). 
$

By standard measure theory, the following properties hold $\mathcal{H}^{n-3}$-a.e. for points $p$ on the essential support of the rectifiable current $Q$:
\begin{itemize}
	\item The measures $\mu_{|F|^2}, \mu_{|\nabla \Phi|^2}$ have finite $(n-3)$-upper density, namely
	\[
	\limsup_{s\to 0}  s^{3-n} \mu_{|F|^2}( B(p,s) ) <+\infty,\quad  	\limsup_{s\to 0}  s^{3-n} \mu_{|\nabla \Phi|^2}( B(p,s) ) <+\infty.
	\]

	\item  The rectifiable current $Q$ has a unique approximate tangent plane at $p$,
	\[
Q_p:= \lim_{s\to 0}   (\lambda_{p,s}^{-1})_* Q = \Theta(p) [\![  T_p Q ]\!].
	\] 
	where the multiplicity $\Theta(p)$ is a priori a real number, and $[\![  T_p Q ]\!]$ denotes the current associated to an oriented $(n-3)$ plane $T_p Q\subset T_pM$.
\end{itemize}
Under these two conditions, a \emph{blow up limit} at $p$ is by definition any subsequential weak limit  $(f_p, \phi_p , Q_p)$ as $s\to 0$,
\[
f_p= \lim    \lambda_{p,s}^* F_\infty, \quad \phi_p= \lim s\lambda_{p,s}^* \Phi_\infty.
\]
This will be a key tool to understand the structure of $(F_\infty, \Phi_\infty, Q)$.

\begin{thm}\label{thm:weaklimitstructure}
	The weak limit $(F_\infty, \Phi_\infty, Q)$ satisfies the following:
	
	\begin{enumerate}

	\item    At any $p$ satisfying the above two conditions, the  blow up limit exists.

	\item The blow up limit is given by the Dirac singularity model:
	\begin{equation}
	\begin{cases}
		\phi_p= -\frac{  \Theta(p)   }{2 \text{dist}(\cdot, T_p Q)},
		\\
	f_p= -*_3 d\phi _p,
	\end{cases}
	\end{equation}
	where $*_3$ is the Hodge star in the 3-plane orthogonal to $T_p Q$, and we view $f_p$ as a singular 2-form on $T_p M$ translation invariant along $T_p Q$. In particular the blow up limit at such points $p$ is unique.

		\item The multiplicity $\Theta(p)$ is an \emph{integer}. The integral current $Q$ is a compactly supported coassociative cycle (resp. special Lagrangian cycle). In particular, $Q$ is smooth except on a subset of Hausdorff dimension at most $n-5$.

	\item The pair $(F_\infty, \Phi_\infty)$ is smooth away from the support of $Q$, and satisfies the following asymptotic estimate: 
	for any fixed $1-n<\nu< -1-n/2$, and $r$ large enough, we have
	\[
	\begin{cases}
	%\tilde{F}_\infty= \frac{F_\infty}{2\pi},\quad  
	 \norm{ \tilde{F}_\infty - \sigma_0  }_{C^{k,\alpha}(r< r(x)<2r) } \leq C r^\nu, \quad F_\infty=2\pi \tilde{F}_\infty,
	\\
	\norm{ \Phi_\infty  }_{C^{k,\alpha}(r< r(x)<2r) } \leq C r^{2-n}.
	\end{cases}
	\]

\end{enumerate}

\end{thm}

\begin{proof}
\textbf{Item 1}. By (\ref{eqn:Chernformaveragebound}), we have a uniform bound for small $s$,
\[
\begin{split}
	\dashint_{B(p,s)} |\tilde{F}_\infty-\sigma_0| \leq & C  r^{\nu} + C \int_{ r/2< r(x)<  4r  }    \frac{  1 }{ \max(\text{dist}(x,p),s)^{n-1}  }d\mu_{|F|^2}  
	\\
	\leq & C  +  Cs^{-2} \sup_{t< 4r } t^{3-n} \mu_{|F|^2}  (B(p,t)),
\end{split}
\]
hence by the finiteness assumption on the upper density, for any fixed $R>0$,
\[
\limsup_{s\to 0}  R^2	\dashint_{B(R)} |  \lambda_{p,s}^* F_\infty| = \limsup_{s\to 0}  (sR)^2	\dashint_{B(p,Rs)} |F_\infty|  \leq C\sup_{t<4r}  t^{3-n} \mu_{|F|^2}  (B(p,t))<+\infty.
\]
This justifies that $\lambda_{p,s}^* F_\infty$ has some weak limit $f_p$ in $L^1_{loc}$-topology, which satisfies the growth bound
\begin{equation}\label{eqn:blowupgrowth1}
R^2\dashint_{B(R)} |  f_p |  \leq C\sup_{t<4r}  t^{3-n} \mu_{|F|^2}  (B(p,t))<+\infty,\quad \forall R>0.
\end{equation}
The same kind of argument based on (\ref{eqn:smallPhi2}), shows that for any fixed $R>0$,
\[
\limsup_{s\to 0}  R	\dashint_{B(R)} s|  \lambda_{p,s}^* \Phi_\infty | = \limsup_{s\to 0} R s	\dashint_{B(p,Rs)} |\Phi_\infty|  \leq C\sup_{t<4r} t^{3-n} \mu_{|\nabla \Phi|^2}  (B(p,t))<+\infty,
\]
so we can take a weak limit $\phi_p$ of the sequence $s \lambda_{p,s}^* F_\infty$ in $L^1_{loc}$-topology, which satisfies the growth bound
\begin{equation}\label{eqn:blowupgrowth2}
	R\dashint_{B(R)} |  \phi_p |  \leq C\sup_{t<4r} t^{3-n} \mu_{|\nabla \Phi|^2}  (B(p,t))<+\infty,\quad \forall R>0.
\end{equation}
By taking the limit of the rescaled equations, we deduce that on $T_p M$, the blow up limit satisfies
\begin{equation}\label{eqn:blowup1}
\begin{cases}
df_p=2\pi Q_p,
\\
f_p \wedge \psi= * d\phi_p
\end{cases}
\end{equation}
in the $G_2$-case, and
\begin{equation}\label{eqn:blowup2}
	\begin{cases}
		df_p= 2\pi Q_p,
		\\
		f_p \wedge \text{Re}\Omega= * d\phi_p
	\end{cases}
\end{equation}
in the Calabi-Yau case.

\textbf{Item 2}. 
We now classify the solutions to (\ref{eqn:blowup1})(\ref{eqn:blowup2}) subject to the growth bounds (\ref{eqn:blowupgrowth1})(\ref{eqn:blowupgrowth2}). We will focus on the $G_2$ case, as the Calabi-Yau case is very similar.

We claim that for any $v\in T_pQ$, the solution is \emph{invariant under the translation} $\tau_v$ along $v$. First we observe the  homogeneous equation without the distributional source term $Q_p$ has only the zero solution under the growth bounds, by applying the mean value property. Hence the difference $(f_p,\phi_p)-\tau_v^*(f_p, \phi_p)=0$, which shows translation invariance.

Next, we compute
\[
d(*d\phi_p)= d(f_p\wedge \psi) = df_p\wedge \psi= 2\pi Q_p\wedge \psi=2\pi k dvol_{T_p Q}, 
\]
where $k\in \R$ is some constant. Together with the growth constraint on $\phi_p$, this shows $\phi_p=- \frac{k}{ 2\text{dist}(\cdot, T_pQ )  }$. Moreover, away from $Q_p$ we have $df_p=d^* f_p=0$, hence $\Lap f_p=0$. Since we are on the Euclidean space, this shows all the components of $f_p$ are harmonic, individually satisfying the growth constraints $O( \frac{1}{ \text{dist}(\cdot, T_pQ )^2  } )$. These harmonic functions belong  to a 3-dimensional vector space, spanned by the 3 components of $d( \frac{1}{ \text{dist}(\cdot, T_pQ ) }  )$.

We now decompose the forms into $(a,b)$ types, where $a$ is the number of $T^*_pQ$ factors, and $b$ is the number of factors in the normal direction. We write $f_p=f_p^{(2,0)}+f_p^{(1,1)}+ f_p^{(0,2)}$. Since $f_p$ is invariant under translation in the $T_pQ$ direction, we see $d(f_p^{(a,b)})$ has type $(a, b+1)$. Hence $df_p=2\pi Q_p$ implies
\[
d(f_p^{(2,0)})=0,\quad d(f_p^{(1,1)} )=0,\quad  d(f_p^{(0,2)} )=2\pi Q_p.
\]
Since the coefficient functions belong to the 3-dim space of harmonic functions, it follows that $f_p^{(2,0)}= f_p^{(1,1)}=0$, so $f_p$ is of pure type $(0,2)$ and $df_p=2\pi Q_p$. The only such $f_p$ whose coefficient functions belong to the 3-dim space of harmonic functions is
\[
f_p=-\Theta(p)  *_3 d( \frac{1}{   2\text{dist}(\cdot, T_pQ) }  ).
\]

Now $f_p\wedge \psi= *d\phi_p$ has type $(4,2)$, so by decomposing $\psi$ into various type components, we find $f_p\wedge \psi^{(4,0)}= * d\phi_p $, while $f_p\wedge \psi^{(3,1)}=0$. This implies $\psi^{(3,1)}=0$, which means $\iota_v \psi |_{T_pQ}=0$ for every $v\perp T_pQ$, hence $T_pQ$ is a \emph{coassociative} 4-plane. We comment that in the Calabi-Yau analogue, we would obtain $\iota_v\text{Re}(\Omega) |_{T_pQ}=0$ for every $v\perp T_pQ$, hence $T_pQ$ is a \emph{special Lagrangian} 3-plane.

In our convention $\psi$ (resp. $\text{Re}\Omega$) gives the orientation of $T_p Q$. Using $f_p\wedge \psi^{(4,0)}= *d\phi_p$, we identify the coefficient $k= \Theta(p)$, namely
\[
\phi_p= -\frac{\Theta(p) }{ 2\text{dist}(\cdot, T_pQ )  }.
\]
This shows $f_p= *_3 d\phi_p$.

Finally, by Lemma \ref{lem:measurelimit1}, $Q\wedge \psi$ (resp. $Q\wedge \text{Re}\Omega$) is a positive measure. Thus under the orientation convention, the real multiplicity $\Theta(p)$ is \emph{positive}. (The caveat here is that the definition of calibrated cycles requires the positivity of the multiplicity function with respect to the orientation of the tangent planes determined by the calibration form.)

\textbf{Item 3}. We now prove the integrality of $\Theta(p)$, which is more delicate. We take the unit sphere $S^2$ inside the normal plane in $T_p M$, so $S^2$ links with $T_p Q\subset T_p M$. Via the exponential map, for each small $s\ll 1$, we obtain a 2-sphere $\Sigma_s=\lambda_{p,s*} (S^2) $ in $B(p,s)\subset M$. We apply the averaging procedure in Section \ref{sect:integrality}, to the surface $\Sigma_s$, with parameter $\tau=s/2$. Thus $\text{Vol}(\Sigma_s)\leq Cs^2$, and Prop. \ref{prop:integralitylimitingChernnumber} gives a uniform estimate in $s$,
\[
\begin{split}
\text{dist}(   	\int_{ \tilde{\Sigma}_s }  \tilde{F}_\infty \wedge   \rho(v)dv, \Z      )\leq &
C\text{Vol}(\Sigma_s)\tau^{1-n} \mu_{|\nabla \Phi|^2}(\bar{U}_s)   ^{1/2}     \mu_{|F|^2}(\bar{U}_s)  ^{1/2}
\\
\leq & Cs^{3-n} \mu_{|\nabla \Phi|^2}(\bar{U}_s)   ^{1/2}     \mu_{|F|^2}(\bar{U}_s)  ^{1/2},
\end{split}
\]
where $U_s= B_{\tau}( \Sigma_s )$ is a shell around the $S^2$ of length scale $\sim s$. We now pass to the $s\to 0$ limit, to deduce
\[
\begin{split}
& \text{dist}(   	\int_{ \tilde{S^2} } f_p  \wedge   \rho(v)dv, \Z      ) 
\\
\leq&  C \limsup_{s\to 0}  s^{3-n}\mu_{|\nabla \Phi|^2}(\bar{U}_s)   ^{1/2}     \mu_{|F|^2}(\bar{U}_s)  ^{1/2}
\\
\leq & C( \limsup_{s\to +\infty} s^{3-n}\mu_{|\nabla \Phi|^2}(\bar{U}_s)  ) ^{1/2}( \limsup_{s\to +\infty} s^{3-n}\mu_{|F|^2}(B(p,s))  ) ^{1/2}.
\end{split}
\]
Now at the point $p$, the factor $\limsup_{s\to +\infty} s^{3-n}\mu_{|F|^2}(B(p,s)) <+\infty$ by assumption. On the other hand, $\mu_{|\nabla \Phi|^2}= Q\wedge \psi$ by Lemma \ref{lem:measurelimit1}, and $Q_p$ is the blow up limit of $Q$ at $p$, hence we can identify the weak limit of measures
\[
\lim_{s\to 0}  s^{3-n} (\lambda_{p,s  }^{-1})_*  \mu_{|\nabla \Phi|^2} = \lim_{s\to 0}  s^{3-n} (\lambda_{p,s  }^{-1})_* (Q\wedge \psi)=    Q_p\wedge \psi.
\]
This limiting measure is the Lebesgue measure on $T_p Q$ with multiplicity $\Theta(p)$. However, the rescaled set $\lambda_{p,s}^{-1}(U_s)$ is by construction contained in a shell around the unit 2-sphere linking $T_p Q$, and in particular stays uniformly bounded away from $T_pQ$. Hence
\[
\limsup_{s\to +\infty} s^{3-n}\mu_{|\nabla \Phi|^2}(\bar{U}_s) =0,
\]
and the upshot is that $\text{dist}(   	\int_{ \tilde{S^2} } f_p  \wedge   \rho(v)dv, \Z      ) =0$, namely the average Chern number for the blow up limit is an integer. In the Dirac singularity model, the average Chern number is the same as multiplicity factor $\Theta(p)$, so $\Theta(p)\in \Z$. The same argument works identically in the $G=SO(3)$ case, since all the 2-spheres are contained in a local ball on $M$, where the $SO(3)$-bundle lifts to an $SU(2)$-bundle.

We have now proved that $Q$ is an integer rectifiable closed current, whose oriented tangent planes $T_p Q$ are $\mathcal{H}^{n-3}$-a.e. coassociative (resp. special Lagrangian).  In short, $Q$ is a \emph{coassociative cycle} (resp. \emph{special Lagrangian cycle}). In particular, $Q$ is a local area minimizer, and by Almgren's big regularity theorem, $Q$ is represented by a smooth $(n-3)$-dimensional submanifold with multiplicity, away from a subset of Hausdorff dimension at most $n-5$. By item 3 in Prop. \ref{prop:weaklimitChernform}, $Q$ has  finite total mass, so using the monotonicity formula for minimal surfaces, we see that the support of $Q$ is contained in a fixed \emph{compact subset} of $M$.

\textbf{Item 4}. The support of $Q$ is a closed subset of $M$. In its complement $M\setminus \text{supp}(Q)$, the singular abelian  $G_2$/Calabi Yau monopole equation implies $F_\infty$ is closed and coclosed, so $(F_\infty, \Phi_\infty)$ is \emph{smooth} by elliptic regularity.

By item 1 in Prop. \ref{prop:weaklimitChernform}, for any fixed $1-n<\nu< -1-n/2$, we have
\[
	\dashint_{r<r(x)<2r} | \tilde{F}_\infty-\sigma_0|  \leq C r^\nu,\quad \forall r>1.
\]
Recall that $\sigma_0$ is closed and coclosed outside a compact set.
Since $Q$ is supported inside a fixed compact subset, for large enough $r$, we can apply elliptic regularity to show
\[
\norm{ \tilde{F}_\infty -\sigma_0  }_{C^{k,\alpha}(r< r(x)<2r) } \leq C r^\nu. 
\]
By Remark \ref{rmk:smallPhi}, we have the $L^1_{loc}$ bound
\[
\dashint_{r<r(x)<2r } |\Phi_\infty| \leq C r^{2-n},
\]
and since $\Phi_\infty$ is harmonic away from the support of $Q$ by the singular abelian $G_2$-monopole equation (resp. Calabi-Yau monopoles), we deduce that for sufficiently large $r$,
\[
\norm{ \Phi_\infty  }_{C^{k,\alpha}(r< r(x)<2r) } \leq C r^{2-n}.
\]
This concludes the asymptotic estimate.
\end{proof}

\begin{rmk}\label{rmk:Cinftylimitnotsequence}
Conceptually, the fact that $M\setminus \text{supp}(Q)$ is  a non-empty open set, only follows
\emph{a posteriori} when we proved $Q$ is a minimal surface; a general rectifiable current may have dense support. So unless we first prove the integer multiplicity of $Q$, we would not know there is any open subset to apply the regularity theory of the elliptic system.

It is also worth noting that we only proved the $C^\infty_{loc}$ estimate for the limit $(F_\infty, \Phi_\infty)$; this argument \emph{does not} show that there is a fixed open subset of $M$ independent of large $m$, where the curvature  has uniform $C^\infty_{loc}$ estimate for the sequence $(A, \Phi)$. It is an interesting question if there is any open dense subset of $M$ where the sequence converges up to gauge in the $C^\infty_{loc}$-topology.
\end{rmk}

\begin{rmk}\label{rmk:monotonicityminsurface}
	The minimal surface $Q$ satisfies the monotonicity formula that for $p\in \text{supp}(Q)$, the $(n-3)$-dimensional volume ratio
	\[
	e^{Cs/r(p)}  s^{3-n}\text{Mass}(Q\lfloor_{B(p,s)} )
	\]
	is increasing for $s\lesssim r(p)$. The exponent $n-3$ here is stronger than the monotonicity formula for Yang-Mills-Higgs theory, which is $n-4$. It is an interesting question whether there is an \emph{effective version of the $(n-3)$-dimensional monotonicity formula} which applies to $(A,\Phi)$ with large but finite $m$.

\end{rmk}

We now prove some homological properties of the weak limit $(F_\infty, \Phi_\infty, Q)$.

\begin{cor}\label{cor:homology}
The limit $(F_\infty, \Phi_\infty, Q)$ satisfies the following:
\begin{enumerate}
\item 
The homology class of  $[Q]\in H_{n-3}(M)$ is Poincar\'e dual to the image of the monopole class $\beta$ under the connecting map $H^2(\Sigma)\to H^3_c(M)$.

\item The class defined by the closed form $\tilde{F}_\infty$ in $H^2(M\setminus \text{supp}(Q), \R)$ is an integral class. In particular, there exists a  connection $A_\infty$ on some $U(1)$-bundle over $M\setminus \text{supp}(Q)$ whose curvature is $F_\infty$.

\item The total mass is
\[
\text{Mass}(Q)= 
\begin{cases}
	\int_Q \psi =	\langle  \beta \cup \psi|_\Sigma, [\Sigma]\rangle, \quad & \text{$G_2$ case},
	\\
	\int_Q \text{Re}\Omega =	\langle  \beta \cup \text{Re}(\Omega)|_\Sigma, [\Sigma]\rangle, \quad & \text{Calabi-Yau case}.
\end{cases}
\]

\end{enumerate}

\end{cor}

\begin{proof}
\textbf{Item 1}. The image of $\beta$ under  the connecting map $H^2(\Sigma)\to H^3_c(M)$ is represented by the compactly supported form $d\sigma_0$. On the other hand, by item 4 in Theorem \ref{thm:weaklimitstructure}, the distributional 2-form $\tilde{F}- \sigma_0$ decays faster than quadratically at infinity, so it restriction to the boundary at infinity is zero. Since 
$
d(\tilde{F}- \sigma_0  )=Q- d\sigma_0,
$
the class $[Q]\in H_{n-3}(M)\simeq H^3_c(M)$ is cohomologous to $[d\sigma_0]\in H^3_c(M)$.

\textbf{Item 2}. Here $M\setminus \text{supp}(Q)$ is an open manifold. By a classical result of Thom, every integral homology class up to dimension $6$ can be represented by a smooth submanifold. We represent an arbitrary class in $H_2(M\setminus \text{supp}(Q), \Z)$ by a smooth submanifold $\Sigma^2\subset M\setminus \text{supp}(Q)$, and apply the averaging procedure in Section \ref{sect:integrality} to $\Sigma^2$ with a sufficiently small parameter $\tau$. By Prop. \ref{prop:integralitylimitingChernnumber},
\[
	\text{dist}(   	\int_{ \tilde{\Sigma} }  \tilde{F}_\infty \wedge   \rho(v)dv, \Z      )\leq C\text{Vol}(\Sigma^2)\tau^{1-n} \mu_{|\nabla \Phi|^2}(\bar{U})   ^{1/2}     \mu_{|F|^2}(\bar{U})  ^{1/2}.
\]
For small enough $\tau$, the neighbourhood $U= B_\tau(\Sigma^2)$ is disjoint from the support of  $Q$. Since the limiting measure $\mu_{|\nabla \Phi|^2}$ is supported on $Q$, this shows  $\mu_{|\nabla \Phi|^2}(\bar{U}) =0$, so
the average Chern number $	\int_{ \tilde{\Sigma} }  \tilde{F}_\infty \wedge   \rho(v)dv$
is an integer. But since $\tilde{F}$ is a smooth closed 2-form on $M\setminus \text{supp}(Q)$, and all the 2-cycles $\exp(\tau v)_*( \Sigma^2 )$ are homologous to $\Sigma^2$ inside $M\setminus \text{supp}(Q)$, the average Chern number is simply $\int_{\Sigma^2} \tilde{F}$.

The conclusion is that $\int_{\Sigma^2} \tilde{F}$ is an integer for any integral class in  $H_2(M\setminus \text{supp}(Q), \Z)$, so $[\tilde{F}]\in H^2(M\setminus \text{supp}(Q), \R)$ is an \emph{integral class}. This allows us to construct a $U(1)$-bundle over $M\setminus \text{supp}(Q)$ whose first Chern class realizes the class $[\tilde{F}]$, and a connection $A_\infty$ whose curvature 2-form is $F_\infty$.

\textbf{Item 3}. Since $Q$ is a coassociative cycle (resp. special Lagrangian cycle), 
\[
\text{Mass}(Q)= 
\begin{cases}
	\int_Q \psi,\quad &\text{$G_2$-monopole case},
\\
\int_Q \text{Re}\Omega, \quad & \text{Calabi-Yau monopole case}.
\end{cases} 
\]
By item 1 and Stokes theorem, in the $G_2$-case,
\[
	\int_Q \psi = \int_M d\sigma_0 \wedge \psi= \int_\Sigma \sigma_0\wedge \psi= \langle \beta\cup \psi|_\Sigma, [\Sigma]\rangle.
\]
The Calabi-Yau case is completely analogous.
\end{proof}

\begin{rmk}
From the intermediate energy formula, the quantity $	\langle  \beta \cup \psi|_\Sigma, [\Sigma]\rangle$ (resp. $ \langle  \beta \cup \text{Re}(\Omega)|_\Sigma, [\Sigma]\rangle$) is non-negative, and it is zero iff $\Phi$ is parallel, so $(A,\Phi)$ reduces to a smooth $U(1)$ $G_2$-monopole (resp. Calabi-Yau monopole). A consequence of item 3 in Cor. \ref{cor:homology} is that this quantity is strictly positive iff $Q$ is a nontrivial coassociative cycle (resp. special Lagrangian cycle), in the homology class determined by the monopole class $\beta$ as in item 1 of Cor. \ref{cor:homology}.

The significance is that if we are given the sequence $(A,\Phi)$ as in the Main Setting, then this provides a  \emph{nonperturbative existence theorem for the coassociative (resp. special Lagrangian) cycle in the prescribed homology class}! In contrast, all the previous existence results for these calibrated cycles to the author's knowledge, depend either on a semi-explicit solvable ansatz (including symmetry reduction, gluing constructions, or integrable system techniques), or relies on the underlying manifold being close to some degeneration limit (eg. the conifold degeneration of Calabi-Yau manifolds, or the large complex structure limit in the SYZ conjecture).

\end{rmk}

\begin{rmk}
	In the $G=SO(3)$ variant case, in general $[2\tilde{F}_\infty]\in H^2(M\setminus \text{supp}(Q),\R)$ defines an integral class. The root of this issue comes from the integrality of the Chern class, \cf Remarks \ref{rmk:SO(3)1}, \ref{rmk:SO(3)2}. 
\end{rmk}

	\section{Energy identity}

	\subsection{Weak limit of energy density}

	We now analyze how the $L^2$-curvature density is distributed in the $m\to +\infty$ limit.
	 We denote $Q_{sm}$ as the smooth locus of the coassociative (resp. special Lagrangian) cycle $Q$.  Using the normal exponential map, 
	 we can identify a small tubular neighbourhood $\mathcal{U}$, with the open subset of the normal bundle of $Q_{sm}$ obtained as a union of 2-discs,
	 \[
	 \bigcup_{y\in Q_{sm}}   D^2_{e^{-t_0(y)}}\cap N_y Q.
	 \]
	 where $t_0(y)$ is some smooth function on $Q_{sm}$, which degenerates to $+\infty$ near the singular points of $Q$. We have a natural projection map $\pi: \mathcal{U}\to Q_{sm}$.  Within $\mathcal{U}$ we decompose the curvature 2-form according to the tangential/normal components,
	\[
	F= F^{tt}+ F^{tn}+F^{nn}.
	\]
	where the superscript $t$ denotes the form factor orthogonal to the normal fibres, and $n$ denotes the form factor along the normal fibres.

	We can thereby refine the limiting curvature density measures $\mu_{|F|^2}$ and $\mu_{|\nabla\Phi|^2}$ into type components:
	\[
	\begin{cases}
		\mu_{|F^{nn}|^2} :=  \lim_{i\to +\infty}  \frac{1}{2\pi m} \int_M    |F^{nn}|^2 dvol,
		\\
		\mu_{|F^{tn}|^2} :=  \lim_{i\to +\infty}  \frac{1}{2\pi m} \int_M |F^{tn}|^2 dvol,
		\\
		\mu_{|F^{tt}|^2} :=  \lim_{i\to +\infty}  \frac{1}{2\pi m} \int_M F^{tt}  |^2 dvol,
		\\
		\mu_{ |\nabla^n \Phi|^2} :=  \lim_{i\to +\infty}  \frac{1}{2\pi m} \int_M |\nabla^n \Phi |^2 dvol,
		\\
		\mu_{ |\nabla^t \Phi|^2} :=  \lim_{i\to +\infty}  \frac{1}{2\pi m} \int_M |\nabla^t \Phi |^2 dvol.
	\end{cases}
	\]
	These are a priori subsequential weak limits, but Prop.  \ref{prop:measurelimit2} says that they are actually completely determined by $Q$, so there is no need to pass to a further subsequence. %The main upshot is that  the limiting curvature density measure is concentrated on the support of $Q$, all the tangential components vanish in the limit, 
	%	 in the normal component, and is proportional to the volume measure on $Q$. 

	\begin{prop}\label{prop:measurelimit2}
	The limiting curvature density measures satisfy the following:
	\begin{enumerate}
		\item 	The density measures are entirely concentrated on $Q_{sm}$,
		\[
		\mu_{|F|^2} (M\setminus Q_{sm}) ,\quad 	\mu_{|\nabla \Phi|^2} (M\setminus Q_{sm}) =0,
		\]
		and all the tangential components vanish,  
		\[
		\mu_{ |\nabla^t \Phi|^2}= \mu_{|F^{tn}|^2}=\mu_{|F^{tt}|^2}=0,
		\]
		and the density measure for  normal components agrees with the volume measure on $Q$ counted with multiplicity,
		\[
	\mu_{|F^{nn}|^2} = \mu_{ |\nabla^n \Phi|^2} =
		\begin{cases}
			Q\wedge \psi,\quad &\text{$G_2$-monopole case},
			\\
			Q\wedge \text{Re}\Omega, \quad & \text{Calabi-Yau monopole case}.
		\end{cases}
		\]

		\item The fibrewise monopole equation holds in the following integral sense:
		\[
		\lim_{i\to +\infty} \int_\mathcal{U} |*_3 F^{nn}- \nabla^n \Phi |^2 dvol =0.
		\]

		\item The second Chern form satisfies the $L^1$-convergence
	$
	\int_M |	m^{-1}  \Tr(F\wedge F)| dvol \to 0
$
as $i\to +\infty$.

	\end{enumerate}

	\end{prop}

	\begin{proof}
		\textbf{Item 1}.
	Let $\rho\geq 0$ be a test function supported on a compact subset of $Q_{sm}$, and we regard $\rho dvol_Q$ as a smooth $(n-3)$-form on $Q_{sm}$, which can be pulled back to the tubular neighbourhood via the projection map $\pi$. On each normal fibre $N_yQ$, under our inner product convention on the Lie algebra $|a|^2= - \frac{1}{2}\Tr (a^2)$, we deduce 
	\begin{equation*}
 | *_3F^{nn} - \nabla^n \Phi|^2dvol_{N_yQ}-  \Tr(\nabla \Phi\wedge F) |_{N_yQ}=  ( |\nabla^n \Phi|^2+ |F^{nn}|^2) dvol_{N_yQ} .
	\end{equation*}
	Hence using $d\Tr( \Phi F)= \Tr(\nabla \Phi\wedge F)$, we have
	\[
	\begin{split}
	&	-\pi^*(\rho dvol_Q)\wedge d\Tr (\Phi F) +\pi^*\rho  | *_3F^{nn} - \nabla^n \Phi|^2dvol_M
		\\
		\leq  & (1+ C \text{dist}(\cdot, Q)) \pi^*\rho( |\nabla^n \Phi|^2+ |F^{nn}|^2) dvol_M .
	\end{split}
	\]
	where the $1+ C\text{dist}(\cdot, Q)$ factor accounts for the discrepancy between $\pi^* dvol_Q\wedge dvol_{N_yQ}$ and the volume form on $M$. We can multiply by a nonnegative smooth cutoff function $\eta$ supported in the tubular neighbourhood, and equal to one along $Q_{sm}$. Upon  integration,
	\[
	\begin{split}
	& 	  \int_M \eta \pi^*(\rho dvol_Q)\wedge d(\frac{-\Tr (\Phi F)}{4\pi m})+  \frac{1}{4\pi m}  \int_M \eta \pi^* \rho | *_3F^{nn} - \nabla^n \Phi|^2dvol
	\\
	\leq &  \frac{1}{4\pi m}   \int_M \eta (1+ C \text{dist}(x, Q)) \pi^*\rho( |\nabla^n \Phi|^2+ |F^{nn}|^2) dvol.
	\end{split}
	\]  
	We recall the weak convergence $Q=\lim_{i\to \infty} d( - \frac{\Tr (\Phi F)}{4\pi m}  ) $, hence we can pass to the limit to obtain
	\begin{equation}\label{eqn:monopole1}
	\begin{split}
		&	 \frac{1}{2} \int_M  \eta(1+ C \text{dist}(\cdot, Q)) \pi^*\rho (d\mu_{|F^{nn}|^2} + d\mu_{ |\nabla^n \Phi|^2} ) 
			\\
			\geq&  \int_Q \rho dvol_Q + \limsup_{i\to +\infty}  \frac{1}{4\pi m}  \int_M  \eta \pi^* \rho | *_3F^{nn} - \nabla^n \Phi|^2dvol.
	\end{split}
	\end{equation}
	where 
	\[
	 \int_Q \rho dvol_Q
		=  \begin{cases}
			\int_Q \rho\psi , \quad &\text{$G_2$-monopole case},
			\\
			\int_Q \rho \text{Re}\Omega, \quad & \text{Calabi-Yau monopole case}. 
		\end{cases}
	\]

	As (\ref{eqn:monopole1}) holds for every $\eta$ and $\rho$, we deduce an inequality of measures on $Q_{sm}$ after dropping the $| *_3F^{nn} - \nabla^n \Phi|^2$ term, 
	\begin{equation}
\frac{1}{2}	(\mu_{|F^{nn}|^2} + \mu_{ |\nabla^n \Phi|^2} )\lfloor_{Q_{sm}} \geq 
	\begin{cases}
			Q\wedge \psi,\quad &\text{$G_2$-monopole case},
		\\
		Q\wedge \text{Re}\Omega, \quad & \text{Calabi-Yau monopole case}.
	\end{cases}
	\end{equation}
In particular, the total measure on $Q_{sm}$ satisfies the comparison
	\[
	\begin{split}
	&	\frac{1}{2}	(\mu_{|F^{nn}|^2} + \mu_{ |\nabla^n \Phi|^2} )(Q_{sm} )
		\\
		\geq &
	\begin{cases}
	\int_Q \psi  = 	\langle  \beta \cup \psi|_\Sigma, [\Sigma]\rangle  ,\quad &\text{$G_2$-monopole case},
		\\
\int_Q \text{Re}\Omega=   	\langle  \beta \cup \text{Re}(\Omega)|_\Sigma, [\Sigma]\rangle, \quad & \text{Calabi-Yau monopole case},
		\end{cases}
	\end{split}
	\]
	where we used the identity in item 3 of Cor. \ref{cor:homology}.

	On the other hand,  the following measure identity holds on a neighbourhood of $Q_{sm}$,
	\[
	\begin{cases}
	\mu_{|F|^2}= \mu_{|F^{nn}|^2}+ \mu_{|F^{tn}|^2}+ \mu_{|F^{tt}|^2} \geq  \mu_{|F^{nn}|^2},
	\\
	\mu_{|\nabla\Phi|^2}= \mu_{ |\nabla^n \Phi|^2}+ \mu_{ |\nabla^t \Phi|^2}\geq  \mu_{ |\nabla^n \Phi|^2},
	\end{cases}
	\]
	and  Lemma \ref{lem:measurelimit1} gives the total measure estimate
	\[
	\max( \mu_{|\nabla \Phi|^2} (M),  \mu_{|F|^2}(M)) \leq  
	\begin{cases}
		\langle  \beta \cup \psi|_\Sigma, [\Sigma]\rangle, \quad & \text{$G_2$ case},
		\\
		\langle  \beta \cup \text{Re}(\Omega)|_\Sigma, [\Sigma]\rangle, \quad & \text{Calabi-Yau case}.
	\end{cases}
	\]
	The compatibility of these inequalities forces the equality to be achieved everywhere, which implies item 1.

	\textbf{Item 2}. We first comment that since $\mu_{|F|^2}$ and $\mu_{|\nabla\Phi|^2}$ put no measure on $M\setminus Q_{sm}$, any weak limit of $ | *_3F^{nn} - \nabla^n \Phi|^2dvol $ must also put no measure on  $M\setminus Q_{sm}$. We feed item 1 back into (\ref{eqn:monopole1}), to see that
	\[
	 \limsup_{i\to +\infty}  \frac{1}{4\pi m}  \int_M  \eta \pi^* \rho | *_3F^{nn} - \nabla^n \Phi|^2dvol =0.
	\]
	Since this holds for any $\rho$ and $\eta$, item 2 follows.

	\textbf{Item 3}. This follows from item 1 since in the expansion of $F\wedge F$ into type components, the term $F^{nn}\wedge F^{nn}=0$ for degree reasons.
	\end{proof}
	
	\begin{cor} \label{cor:fibrewisestrongL1}
 The following strong $L^1$-convergence results hold for any given compact subset  $K\subset Q_{sm}$:
 
 \begin{enumerate}
 	\item The fibrewise energy function $E_y= \int_{\mathcal{U}\cap \pi^{-1}(y)} |F|^2 dvol_{N_yQ}$ satisfies the $L^1$-convergence
 	\[
 	\lim_{i\to +\infty}  \int_{K} | \frac{1}{2\pi m} E_y- \Theta(y)| dvol_Q =0,
 	\]
 	where $\Theta(y)$ is the multiplicity function of the current $Q$.

 	\item  For any $y\in Q_{sm}$, we take the family of small 2-spheres linking with $Q_{sm}$,
 	\[
 	S^2_{y,t}= \{  dist(\cdot , y)=e^{-t}  \}\subset \pi^{-1}(y)\cap \mathcal{U},\quad t_0(y) \leq t\leq t_0(y)+1.
 	\]
Then the mollified Chern form $\tilde{F}$ satisfies
 	\[
 	\lim_{i\to +\infty} \int_{K} \int_{t_0(y)}^{t_0(y)+1} | \int_{S^2_{y,t} } \tilde{F}  - \Theta(y)| dvol_Q dt =0.
 	\]

 \end{enumerate}

	\end{cor}

\begin{proof}
\textbf{Item 1}. 
By item 2 in Prop. \ref{prop:measurelimit2}, the fibrewise integral of $|*_3 F^{nn}- \nabla^n \Phi|^2$ converges strongly to zero as a function on $Q_{sm}$. Consequently, we can replace $E_y$ by 
\[
-\frac{1}{2}\int_{\mathcal{U}\cap \pi^{-1}(y)} \Tr( \nabla \Phi\wedge F  ) .
\]
Moreover, using the curvature concentration by  item 1 in Prop. \ref{prop:measurelimit2}, it suffices to prove
\[
\int_{Q_{sm}}\rho | \int_{\mathcal{U}\cap \pi^{-1}(y)}   \frac{\eta}{4\pi m}\Tr( \nabla \Phi\wedge F  ) + \Theta(y) | dvol_Q\to 0,\quad i\to +\infty,
\]
where $\eta$ is a cutoff function which is one on $Q_{sm}$ and is supported in $\mathcal{U}$, and $\rho$ is a cutoff function compactly supported in $Q_{sm}$. By a covering argument, 
we can reduce to the situation that $\text{supp}(\rho)$ is contained in a small ball  $Q_{sm}\cap B(y_0, s)$ around a point $y_0\in Q_{sm}$, where
the multiplicity $\Theta(y)$ can be assumed constant.

For any $y, y'\in B(y_0, s)\cap Q_{sm}$, we can take a 4-manifold $X_{y,y'} \subset \mathcal{U}$ with boundary on the two fibres $\mathcal{U}\cap \pi^{-1}(y)$ and $\mathcal{U}\cap \pi^{-1}(y')$, for instance by lifting the short geodesic arc in $Q$ joining $y, y'$. Since $d\Tr(\nabla \Phi\wedge F)=0$, the Stokes formula and Fubini theorem imply that
\[
\begin{split}
& \int_{y,y'\in  Q_{sm} } \rho(y)\rho(y') |\int_{\mathcal{U}\cap \pi^{-1}(y)}   \frac{\eta}{4\pi m}\Tr( \nabla \Phi\wedge F  )- \int_{\mathcal{U}\cap \pi^{-1}(y')}   \frac{\eta}{4\pi m}\Tr( \nabla \Phi\wedge F  )| dydy'
\\
=&    \int_{y,y'\in Q_{sm}  }  \rho(y)\rho(y') | \int_{X_{y,y'}} d\eta \wedge \frac{1}{4\pi m}\Tr( \nabla \Phi\wedge F  )| dydy'
\\
\leq &  C(\eta) \int_{\text{supp}(d\eta)} m^{-1}|\nabla \Phi| |F| dvol \to 0,
\end{split}
\]
where the convergence follows from item 1 in Prop. \ref{prop:measurelimit2} and the fact that the support of $d\eta$ is disjoint from $Q_{sm}$. On the other hand, the average value satisfies 
\[
\int_{ y\in Q_{sm}  } \rho(y) (\int_{\mathcal{U}\cap \pi^{-1}(y)}   \frac{\eta}{4\pi m}\Tr( \nabla \Phi\wedge F   ) + \Theta(y) )dvol_Q(y)
\]
by the weak convergence $Q=\lim - \frac{1}{4\pi m} \Tr(\nabla \Phi\wedge F)$.  Combining the two estimates, and using that $\Theta(y)$ is locally constant, we see item 1.

\textbf{Item 2}. By a covering argument, we may restrict to a small ball $y\in B(y_0, s)\cap  Q_{sm}$, so that $\Theta(y)$ is a local constant. By the same idea as Lemma \ref{lem:Chernnumberoscillation}, the $L^1$-oscillation
\[
\int_{y,y',t,t'} |\int_{S^2_{y,t} } \tilde{F}- \int_{S^2_{y',t'} } \tilde{F}| \leq C \int_{ \text{supp}( \cup	S^2_{y,t} )  }m^{-1} |\nabla \Phi| |F|\to 0,
\]
using that the union of the 2-spheres is disjoint from the support of $Q$. On the other hand, using that $Q=\lim_{i\to +\infty} d\tilde{F}$, the average Chern number satisfies
\[
\int_{y, t} (\int_{S^2_{y,t} } \tilde{F}- \Theta(y) ) \to 0.
\]
Combining the two facts shows item 2.
\end{proof}

	\begin{rmk}\label{rmk:Cinftyconvergenceenergy}
	As a caveat, the limiting density measures involve the normalisation factor $m^{-1}$, so $\mu_{|F|^2} (M\setminus Q_{sm})=0$ says only that on any fixed compact subset $K$ disjoint from the support of $Q$, we have $\int_K |F|^2=o(m)$ as $i\to +\infty$. On the other hand, if we want to deduce $C^\infty_{loc}$ convergence of the sequence $(A,\Phi)$ on some open subset bounded away from $Q$, by applying  the standard $\epsilon$-regularity theorem argument (See Prop. \ref{prop:epsilonregularity}), we need to establish 
	the much stronger uniform estimate that $\int_K |F|^2= O(1)$ instead of $o(m)$.

	\end{rmk}

		%energy cannot concentrate along coassociative

	\subsection{Extraction of monopole bubbles}

	Our next goal is to show that almost all the limiting curvature measure can be accounted for by $3$-dimensional monopoles on the normal fibres of $\pi: \mathcal{U}\to Q_{sm}$, which is a geometric  improvement on item 2 in Prop. \ref{prop:measurelimit2}. The strategy involves two ingredients: 
	\begin{itemize}
		\item Suppose the curvature is sufficiently concentrated, one needs a criterion to extract monopoles as bubbling limits.

		\item  The characteristic length scale of the monopoles is $\sim m^{-1}$, while Prop. \ref{prop:measurelimit2} only shows curvature concentration within an $o(1)$-neighbourhood of $Q$ as $i\to +\infty$. The problem is that the curvature density may a priori be dispersed within the $o(1)$ neighbourhood, to prevent monopole bubble extraction. One needs to rule this out by showing that most of the curvature density is in fact concentrated in a quantitatively sufficiently small set.

	\end{itemize}

	We begin with a monopole extraction criterion, which makes minimal use of the general compactness theory for Yang-Mills-Higgs energy in Section \ref{sect:YMcompactness}. The setting of Prop. \ref{prop:monopoleextraction} is slightly more general;
	we will apply it to a suitable rescaling of the $(A,\Phi)$ in our main setting.

	\begin{prop}\label{prop:monopoleextraction}
	Given $0<\kappa<1$, $\Lambda\geq 1$ and $R_0\geq 1$, there is some $0<\epsilon\ll 1$ and $R\gg 1$ depending on $\kappa$ and the upper bound on $ \Lambda, R_0$, such that the following holds. The constants $C$ are independent of all parameters.

	Let $\R^n= P\times P^\perp$, where $P$ is an $(n-3)$-dimensional vector subspace, and $P^\perp$ be its orthogonal subspace, and let $g$ be an almost flat Riemannian metric on the large ball $B(R)\subset \R^n$, such that $\norm{g- g_{Eucl}}_{C^{l,\alpha}}<\epsilon$.
	Let $(A,\Phi)$ be a Yang-Mills-Higgs solution on the large ball $B(R)$, such that 
	\begin{enumerate}
		\item  The Higgs field $|\Phi|\leq 1$, and $\dashint_{B(s)} (1-|\Phi|^2 )  \leq  C\Lambda^{-1} s^{\kappa-1}$ for $R\geq s\geq R_0$.
		
		\item  Let $F=F^{tt}+ F^{tn}+ F^{nn}$ and $\nabla \Phi= \nabla^t \Phi+ \nabla^n \Phi$ be the decomposition of $F, \nabla \Phi$ according to form factors parallel to $P$ or $P^\perp$. We suppose
		\[
		\int_{B(R) } |F^{tt}|^2+ |F^{tn}|^2+ |\nabla^t \Phi|^2+ | F^{nn}- *_3 \nabla^n \Phi|^2 <\epsilon.
		\]
	\item The $L^2$-integral  $\int_{ \{ 0 \}\times P^\perp\cap B(R)} \frac{ |F|^2}{2\pi } \leq k+ \frac{1}{10} $, where $k$ is a fixed positive integer, while
$
		\int_{B(R_0)} |F|^2
			\geq C^{-1}\Lambda^{-1} .
$

	\end{enumerate}
	
	Then up to a gauge transformation, we can find a nontrivial monopole $(A', \Phi')$ on $\R^3$, with mass $\limsup_{|x|\to +\infty} |\Phi'|=1$  and charge at most $k$, identified with its pullback to $\R^n$,
	such that
	\begin{enumerate}
		\item  $\sum_{i=0}^l  R_0^i \norm{ \nabla_{A'}^i (A-A', \Phi-\Phi') }_{C^0(B(100R_0))}  \leq C\Lambda^{-1} R_0^{-n}$.

		\item  On the spheres $\{ |x|=s   \}\subset P^\perp\simeq \R^3$, the estimate
		\[
	\sup_{|x|=s }	|F| \leq C s^{-2}, \quad \sup_{ |x|=s  }( 1-|\Phi|^2) \leq C\Lambda^{-1} ,\quad  \sup_{ |x|=s  } |[\nabla \Phi, \Phi ] |\leq Ce^{- cR_0},
		\]
		holds for a subset of values $s\in [ R_0, 100 R_0  ]$ with Lebesgue measure at least $98R_0$.

			\end{enumerate}

	\end{prop}

	\begin{proof}
	Suppose the contrary, then we can find a sequence of counterexamples where $R\to +\infty$, and $\epsilon\to 0$ sufficiently fast depending on $\Lambda, R, R_0, \kappa$.

	We first claim that the sequence $(A,\Phi)$ converges in $C^\infty_{loc}$-topology up to gauge, on any compact ball in $\R^n$. Since $|\Phi|\leq 1$ by hypothesis item 1, and $\Lap |\Phi|^2= -2 |\nabla \Phi|^2$ by the Yang-Mills-Higgs equation, we can use Green representation formula to deduce that on any $B(p,2)\subset B(R)$,
	\[
	\int_{B(p,1) } \frac{ |\nabla \Phi|^2} {  |x-p|^{n-2}   } dvol(x)\leq  C\int_{B(p,2)} |\Phi|^2 \leq C,
	\]
	so for $s\leq 1$ we have
$
	s^{4-n} \int_{B(p,s)} |\nabla \Phi|^2 \leq Cs^2,
$
hence by hypothesis item 2, the normalised Yang-Mills-Higgs energy satisfies
	\[
s^{4-n}	\int_{B(p,s)} |F|^2 + |\nabla \Phi|^2 \leq  Cs^2+ \epsilon s^{4-n}.
	\]
	We can fix $s=s_0$ small enough, and $\epsilon$ small enough depending on $s_0$, so that the RHS bound is sufficiently small to apply the $\epsilon$-regularity theorem of the Yang-Mills-Higgs equation \cite{UhlenbeckSmith}, which implies $C^\infty_{loc}$ subsequential convergence up to gauge. The limit $(A',\Phi')$ is a smooth YMH solution on $\R^n$.

	By hypothesis item 2, all the tangential curvature components parallel to $P$ vanish in the limit, so $(A', P')$ dimensionally reduces to a YMH solution on $\R^3$, and $F'=*_3\nabla_{A'} \Phi'$ as a consequence of $\int | F^{nn}- *_3 \nabla^n \Phi|^2\to 0$. By hypothesis item 1, the Higgs field $|\Phi'|\leq 1$, and
$
\dashint_{B(s)} (1-|\Phi'|^2 )  \leq  C\Lambda^{-1} s^{\kappa-1}$ for $ s\geq R_0$. In particular, 
\[
\limsup_{|x|\to +\infty} |\Phi'|=1.
\]
By hypothesis item 3, we deduce
	\[
	\int_{ \{ 0 \}\times P^\perp\cap B(R)} \frac{ |F'|^2}{2\pi } \leq k+ \frac{1}{10} ,\quad 	\int_{B(R_0)} |F'|^2
	\geq C^{-1}\Lambda^{-1} .
	\]
	so $(A',\Phi')$ is a \emph{nontrivial finite energy monopole} on $P^\perp\simeq \R^3$, with mass one and charge $k'\leq k$. By the $C^\infty_{loc}$ convergence, we obtain conclusion \textbf{item 1}.

The charge $k'$ monopole $(A', \Phi')$ on $P^\perp\simeq \R^3$ satisfies the following strong-weak field decomposition (See \cite[Section 5.1]{Singer} for an exposition):
\begin{itemize}
	\item  (Strong field) There are $N\leq  N(k')$ distinct `monopole cluster centres' $z_i\in \R^3$, such that $\int_{B(z_i, 1)\cap P^\perp} |F'|^2 \geq \epsilon_0'$ for some fixed $\epsilon_0'>0$. Globally on $\R^3$, the $L^\infty$ norm of the curvature $F'$ is bounded. 
	
	\item (Weak field) Away from the cluster centres, there are numbers $\alpha_1,\ldots \alpha_N\in \R$, such that 
	\[
	\begin{cases}
	| |\Phi' |(z)- 1 + \sum_1^N \alpha_i |z- z_i|^{-1} | \leq C(k') \sum_i \frac{1}{  |z-z_i|^2 },
	\\
	| (\nabla_{A'} \Phi'(z), \Phi'(z)) + \sum_1^N \alpha_i d( |z- z_i|^{-1}  ) |  \leq C(k') \sum_i \frac{1}{  |z-z_i|^3 },
	\\
	| [\Phi'(z), \nabla_{A'} \Phi'(z)] |  \leq C(k') \sum_1^N e^{ - |z-z_i|/2   },
	\\
	|F'|= |*_3 F'|= |\nabla_{A'}\Phi'|.
	\end{cases}
	\]
\end{itemize}
	In particular, taking a generic value of $s\in [ R_0, 100 R_0]$, we can arrange the sphere $|x|=s$ to be bounded away from the monopole cluster centres $z_i$ with distance $\gtrsim s$, so within an $O(s)$-tubular neighbourhood of the sphere disjoint from $z_i$, we have $\sup  |F| \leq Cs^{-2}$. Moreover, $\dashint_{B(2s)} (1-|\Phi'|^2) \leq C\Lambda^{-1}$, so the weak field case above implies $\sup_{|x|=s}(1- |\Phi'|^2) \leq C\Lambda^{-1} $. Using the $C^\infty_{loc}$ convergence, this shows conclusion \textbf{Item 2}.
	\end{proof}

	\subsection{Energy identity}

We identify the tubular neighbourhood $\mathcal{U}$ with a subset of the normal bundle, and use the fibrewise linear structure to define the vertical dilation maps on the normal fibres,
	\[
	\lambda_{p, m^{-1} }^{nor}: N_y Q\simeq \R^3\to \pi^{-1}(y)\cap \mathcal{U}, \quad x\mapsto p+ m^{-1} x.
	\]
		We now show that on most normal fibres the restriction of $(A,\Phi)$ is modelled on 3-dimensional monopoles, and these monopole bubbles account for almost all the limiting $L^2$-curvature density measure inside any fixed large open set.

	\begin{thm}(Quantitative energy identity)\label{thm:energyidentity}
	For any small $\delta>0$, the following holds for $m$ sufficiently large depending on $\delta$.

	Consider the region $\{    r(x)<r   \}\subset M$ for any fixed large number $r$.
	There is a subset $E\subset Q_{sm}$ with $\mathcal{H}^{n-3}(E)<\delta$, such that for every $y\in Q_{sm}\setminus E$, we can find finitely many disjoint balls $B(x_j, m^{-1} R_j)$ with $R_j\geq 1$ inside the normal fibre $\pi^{-1}(y)\cap \mathcal{U}$,  and monopoles $(A_j', \Phi_j')$ on $\R^3$ with mass one and charge $k_j\leq \Theta(y)$, such that

	\begin{enumerate}
		\item  Under the dilation maps, the rescaled  pullback of $(A,\Phi)$ is $C^\infty_{loc}$ close to $(A'_j, \Phi'_j)$:
		\[
		\norm{ (\lambda_{x_j, m^{-1}}^{nor*}A, m^{-1} \lambda_{x_j, m^{-1}}^{nor*} \Phi)-  (A'_j,\Phi_j')  }_{C^l (B(R_j)\cap N_yQ ) }  \leq C\delta,
		\]

		%\item The sum of the monopole charges equals the multiplicity of $Q$ at the point: $\sum k_j = \Theta(y)$.

		\item 
	For any $y\in Q_{sm}\setminus E$, the fibrewise $L^2$ curvature integral satisfies
		\[
		\begin{cases}
		|\sum_j	\int_{ \pi^{-1}(y)\cap B(x_j, m^{-1}R_j) }\frac{ |F|^2 }{2\pi m} - \Theta(y) |  \leq  C\delta,
			\\
				\int_{ \pi^{-1}(y)\setminus  \cup B(x_j, m^{-1}R_j) }\frac{ |F|^2 }{2\pi m} dvol_Q \leq  C\delta,
		\end{cases}
		\]

\item The curvature integral on the rest of the manifold
\[
\int_{ \{  r(x)<r \}\setminus (  \mathcal{U}\cap \pi^{-1}(Q_{sm}\setminus E)     ) } \frac{ |F|^2}{2\pi m} dvol(x) \leq  C\delta.
\]
	\end{enumerate}

	\end{thm}

	\begin{rmk}
	For any fixed $m$, the total $L^2$-curvature integral $\int_M |F|^2=+\infty$, so for item 3 above to hold, we must restrict to a fixed open subset $\{  r(x)<r \}$ independent of $m$. This issue is closely related to the fact that (\ref{eqn:weaklimitmeasure}) only holds in the sense of \emph{weak convergence}; an infinite amount of curvature density measure disappears into the spatial infinity in taking this weak limit.

	\end{rmk}

	\begin{proof}
		\textbf{Step 1: removing some bad sets}. 
	After deleting a  closed subset $E_1$ with $\mathcal{H}^{n-3}$-measure $<\delta/3$, we can achieve the following properties when $i$ is sufficiently large:
	\begin{itemize}
		\item The set $Q_{sm}\setminus E_1$ is bounded away from the singular set in $Q$.
		
		\item  For each $y\in Q_{sm}\setminus E_1$, the fibrewise energy satisfies
		\[
		\int_{\mathcal{U}\cap \pi^{-1}(y)} \frac{ |F|^2}{2\pi m} dvol_{N_y Q} \leq \Theta(y)+ \frac{1}{10},
		\]
		and the mollified curvature satisfies
		\[
		\int_{t_0(y)}^{t_0(y)+1} | \int_{S^2_{y,t}} \tilde{F}- \Theta(y) |dt  \leq C\delta. 
		\]
		This is a consequence of the strong convergence in Cor. \ref{cor:fibrewisestrongL1}.

		\item For each $y\in Q_{sm}\setminus E_1$, the fibrewise energy satisfies
		\[
	 m^{-1}	\int_{\mathcal{U}\cap \pi^{-1}(y)} |F^{tt}|^2+ |F^{tn}|^2+ |F^{nn}-*_3 \nabla^n \Phi|^2  \leq C\delta.
		\]
		This is a consequence of item 1 in Prop. \ref{prop:measurelimit2}.
	\end{itemize}

	We fix any $0<\kappa<1$, and 
	let $\Lambda\gg 1$, $R_0\gg 1$ be some large parameters depending on $\delta$ to be determined, and we will always assume $m$ is sufficiently large depending on all the other parameters. Let $U$ be the subset $\mathcal{U}\cap \pi^{-1}(Q_{sm}\setminus E_1)$, then we can apply the \emph{Vitali covering argument}, to obtain a cover $B(p_j, 5s_j)$ for $\mathcal{C}_\Lambda\cap U$, with the properties listed in Lemma \ref{lem:Vitali}. In particular, by item 3 in Lemma \ref{lem:Vitali},
	\[
	\Lambda^{-1}m^{1+\kappa} \sum_j s_j^{ n-3+\kappa } \leq C\int_{U'} |F|^2 \leq Cm,
	\]
hence
$
	\sum_{s_j\geq R_0 m^{-1}}  s_j^{n-3}  \leq C\Lambda R_0^{-\kappa}.
$
We will assume that $R_0$ and $\Lambda$ satisfy that
$
\Lambda R_0^{-\kappa}\ll \delta, 
$
so under the projection map $\pi: \mathcal{U}\to Q_{sm}$, the image 
\[
E_2:= \text{image}(  \bigcup \{  B(p_j,5s_j) | s_j\geq R_0m^{-1}  \}       )
\]
has $(n-3)$-volume much smaller than $\delta$. After absorbing $E_2$ into the error set $E_1$, we may assume that in this covering, all $s_j\leq R_0 m^{-1}$. Since $s_j\geq m^{-1}$ by Lemma \ref{lem:Vitali}, we have $m^{-1}\leq s_j\leq R_0 m^{-1}$.

		\textbf{Step 2: monopole extraction criterion}. 
	We plan to apply  Prop. \ref{prop:monopoleextraction} to the rescaled $G_2$/Calabi-Yau monopole $(\lambda_{p_j, m^{-1} }^*  A,   m^{-1} \lambda_{p_j, m^{-1}}^* \Phi)$, where we recall that $\lambda_{p,s}(x)= \exp_p(s x)$ is the dilation map. We denote  $k= \Theta(y)$, which is a local constant, and let $P^\perp$ be the 3-plane tangent to the normal fibre at $p_j$, and $P$ be its orthogonal complement.  Let $\epsilon, R$ be the constants from Prop. \ref{prop:monopoleextraction}. We will consider the balls $B(p_j , R m^{-1})\subset M$  where 
	\begin{equation}\label{eqn:almostmonopoleintegral}
	\int_{B(p_j, R m^{-1})} |F^{tt}|^2+ |F^{tn}| ^2+ |\nabla^t \Phi|^2 +  |F^{nn}- *_3 \nabla^n \Phi|^2 < \epsilon m^{n-4}.
	\end{equation}

We collect together all the balls $B(p_j, Rm^{-1})$ such that (\ref{eqn:almostmonopoleintegral}) \emph{fails}. By the Vitali cover Lemma \ref{lem:Vitali}, the balls $B(p_j, s_j)$ are disjoint, and we arranged $R_0m^{-1}\geq s_j \geq m^{-1}$. Then
\[
\begin{split}
&\sum_{ j: fail}  s_j^{n-3} \leq R_0^{n-3}\sum_{j : fail} m^{3-n} 
\\
\leq & R_0^{n-3}\epsilon^{-1} m^{-1} \sum_{j:fail}\int_{B(p_j, Rm^{-1}) }   |F^{tt}|^2+ |F^{tn}| ^2 + |\nabla^t \Phi|^2+ |F^{nn}- *_3 \nabla^n \Phi|^2
\\
\leq & C R_0^{n-3}\epsilon^{-1} m^{-1} R^n  \int_{ B_{Rm^{-1}}(\mathcal{U} ) }|F^{tt}|^2+ |F^{tn}| ^2 + |\nabla^t \Phi|^2+  |F^{nn}- *_3 \nabla^n \Phi|^2.
\end{split}
\]
Here the third line uses that all $B(p_j, Rm^{-1})$ are contained inside $B_{Rm^{-1}}(\mathcal{U} )$, and each point in $B(p_j, Rm^{-1})$ falls into at most $CR^n$ balls $B(p_j, Rm^{-1})$ by the disjointenss of $B(p_j, m^{-1})$. On the other hand, Prop. \ref{prop:measurelimit2} implies that 
\[
\lim_{i\to +\infty} m^{-1} \int_{ B_{Rm^{-1}}(\mathcal{U} ) }|F^{tt}|^2+ |F^{tn}| ^2 +  |\nabla^t \Phi|^2+ |F^{nn}- *_3 \nabla^n \Phi|^2=0,
\]
so for fixed $\epsilon, R, R_0$, by taking $i$ sufficiently large, we can assume
$
\sum_{j : fail} s_j^{n-3} \ll \delta.
$
We let $E_3$ be the image under the $\pi$-projection,
\[
E_3:= \text{image}(  \bigcup \{  B(p_j,5s_j) | (\ref{eqn:almostmonopoleintegral}) \text{ fails} \}       ).
\]
After absorbing $E_3$ into the error set $E_1, E_2$, we may assume that (\ref{eqn:almostmonopoleintegral}) holds for the cover. Let $E=E_1\cup E_2\cup E_3$.

We can now verify the hypotheses of Prop. \ref{prop:monopoleextraction} for $(\lambda_{p_j, m^{-1} }^*  A,   m^{-1} \lambda_{p_j, m^{-1}}^* \Phi)$. The induced metric on the rescaled ball $B(R)$ is almost flat when $m$ is large enough. Hypothesis item 1 in Prop. \ref{prop:monopoleextraction} follows from $|\Phi|\leq m$, and the Higgs field estimate item 4 in Lemma \ref{lem:Vitali}. Hypothesis item 2 follows from (\ref{eqn:almostmonopoleintegral}). Hypothesis item 3 follows from Step 1 above, and the fact that
\[
\int_{B(p_j, R_0m^{-1})}  |F|^2\geq  \int_{B(p_j, s_j)} |F|^2 \geq C^{-1}\Lambda^{-1} s_j^{n-3+\kappa} m^{1+\kappa}\geq C^{-1}\Lambda^{-1} m^{4-n},
\]
which uses	item 2 in Lemma \ref{lem:Vitali}, and $m^{-1}\leq s_j\leq R_0 m^{-1}$ which we arranged in Step 1.

By the conclusion item 1 of Prop. \ref{prop:monopoleextraction}, we can find monopoles $(A_j',\Phi'_j)$ with mass one and charge $\leq k= \Theta(y)$ on $B(100 R_0)\subset \R^n$, such that 
	\begin{equation}
	\sum_{i=0}^l  R_0^i \norm{ \nabla_{A'}^i (\lambda_{p_j, m^{-1}}^* A-A', m^{-1}\lambda_{p_j, m^{-1}}^* \Phi-\Phi'_j) }_{C^0(B(100R_0))}  \leq C\Lambda^{-1} R_0^{-n}.
	\end{equation}
We will require $\Lambda\geq \delta^{-1}$. 
In particular, upon restriction to the normal fibres, this implies the estimate in our claim \textbf{item 1}.

The upshot is that near the curvature concentration locus, after deleting a sufficiently small set, the fibrewise geometry up to scaling is modelled on 3-dimensional monopoles in the smooth norm. For instance, the $L^2$-curvature of $F$ on these balls can be computed up to $O(\Lambda^{-1})$ error by the rescaled monopole models.

\textbf{Step 3: effective topological energy formula}. It remains to prove that the regions we have covered contain most of the $L^2$-curvature. As a caveat, the total volume in $M$ covered by the balls $B(p_i, 100 R_0 m^{-1})$ is roughly of order $O(m^{-1})$, which is much smaller than the volume of $\mathcal{U}$, so the main problem is that the $L^2$-curvature may a priori  be dispersed in $\mathcal{U}$.

Let $y\in Q_{sm}\setminus E$. The normal fibre $\mathcal{U}\cap \pi^{-1}(y)$ contains a bounded number of bubble regions $B(p_j, 10R_0m^{-1})\cap \mathcal{U}\cap \pi^{-1}(y)  $ with $\text{dist}(p_j, \mathcal{U}\cap \pi^{-1}(y))\leq 5s_j$, because the total fibrewise $L^2$-curvature satisfies
		\[
	\int_{\mathcal{U}\cap \pi^{-1}(y)} \frac{ |F|^2}{2\pi m} dvol_{N_y Q} \leq \Theta(y)+ \frac{1}{10},
	\]
	and each monopole bubble takes a definite portion of the energy. Up to merging a few bubbles which are separated by distance $\lesssim R_0m^{-1}$, %and changing the parameter $R_0$ by a cosmetic constant factor, 
	we 
	may assume that the balls $B_j:= B(p_j, R_j m^{-1})\cap \mathcal{U}\cap \pi^{-1}(y)$  inside the normal fibre are mutually disjoint, where $R_j>10R_0$ is a parameter comparable to $R_0$ satisfying a `good radius property' below. %By Step 1, these regions cover $\mathcal{C}_\Lambda\cap \mathcal{U}\cap \pi^{-1}(y)$. 
	(At this stage of the argument, these bubble regions on the normal fibre may a priori be empty.)

	We now count the energy within these bubble regions, which is based on the classical Bolgomolny trick for monopoles. By Step 1,
	\[
	 m^{-1}	\int_{\mathcal{U}\cap \pi^{-1}(y)} |F^{tt}|^2+ |F^{tn}|^2+ |F^{nn}-*_3 \nabla^n \Phi|^2  \leq C\delta,
	\]
hence we can compute the energy up to $O(\delta)$ error by the Stokes formula:
	\[
	\begin{split}
		\int_{  B_j } \frac{ |F|^2}{ 2\pi m }=& -  \int_{  B_j } \frac{ \Tr (\nabla \Phi\wedge F)}{ 4\pi m } +O(\delta)
		\\
		= &  -  \int_{  \partial B_j } \frac{ \Tr ( \Phi  F)}{ 4\pi m } +O(\delta).
	\end{split}
	\]
	However, by applying the rescaled version of conclusion item 2 in Prop. \ref{prop:monopoleextraction}, we can pick the $10 R_0< R_j< 100 R_0$ such that the `good radius property' holds: 
	\[
\sup_{\partial B_j } |F|\leq C(R_j m^{-1})^{-2}, \quad \sup_{\partial B_j } |1- \frac{ |\Phi|^2}{m^2} |\leq C\Lambda^{-1},
\quad 
\sup_{  \partial B_j} |[\nabla \Phi, \Phi ] |\leq Cm^3 e^{- c R_0}.
\]
We recall that the Chern form of the $U(1)$-connection is
	\[
	F_{U(1)}=   \frac{-1}{4\pi  } \Tr \left(    \frac{\Phi}{|\Phi|}	(F - \frac{1}{8} [ \nabla (  \frac{\Phi}{|\Phi|}  )\wedge \nabla  ( \frac{\Phi}{|\Phi|})   ])   \right),
	\]
	whence on $\partial B_j$ we have
	\[
	\begin{split}
	& \int_{\partial B_j  }	|F_{U(1)} + \frac{ \Tr ( \Phi  F)}{ 4\pi m } |
	\\
	\leq & C\Lambda^{-1} + C\int_{\partial B_j  } | \nabla (  \frac{\Phi}{|\Phi|}  ) |^2
	\\
	\leq  & C\Lambda^{-1} + C\int_{\partial B_j } \frac{ |[\nabla \Phi,  \Phi] |^2}{m^4} \leq C\Lambda^{-1} + C e^{-cR_0}.
	\end{split}
	\]
	As long as $\Lambda\geq \delta^{-1}$, and $e^{-cR_0}\leq \delta$, the RHS is bounded by $C\delta$. Combining the above, we arrive at
	\begin{equation}
		\int_{  B_j } \frac{ |F|^2}{ 2\pi m } - \int_{\partial B_j   }  F_{U(1)} |\leq C\delta.
	\end{equation}
	This can be viewed as an \emph{effective analogue of the topological energy formula for monopoles} on $\R^3$, which relates the total $L^2$-energy of the monopole to the degree of the asymptotic $U(1)$-bundle.

	\textbf{Step 4: topological degree}. 
	We now require that $\Lambda\geq C\Lambda_0$ be large enough so that $\mathcal{C}_{\Lambda}\supset B_{m^{-1}  } (\mathcal{C}_{\Lambda_0 }  ) $ by item 4 in Lemma \ref{lem:characteristicradius}. 	By the Vitali cover construction in Lemma \ref{lem:Vitali}, the union of $B(p_j, 5s_j)$ cover $\mathcal{U}\cap \mathcal{C}_\Lambda$, and using $R_j m^{-1}>5R_0m^{-1}\geq 5s_i$, we see
the union of $B_j= B(p_j, R_j m^{-1})\cap  \mathcal{U}\cap \pi^{-1}(y)$ covers the intersection $\mathcal{U}\cap \pi^{-1}(y)\cap \mathcal{C}_\Lambda$.

	By step 1, the normal fibre satisfies 
	\[
	\int_{t_0(y)}^{t_0(y)+1} | \int_{S^2_{y,t}} \tilde{F}- \Theta(y) |dt  \leq C\delta. 
	\]
Since there are only a bounded number of balls $B_j$, whose radii are comparable to $R_0 m^{-1}$,  we see that for most $t\in [t_0(y), 2t_0(y)]$, the 2-sphere $S^2_{y,t}$ is disjoint from $B_{m^{-1}}(\mathcal{C}_{\Lambda_0})$, so that $\tilde{F}=F_{U(1)}$, and
$
\int_{S^2_{y,t}} F_{U(1)}= \Theta(y) \in \Z.
$

Inside the normal fibre, this 2-sphere $S^2_{y,t}= \partial D^3$ encloses a bounded number of balls $B_j$ from Step 3 above. On $D^3\setminus \cup B_j$, there is no point of $\mathcal{C}_\Lambda$, and since the zero locus $\Phi^{-1}(0)$ is contained in $\mathcal{C}_{\Lambda_0}\subset \mathcal{C}_\Lambda$ by Lemma \ref{lem:curvatureconcentrationlocus1}, the Chern form $F_{U(1)}$ is a \emph{smooth} closed 2-form on $D^3\setminus \cup B_j$. Thus by Stokes formula,
\[
\Theta(y)= \int_{S^2_{y,t} } F_{U(1)} = \sum_j  \int_{\partial B_j } F_{U(1)}.
\]
Using the effective topological energy formula in Step 3,
\begin{equation}
|\sum_j \int_{ B_j } \frac{ |F|^2}{2\pi m} - \Theta(y) |\leq C\delta. 
\end{equation}
This is the first estimate in \textbf{Item 2}.

Now as $y$ varies over $Q_{sm}\setminus E$, the total integral of $\frac{|F|^2}{2\pi m}$ inside the region covered by the monopole bubbles, is at least
\[
\begin{split}
& \int_{Q_{sm}\setminus E} (1-C\delta) \Theta(y) dvol_Q
 \\
 \geq & (1-C\delta) \text{Mass}(Q)
\\
=& (1-C\delta) 		\begin{cases}
	\langle  \beta \cup \psi|_\Sigma, [\Sigma]\rangle, \quad & \text{$G_2$ case},
	\\
	\langle  \beta \cup \text{Re}(\Omega)|_\Sigma, [\Sigma]\rangle, \quad & \text{Calabi-Yau case}.
\end{cases}
\end{split}
\] 
On the other hand, the total integral inside $\{  r(x)<r  \}$, is bounded above by the monotonicity formula Lemma \ref{lem:monotonicity2},
\[
\int_{r(x)<r}   \frac{ |F|^2}{ 2\pi m  } dvol \leq 	Cr^{n-4}m^{-1} +\begin{cases}
 \frac{E^\psi}{2\pi m},\quad & \text{$G_2$ case},
	\\
 \frac{E^\Omega}{2\pi m},  \quad & \text{Calabi-Yau case}.
\end{cases}
\]
For fixed $r$, by taking $m$ large enough, we can ensure $r^{n-4} m^{-1}\leq \delta$. Using the intermediate energy formula (\ref{eqn:intermediateenergy2})(\ref{eqn:intermediateenergyCY2}), we obtain
\[
\int_{r(x)<r}   \frac{ |F|^2}{ 2\pi m  } dvol \leq  C\delta+ \begin{cases}
	\langle  \beta \cup \psi|_\Sigma, [\Sigma]\rangle, \quad & \text{$G_2$ case},
	\\
	\langle  \beta \cup \text{Re}(\Omega)|_\Sigma, [\Sigma]\rangle, \quad & \text{Calabi-Yau case}.
\end{cases}
\]
Contrasting these two bounds shows that the integral of $\frac{|F|^2}{2\pi m}$ in the complement of the monopole bubble regions inside $\{  r(x)<r  \}$ is bounded by $C\delta$, which is the content of the second estimate of item 2, together with \textbf{item 3}.

	We note that throughout the argument, we have imposed the following conditions on $R_0, \Lambda$:
	\[
	\Lambda R_0^{-\kappa} \ll \delta, \quad \Lambda\geq \delta^{-1}, \quad e^{-cR_0} \leq \delta,\quad \Lambda\geq C\Lambda_0.
		\]
	These choices are consistent. We conclude the energy identity theorem.
	\end{proof}

	%may want some decay estimate away from $\mathcal{C}_\Lambda$.

	\section{Discussions}\label{sect:open}

	\subsection{Bubbling phenomenon}

	This section contains semi-heuristic discussions on the bubbling phenomenon for Yang-Mills-Higgs connections. It is not directly used in the proof of the main results, but rather indicates why improving these statements may be subtle.

	\subsubsection{Convergence theory}\label{sect:YMcompactness}

	The best known general analytic convergence theory \cite{Tian}\cite{Tao}\cite{Naber} in the local setting can be summarized as follows; the notations are independent of the main text. Given a sequence of Yang-Mills connections $A_i$ on $B_2\subset \R^n$ with $n\geq 4$, assuming the uniform bound on $L^2$ energy $\int |F_{A_i}|^2 dvol\leq \Lambda$, then after passing to subsequences, the followings hold true on an interior ball $B_1$:
	
	\begin{itemize}
		\item (Codimension 4 convergence \cite{Tian}) Away from a possibly empty $(n-4)$-rectifiable closed subset $\mathcal{S}$ with $\mathcal{H}^{n-4}(\mathcal{S})\leq C(n,\Lambda)$, modulo gauge transforms $A_i$ converge in $C^\infty_{loc}$ topology to a limiting smooth YM connection $A_\infty$ on $B_1\setminus \mathcal{S}$. If $A_i$ are higher dimensional instantons (eg. Hermitian-Yang-Mills connections, $G_2$-instantons,  or $Spin(7)$-instantons), then so is $A_\infty$, and $\mathcal{S}$ is a calibrated current (resp. complex codimension 2 subvarieties, associative currents, Cayley currents).

		\item  (Removable singularity \cite{Tao}\cite{UhlenbeckSmith}) Assume further that $A_\infty$ is stationary on $B_1$. There is a possibly empty subset $Sing\subset \mathcal{S}$ with $\mathcal{H}^{n-4}(Sing)=0$, such that after gauge transform, $A_\infty$ extends to a smooth connection on $B_1\setminus Sing$. The smallest such $Sing$ is called the \emph{essential singular set} of $A_\infty$.

		\item
		(Transverse bubbles \cite{Tian}\cite{Naber})  A transverse bubble $A_x$ at $x\in \mathcal{S}$ is a smooth Yang-Mills connection over $\R^n$ isomorphic to the pullback of a YM connection on $\R^4\simeq \R^n/T_x\mathcal{S}$, such that for some sequence of centre points $x_i\to x$ and scale parameters $\epsilon_i$ depending on $x_i$, the blow up sequence $\lambda_{ x_i, \epsilon_i }^* A_i$ subsequentially converge to $A_x$ up to gauge transform, where $\lambda_{ x_i, \epsilon_i }$ refers to the scaling diffeomorphism $y\mapsto  x_i+\epsilon y$. Nontrivial transverse bubbles exist at $\mathcal{H}^{n-4}$-a.e $x\in \mathcal{S}$. If $A_i$ are higher dimensional instantons, then $A_x$ are isomorphic to the pullback of ASD connections on $\R^4$.

		\item  (Energy identity \cite{Tian}\cite{Naber}) The curvature density converges weakly as measures:
		\[
		|F_{A_i}|^2 dvol\to |F_{A_\infty}|^2 dvol +\nu, \quad \nu= \Theta \mathcal{H}^{n-4}\lfloor _\mathcal{S} ,
		\]
		where $\nu$ is a measure supported on $\mathcal{S}$ (called the defect measure), and the density $\Theta: \mathcal{S} \to \R_+$ is a bounded function. For $\mathcal{H}^{n-4}$-a.e. $x\in \mathcal{S}$, the function $\Theta(x)$ can be computed as the sum of $L^2$-energies associated with the transverse bubble connections, where one needs to account for multiple bubbling phenomenon.

		\item  (Curvature Hessian estimate \cite{Naber}) $\int_{B_1} |\nabla_{A_i}^2 F_{A_i}| \leq C(n, \Lambda)$.
	\end{itemize}

	These results have partially been extended to the Yang-Mills-Higgs setting. However, 
	in the context of the Donaldson-Segal programme, one major subtlety is that as $m\to +\infty$, the total $L^2$-curvature on the neighbourhood of the coassociative/special Lagrangian cycle in the model case is proportional to $m$, and in particular goes to infinity as $m\to +\infty$.

	\subsubsection{Bubbling along associative submanifolds}\label{sect:bubbling}

	Conversely,  Walpuski \cite{Walpuski} used gluing techniques to construct examples of $G_2$-instantons (namely $G_2$-monopole with zero Higgs field) on some compact $G_2$-manifolds $M$, exhibiting the codimension 4 bubbling phenomenon. In these examples, one has a sequence of $G_2$-instantons $A_i$ which   $C^\infty_{loc}$-converge away from a smooth associative submanifold $N^3\subset M$, to some smooth $G_2$-instanton $A_\infty$ over $M$ with smaller $L^2$-curvature. The curvature density of $A_i$ concentrates along the associative submanifold $N^3\subset M$, namely we have the weak convergence of measures
	\[
	|F_{A_i}|^2 dvol\to |F_{A_\infty}|^2 dvol+ {\Theta} \mathcal{H}^{n-4}\lfloor _{N^3} .
	\]
	The restrictions of the $G_2$-instantons to the normal planes along $N^3$ is modelled on ADHM instantons on $\R^4$. The variation of these ADHM instantons along $N^3$ is encoded by a section of an ADHM instanton moduli bundle over $N^3$, satisfying a nonlinear Dirac type equation known as the Fueter equation.

Morally speaking, when the length scale of curvature concentration is much smaller than the inverse of the mass, then the $G_2$-monopole equation is a perturbation of the $G_2$-instanton equation. This can be understood heuristically by a scaling analysis: suppose $(A,\Phi)$ is a $G_2$-monopole on the Euclidean ball of radius $r\ll m^{-1}$, with 
\[
|F_A|\sim r^{-2}, \quad |\Phi|\sim m,\quad |\nabla \Phi|\sim \frac{m}{r}.
\]
Upon the rescaling $\lambda_r: x\mapsto rx$, we obtain a new $G_2$-monopole $(\tilde{A},\tilde{\Phi})$ on the unit ball in $\R^7$, with $\tilde{A}= \lambda_r^* A$ and $\tilde{\Phi}= r \lambda_r^* \Phi$, so
\[
|F_{\tilde{A} } |\sim 1, \quad |\tilde{\Phi}|\sim m r \ll 1,\quad |\nabla_{\tilde{A}} \tilde{\Phi}  |\sim mr\ll 1, \quad | F_{\tilde{A}}\wedge \psi |= |  \nabla_{\tilde{A}}  \tilde{\Phi} |\ll 1,
\]
so the $G_2$-instanton equation $F_{\tilde{A}}\wedge \psi=0$ holds approximately. This suggests an affirmative answer to the following problem:

\begin{Question}
Can one extend the gluing technique in \cite{Walpuski} to the case of $G_2$-monopoles (resp. Calabi-Yau monopoles), to exhibit the ADHM instanton bubbling phenomenon along associative submanifolds (resp. holomorphic curves), such that the length scale of the ADHM instantons are much smaller than $m^{-1}$?
\end{Question}

\begin{rmk}
G. Oliveira and D. Fadel inform the author that there is some ongoing work with L. Foscolo, S. Esfahani and A. Nagy on this gluing construction.
\end{rmk}

	%Now the Walpuski construction produces sequences of $G_2$-instantons with bounded energy. 

	When the mass tends to infinity, the $L^2$-energy in a fixed ball  can also go to infinity, so in the convergence theory, we lose the uniform bound on the Hausdorff $(n-4)$-measure of the bubbling locus. This suggests that the bubbling locus can be arbitrarily complicated as $m\to +\infty$, at least in local examples on $B_1\subset \R^n$. In the local setting, a natural replacement for the mass $m$ is $\sup_{B_1} |\Phi|$.

	\begin{Question}
		On the standard Euclidean ball $B_1\subset \R^7$, can we find a double sequence of finite energy $G_2$-monopoles $(A_{i,j},\Phi_{i,j})$ with $\sup_{B_1} |\Phi_{i,j}|\to +\infty$, and a sequence of (possibly disconnected) associative subamanifolds $N^3_i\subset B_1$, such that for any fixed $i$, the $G_2$-monopoles $(A_{i,j},\Phi_{i,j})$ have $L^2$-curvature concentration along $N_i^3$ as $j\to +\infty$, namely 
		\[
			|F_{A_{i,j}}|^2 dvol\to |F_{A_{i,\infty}}|^2 dvol+ {\Theta_i} \mathcal{H}^{n-4}\lfloor _{N_i^3} ,\quad j\to +\infty,
		\]
		 and the Hausdorff measure $\mathcal{H}^{n-4}(N_i^3)\to +\infty$ as $i\to +\infty$? Furthermore, is it possible that the support of $N_i$ becomes increasingly dense in $B_1$, so that $B_1$ is the limit set of $N_i^3$ as $i\to +\infty$?

\end{Question}

\begin{rmk}
	In the bounded energy setting, an important ingredient to extract an $(n-4)$-rectifiable bubbling locus is the $\epsilon$-regularity theorem, which ensures that the defect measure $\Theta \mathcal{H}^{n-4}\lfloor \mathcal{S}$ has a uniform positive density lower bound $\Theta\geq \epsilon>0$. In the setting of the above Question, if the total energy in $B_1$ tends to infinity, then in order to extract a limit for the defect measures $\Theta_i \mathcal{H}^{n-4}\lfloor _{N_i^3}$, we should divide by the total $L^2$-curvature integral, which tends to infinity. After this normalisation, we no longer have a uniform lower density bound. A general sequence of real coefficient rectifiable $(n-4)$-currents in $B_1$, without a uniform lower density bound, can converge to a limiting current whose support is the entire ball $B_1$. The upshot is that there is no known measure theoretic mechanism to prevent the support of $N_i^3$ from becoming arbitrarily dense in $B_1$.

	\end{rmk}

	\begin{Question}
	In the global Setting of Section \ref{sect:largemasslimit}, after passing to subsequence and gauge transformations, do we have $C^\infty_{loc}$ convergence of the $G_2$-monopoles $(A_i,\Phi_i)$, away from some coassociative cycle in $M$?
	\end{Question}

If the bubbling locus can become arbitrarily dense in $B_1$, then in the \emph{local} version of the convergence theory, one \emph{cannot expect to have $C^\infty_{loc}$ convergence} for $G_2$-monopoles away from a proper closed subset, as the mass parameter tends to infinity.  The best one can hope for is convergence of the curvature in some $L^p$-topology.

 The status of globally defined $G_2$-monopoles on an asymptotically conical $G_2$-manifold is less clear. All the arguments in this paper fall short of ruling out that the curvature concentration locus may become arbitrarily dense as $m\to +\infty$ on some open subset of $M$, and in particular we failed to prove  $C^\infty_{loc}$ convergence of $(A,\Phi)$ on an open dense subset of $M$. We explained how each argument is inadequate in Remarks \ref{rmk:denseconcentration}, \ref{rmk:Cinftylimitnotsequence}, \ref{rmk:Cinftyconvergenceenergy}.

	\subsection{More open problems}

We will list a number of open problems on the analytic aspect of the $G_2$-monopoles.

\begin{Question}
What can be said if we replace the structure group $SU(2)$ by a general compact Lie group $G$?
\end{Question}

The special feature of $G=SU(2)$ (or $SO(3))$ in this paper is that the rank of its Cartan torus is one, so the centralizer of a nonzero element $\Phi\in \mathfrak{g}$ is just the Cartan subalgebra generated by $\Phi$, which is in particular abelian. Using the exponential decay of the commutator term $|[F,\Phi]|$, this `abelianization effect' explains why we can extract a singular abelian $G_2$-monopole (resp. Calabi-Yau monopole) in the limit. %The arguments go through almost verbatim for the group $SO(3)$.

For other structure groups, a nonzero element $\Phi\in \mathfrak{g}$ may lie on some wall in a Cartan subalgebra $\mathfrak{h}$, so that its centralizer is a \emph{nonabelian} Lie subalgebra of $\mathfrak{g} $  containing $\mathfrak{h}$ properly. Thus we expect that the possibility to extract an abelian $G_2$-monopole, will strongly depend on the genericity of the Higgs field as an element of the Lie algebra. If this genericity fails, then the best hope is to extract a $G_2$-monopole with smaller structure group.

	\begin{Question}
When can a coassociative cycle (resp. special Lagrangian cycle)  be realised as the  $(n-3)$-current $Q= d\tilde{F}_\infty$ of a sequence of $G_2$-monopoles (resp. Calabi-Yau monopoles) with mass tending to infinity?
	\end{Question}

	As suggested by the adiabatic calculation in Donaldson-Segal \cite[Section 6.3]{DonaldsonSegal}, one expects that the coassociative cycle (resp. special Lagrangian cycle) $Q$ should be equipped with the extra data of a Fueter section, which governs how the monopoles in the normal 3-planes vary along $Q$. One can hope that the existence of the Fueter section may impose some further constraint on $Q$. The following is an expanded exposition on Fueter sections for monopole moduli bundles.

	%By Almgren's big regularity theorem, $Q$ is smooth away from some closed subset of Hausdorff codimension at least two. 
	We first discuss the $G_2$-monopole setting, focusing on the smooth locus of the coassociative cycle $Q$. By the coassociative condition, the normal bundle $NQ$ is naturally isomorphic to $\Lambda^2_+ Q$, via the contraction $v\mapsto \iota_v \phi$. Now $\Lambda^2_+Q$ can be viewed as an associated bundle for the frame bundle $Fr_Q$, via the first factor of the Lie group representation $SO(4)\to SO(3)_+\times SO(3)_-$. The singular abelian monopole solution on $M\setminus Q$ with Dirac singularity along $Q$ provides a $U(1)$-bundle $\mathcal{L}\to Q$. The fibred product $\mathcal{Q}=\mathcal{L}\times_Q Fr_Q$ gives an $SO(4)\times U(1)$-principal bundle over $Q$.

	Let $M_k$ be the moduli space of framed centred $SU(2)$-monopoles of charge $k$ over $\R^3$, which is a hyperk\"ahler manifold admitting an $SO(3)$ action rotating the 2-sphere of complex structures. The framing data refers to the identification for the infinity of the $SU(2)$-monopole with the standard $U(1)$-monopole on $\R^3$ with Dirac singularity, and this framing data admits a $U(1)$-action. We can then take the associated bundle
	\[
\underline{M}_k=	\mathcal{Q}\times_{SO(4)\times U(1)} M_k \to Q,
	\]
	using the representation $SO(4)\to SO(3)_+$ to make $SO(4)$ act on $M_k$. The fibres of this bundle are isomorphic to $M_k$. As an observation in linear algebra, we note that a quaternionic triple of complex structures $I_1, I_2,I_3$ on any tangent space of $Q$, corresponds to an oriented orthonormal basis of $\Lambda^2_+$, and in turn corresponds to a hyperk\"ahler triple of complex structures on the fibres of $\underline{M}_k$.

	The monopole moduli bundle $\underline{M}_k$ inherits a connection  $\tilde{\nabla}$ from 
	the Levi-Civita connection on $Q$, and the $U(1)$-connection on $\mathcal{L}\to Q$. Let $s$ be a section of $\underline{M}_k\to Q$. We say that $s$ is a \emph{Fueter section}, if for any tangent vector $v\in TQ$, we have
	\begin{equation}
		(\tilde{\nabla}_v - I_1 \tilde{\nabla}_{I_1 v} -I_2 \tilde{\nabla}_{I_2 v}- I_3\tilde{\nabla}_{I_3v} ) s=0.
		\end{equation}
	We note the equation is independent of the choice of the quaternionic triple $I_1,I_2,I_3$. Moreover, if we replace $v$ by $I_i v$, the new equation is equivalent to the old one by applying the quaternionic relations. The Fueter section is equivalent to the adiabatic equation of Donaldson-Segal \cite[eqn (46)]{DonaldsonSegal}, after unravelling the hyperk\"ahler structure on $M_k$.

	Next we sketch the modifications for the Calabi-Yau monopoles. Here $Q$ is a special Lagrangian cycle, so there is a natural bundle isomorphism $TQ\simeq NQ$ via the complex structure $v\mapsto Jv$. The fibred product $\mathcal{Q}=\mathcal{L}\times_Q Fr_Q$ gives an $SO(3)\times U(1)$-principal bundle over $Q$, and we can take the associated monopole moduli bundle
	\[
	\underline{M}_k=	\mathcal{Q}\times_{SO(3)\times U(1)} M_k \to Q,
	\]
	which inherits a connection $\tilde{\nabla}$. 
	As an observation in linear algebra, we note that an oriented orthonormal basis $e_1,e_2,e_3$ on any tangent space of $Q$, corresponds to an oriented orthonormal basis of the normal space, and in turn corresponds to a hyperk\"ahler triple of complex structures $I_1,I_2,I_3$ on the fibre of $\underline{M}_k$. We say that $s$ is a \emph{Fueter section}, if 
	\begin{equation}
		(I_1 \tilde{\nabla}_{e_1} + I_2 \tilde{\nabla}_{e_2} +I_3 \tilde{\nabla}_{e_3} ) s=0.
	\end{equation}
	This equation is independent of the choice of oriented orthonormal basis.

	\begin{rmk}
	Here the Fueter sections land inside the framed \emph{monopole moduli bundle} $\underline{M}_k$, which is not to be confused with the Fueter sections into the framed \emph{ADHM instanton moduli bundle} that occurs in Walpuski's gluing construction. These describe two distinct phenomena in gauge theory.
	\end{rmk}

	From the analytical viewpoint, one needs to address the compactification question for the space of Fueter sections (\cf \cite{LiSaman}). This is closely related to two basic phenomena:
	\begin{enumerate}
		\item  A sequence of area minimizing integral currents can converge to a current with higher multiplicity. 
		
		\item A sequence of monopoles on $\R^3$ can go off to the infinity in the moduli space $M_k$, by decomposing into several clusters of monopoles with smaller charge $k_i$, such that $k=\sum k_i$, and the cluster centres are far separated.
	\end{enumerate}
As a somewhat oversimplified picture of the compactification, one needs to introduce a multivalued section $\mathfrak{s}$ of the normal bundle of $Q$ to encode the centres of the individual monopole clusters. In the $G_2$ case, we expect $\mathfrak{s}$ to be a multivalued self-dual harmonic 2-form over the coassociative cycle $Q$, namely a multivalued section of $\Lambda^2_+Q\to Q$ satisfying $d\mathfrak{s}=0$. In the Calabi-Yau 3-fold case, we expect $\mathfrak{s}$ to be a multivalued harmonic 1-form on $Q$, namely a multivalued section of $T^*Q\simeq TQ\simeq NQ$ satisfying $d\mathfrak{s}= d^*\mathfrak{s}=0$. Moreover, for each branch of the multivalued section, we need a Fueter section $s_i$ into a monopole moduli bundle with smaller charge $k_i$ (here $\sum k_i=k$), to describe how the internal behaviour within the monopole cluster varies along $Q$. We call the combined data $(\mathfrak{s}, \{  s_i \})$ a `multivalued Fueter section'.

\begin{Question}\label{Q:Fuetercompactification}
Taking into account all the possibilities of partitions $k=\sum k_i$, and the phenomenon of `clusters within clusters', can one prove that these multivalued Fueter sections give a  way to compactify the space of Fueter sections? 
\end{Question}

The multivalued Fueter sections provide the following refined description of the $G_2$-monopoles (resp. Calabi-Yau monopole) near $Q$ in the large mass limit. The curvature density is concentrated near the various branches of the multivalued section $\epsilon \mathfrak{s}$ around $Q$, where $m^{-1}\ll \epsilon\ll 1$ is a scaling parameter. On the 3-planes normal to each branch, the restriction of the $G_2$-monopole is well approximated by monopoles with charge $k_i$. The variation of these monopoles along the branch is encoded by the Fueter section $s_i$.

The following question seems very challenging:
	
	\begin{Question}
	In the Setting of Section \ref{sect:largemasslimit},  can one extract a nontrivial  (multivalued) Fueter section from the sequence of $G_2$-monopoles (resp. Calabi-Yau monopoles) with mass tending to infinity?
	\end{Question}

	\section*{Appendix: Hodge theory on AC manifolds}

	In this appendix $M$ is an irreducible AC $G_2$-manifold (resp. AC Calabi-Yau 3-fold). In particular $M$ is Ricci-flat, so there is only one end at infinity with link $\Sigma$.  The Laplacian here will be the Hodge Laplacian, which agrees with the rough Laplacian on 1-forms by Ricci-flatness. We will review some well known facts about the mapping properties of weighted H\"older spaces (\cf \cite[Section 4,5, Appendix]{Haskins}), and then adapt them to spaces with non-standard $L^1$-type conditions.

	Recall the $L^2$-cohomology of $M$ is 
	\[
	L^2\mathcal{H}^k(M)= \{       \sigma \in \Omega^k(M)\cap L^2:  \Lap_{Hodge}\sigma=0      \}
	= \{       \sigma \in \Omega^k(M)\cap L^2: d\sigma=0  =d^*\sigma    \}.
	\]
	We have a long exact sequence in cohomology with real coefficient for the pair $(M,\Sigma)$,
	\begin{equation}\label{eqn:LES}
	\ldots \to H^{k-1}(\Sigma)\to H_c^k(M)\to H^k(M)\to H^k(\Sigma)\to \ldots. 
	\end{equation}
	There is a natural isomorphism
	\begin{equation}\label{eqn:L2cohomology}
	L^2\mathcal{H}^k(M)= \begin{cases}
		H^k_c(M),\quad & k<n/2,
		\\
		Im(H^k_c(M)\to H^k(M)),\quad & k=n/2,
		\\
		H^k(M),\quad& k>n/2.
	\end{cases}
	\end{equation}

	For $l\geq 0$, $0<\alpha<1$ and a given rate $\nu\in \R$, we complete the space of $k$-forms into the weighted H\"older space $C^{l,\alpha}_\nu$ with respect to the norm
	\[
	\norm{u}_{C^{l,\alpha}_\nu} =\sum_{j=0}^l \norm{r^{-\nu+j}  \nabla^j u}_{C^0}+ [r^{-\nu+l}\nabla^l u]_\alpha.
	\]
The $\nu=0$ case is denoted as the $C^{l,\alpha}$ norm. 	In our convention $ \min_M r(x)>1$.

	\begin{lem}\label{lem:Lap1form}
	Let $-n< \nu<-2$, then the Hodge Laplacian on 1-forms (resp. functions) $\Lap:  C^{l+2,\alpha}_{\nu+2}\to C^{l,\alpha}_\nu$ is an isomorphism, and $\norm{u}_{C^{l+2,\alpha}_{\nu+2}}\leq C \norm{\Lap u}_{C^{l,\alpha}_\nu}$. 
	\end{lem}
	
	\begin{proof}
	(Sketch) By the Ricci-flatness of $M$, the Hodge Laplacian agrees with the rough Laplacian on 1-forms, so there is no harmonic 1-form with negative growth rate on $M$ by a maximum principle argument. This shows injectivity. One can show that the Laplacian is Fredholm by ruling out indicial roots in the given range. The obstruction to surjectivity comes from harmonic 1-forms (resp. functions) with growth rate $O(r^{-\nu-n} )$, which likewise vanishes.
	\end{proof}

	\begin{lem}\label{lem:2form1}
		Let $-n< \nu<-3$. 
	Suppose $u$ is a 2-form, $(du, d^* u)\in C^{l,\alpha}_\nu,$ and $u=o( r^{-2})$. Then there is some $\sigma\in L^2\mathcal{H}^2$, such that 
	\[
	\norm{u- \sigma}_{C^{l+1,\alpha}_{\nu+1}} \leq C \norm{du}_{C^{l,\alpha}_\nu} + C\norm{d^*u}_{C^{l,\alpha}_\nu}.
	\]
	
	\end{lem}

	\begin{proof}
		Since $du$ is exact,
	by the Poincar\'e lemma on the conical end, we can find a 2-form $v_1$ with $\norm{v_1}_{C^{l+1,\alpha}_{\nu+1}}\leq C \norm{du}_{C^{l,\alpha}_\nu}$, such that $du=dv_1$ outside a fixed compact set in $M$. We can replace $u$ by $u-v_1$, so without loss $du$ is compactly supported. Since $\nu<-3$ and $u\in o(r^{-2})$, under  the long exact sequence (\ref{eqn:LES}), we see that $[du]\in Im(H^2(\Sigma)\to H^3_c(M))$ vanishes. Thus we can find a compactly supported 2-form $v_2$  with $\norm{v_2}_{C^{l+1,\alpha}_{\nu+1}} \leq C \norm{du}_{C^{l,\alpha}_\nu}$, such that $dv_2=du$. We can replace $u$ by $u-v_2$, so without loss $du=0$.

	Next we solve the Poisson equation on 1-forms using Lemma \ref{lem:Lap1form} and the fact that $-n<\nu<-2$, to produce a 1-form $v_3$ with $\Lap v_3= d^* u$, satisfying the bound $\norm{v_3}_{C^{l+2,\alpha}_{\nu+2}} \leq C \norm{d^*u}_{C^{l,\alpha}_\nu}$. Thus $\Lap d^* v_3=d^*\Lap v_3=0$, and the function $d^*v_3=O(r^{\nu+1})$ decays at infinity, so $d^*v_3=0$ by the maximum principle. Hence $d^* d v_3= d^*u$. We can then replace $u$ by $u-dv_3$, so without loss $du=d^*u=0$.

	The closed 2-form $u$ restricts to the zero class in $H^2(\Sigma)$ because $u=o(r^{-2})$, so the long exact sequence shows that $[u] $ lifts to a class $ H^2_c(M) $. 
		In our setting $n=6$ or $7$, so $n/2>2$. By the natural isomorphism (\ref{eqn:L2cohomology}), we can find an $L^2$-harmonic form $\sigma\in L^2\mathcal{H}^k(M)$ representing the class $[u]$. By the $L^2$-integrability and the regularity of the harmonic forms, $\sigma\in C^{l+1,\alpha}_{-n/2}$, and in particular $|\sigma|=o(r^{-2})$. Without loss we replace $u$ by $u-\sigma$, so that the class $[u]$ vanishes in $H^2_c(M)$.

		Since $u$ is harmonic and $u=o(r^{-2})$, we can find some $\nu'<-2$ such that $u\in C^{l,\alpha}_{\nu'}$. Since $u$ is exact, by the Poincar\'e lemma argument above, we can find some 1-form $v_4$ with $dv_4=u$ and $\norm{v_4}_{C^{l+1,\alpha}_{\nu'+1}} \leq C \norm{u}_{C^{l,\alpha}_{\nu'}}$. By Lemma \ref{lem:Lap1form}, we can solve a Poisson equation on functions, to find a function $v_5$ with $\Lap v_5= d^* v_4$ and $\norm{v_5}_{C^{l+2,\alpha}_{\nu'+2}} \leq C \norm{d^* v_4}_{C^{l,\alpha}_{\nu'}}\leq    C \norm{u}_{C^{l,\alpha}_{\nu'}}$. Replacing $v_4$ by $v_4-dv_5$, we may assume without loss that $d^* v_4=0$. Thus by the coclosed property of $u$, we observe
		\[
	\Lap v_4=	(dd^* + d^*d) v_4= d^* dv_4= d^* u=0,
		\]
		so $v_4$ is a harmonic 1-form with growth control $v_4= O(r^{\nu'+1})$, whence $v_4=0$ and $u=dv_4=0$. This argument shows that there is no exact and coclosed 2-form with $u=o(r^{-2})$.
	\end{proof}

	We will need a variant involving two different weights. 
	
		\begin{lem}\label{lem:2form1'}
		Let $-n< \nu_1<\nu_2<-3$. 
		Suppose $u$ is a 2-form, such that 
		\[
		\norm{ (du,d^*u) }_{ C^{l,\alpha}(\{ r<r(x)<2r \}  ) } \leq     \max( A_1 r^{\nu_1},  A_2 r^{\nu_2}),\quad  \forall r>1,
		\]
		then there is some $\sigma\in L^2\mathcal{H}^2$, such that 
		\[
		r^{-1}\norm{   u-\sigma }_{ C^{l+1,\alpha}(\{ r<r(x)<2r \}  ) }
		 \leq  
		C	\max( A_1 r^{\nu_1},  A_2 r^{\nu_2}), \quad \forall r>1.
		\]
		
	\end{lem}

	\begin{proof}
	The proof is almost identical to Lemma \ref{lem:2form1}, with the following extra trick: there is a critical radius  $r_{crit}$ where $A_1 r^{\nu_1}$ is comparable to $A_2 r^{\nu_2}$, and depending on whether we are working in $\{ r(x)\gtrsim r_{crit}\}$ or 
$ \{   r(x)\lesssim r_{crit} \}$, we apply the results about usual weighted H\"older spaces, using the weight that dominates. 
	\end{proof}

	We will now adapt the above mapping property to some $L^1$-type spaces.

	\begin{lem}\label{lem:2form2}
		Let $-n< \nu<-3$. 
	Suppose $u$ is a 2-form with locally $L^1$-coefficients, $\dashint_{r<r(x)<2r} |u|=o( r^{-2})$, and for any $r\geq 1$, 
	\[
\dashint_{ r< r(x)< 2r } | du|+ |d^*u| \leq  \max( A_1 r^{\nu_1},  A_2 r^{\nu_2}).
	\]
	Then there is some $\sigma\in L^2\mathcal{H}^2$, such that 
	\[
  r^{-1} \dashint_{r<r(x)< 2r } |u-\sigma| \leq C   \max( A_1 r^{\nu_1},  A_2 r^{\nu_2}) .
	\]

	\end{lem}

	\begin{proof}
			Let $\{   \phi_i \}$ be a partition of unity on $M$, adapted to the dyadic annuli regions $\{  2^{i-1}\lesssim  r(x)\lesssim 2^i  \}$, and let $\tilde{\phi}_i$ be cutoff functions supported on these regions, such that $\tilde{\phi}_i=1$ on the support of $\phi_i$. We construct a pseudodifferential operator by patching together cutoff versions of approximate Green operators:
	\[
	\mathcal{P} (f)= \sum \tilde{\phi}_i P_i (\phi_i (d+d^*)f) ,
	\]
	where $P_i$ is a parametrix for the inverse of the Hodge Laplacian on the $i$-th annulus region (or the compact ball), with leading order symbol $\frac{1}{|\xi|_g^2} $ for any large enough cotangent vector $\xi$.   For any given  $l\geq 0$, $0<\alpha<1$, we can arrange that the error term $\mathcal{R}= (d+d^*) \mathcal{P}- Id$ has the following smoothing effect: 
	\begin{equation}\label{eqn:smoothing}
		\norm{	\mathcal{R}f }_{  C^{l,\alpha}(\{ r<r(x)<2r  \}) } \leq C  \dashint_{r/2< r(x)< 4r } | f| dvol.
	\end{equation}
	The estimate is manifestly invariant under rescaling, which is how we see the $r$-dependence.

	We choose $f=(d+d^*) u$ where $u$ is a 2-form, then $(d+d^*)f= \Lap_{Hodge} u$ is again a 2-form (in general with distributional coefficients), so $\mathcal{P}(f)$ is still a 2-form. Applying the rescaling trick and using the fact that $ \tilde{\phi}_i P_i (\phi_i (d+d^*))$ is a Calderon-Zygmund operator of order $-1$, we deduce for $r\sim 2^i$,
	\[
	r^{-n-1}	\int_M  |  \tilde{\phi}_i P_i (\phi_i (d+d^*)f)  | dvol \leq C \dashint_{\text{supp}(\phi_i)} |f| dvol \leq C\max( A_1 r^{\nu_1},  A_2 r^{\nu_2}).
	\]
	Since each dyadic annulus supports only a bounded number of cutoff functions $\phi_i, \tilde{\phi}_i$, we can sum over the contributions to obtain the $L^1$-decay estimate on $\mathcal{P}(f)$,
	\begin{equation*}
	r^{-1}	\dashint_{ r<r(x)<2 r}  |\mathcal{P}(f)| \leq C \max( A_1 r^{\nu_1},  A_2 r^{\nu_2})   .
	\end{equation*}

	Now $u-\mathcal{P}f$ is a still a 2-form, and 
	\[
	(d+ d^*) (u- \mathcal{P}f)= -\mathcal{R}f 
	\]
satisfies the weighted H\"older estimates by the above smoothing effect,
\[
	\norm{	\mathcal{R}f }_{  C^{l,\alpha}(\{ r<r(x)<2r  \}) } \leq C  \max( A_1 r^{\nu_1},  A_2 r^{\nu_2})   .
\]
By the elliptic regularity estimate for $u-\mathcal{P}f$ on the large annuli,
	\[
	\norm{ u-\mathcal{P}f }_{C^0 ( \{  r< r(x)<2r\}  )  }    \leq C  \dashint_{r/2< r(x)<4 r} |u-\mathcal{P}f| +C r \norm{  \mathcal{R}f }_{  C^{l,\alpha}(  \{ r/2< r(x)<4 r\}  ) }  = o(r^{-2}),
	\]
	where the second inequality uses the $L^1$ decay assumption on $u$, the $L^1$ decay estimate on $\mathcal{P}f$, and the fact that $\max(\nu_1,\nu_2)<-3$. We can now apply Lemma \ref{lem:2form1} to $u-\mathcal{P}f$, to find some $\sigma\in L^2\mathcal{H}^2$ so that 
	\[
	r^{-1}	\norm{    u-\mathcal{P}f-\sigma  }_{  C^{l+1,\alpha}(\{ r<r(x)<2r  \}) } \leq C  \max( A_1 r^{\nu_1},  A_2 r^{\nu_2})   .
	\]
 The Lemma follows.
	\end{proof}

%	\subsection{Morrey space estimate}
	
%We need the following mapping property with respect to the $L^1$-type Morrey norm estimate. In our case $n=6$ or $7$.

%\begin{lem}\label{lem:Morrey1}
%Let $B(p,2s)$ be a geodesic ball where the metric is $C^\infty$ uniformly equivalent to the standard Euclidean metric. Suppose $\Lap u=f$, and
%\[
 % \sup_{B(x,t)\subset B(p,2s)}  t^{4-n} \int_{B(x,t) } |f| dvol \leq A s^4,\quad  \int_{B(p,2s) }|u| dvol \leq As^{n+2},
%\]
%Then
%\[
 %\sup_{B(x,t)\subset B(p,s)}  t^{2-n} \int_{B(x,t) } |u| + t|\nabla u| dvol \leq CAs^2.
%\]
%\end{lem}

	\begin{Acknowledgement}
	The author is a Royal Society University Research Fellow based at Cambridge University. He would like to thank G. Oliveira, D. Fadel, D. Stern, A. Pigati and D. Parise for their interest and useful comments on this paper, and S. Esfahani for related discussions on the Donaldson-Segal programme and Fueter sections.
	\end{Acknowledgement}


\begin{thebibliography}{7}
		
		
		\bibitem{Ambrosio}  
		
		Ambrosio, Luigi; Kirchheim, Bernd. Currents in metric spaces. Acta Math. 185 (2000), no. 1, 1--80.
		
		
		\bibitem{DonaldsonSegal} Donaldson, Simon; Segal, Ed. Gauge theory in higher dimensions, II. Surveys in differential geometry. Volume XVI. Geometry of special holonomy and related topics, 1--41, Surv. Differ. Geom., 16, Int. Press, Somerville, MA, 2011.
	
		
		
		\bibitem{FadelOliveira} Fadel, Daniel; Oliveira, Goncalo. The limit of large mass monopoles. Proc. Lond. Math. Soc. (3) 119 (2019), no. 6, 1531--1559. 
		

\bibitem{Fadel}
D. Fadel. On the behavior of sequences of arbitrarily large mass monopoles in dimen-
sions 3 and 7. PhD thesis. University of Campinas, 2020.


\bibitem{Oliveira1} 

Fadel, Daniel; Nagy, Ákos; Oliveira, Goncalo. The asymptotic geometry of $\rm G_2$-monopoles. Mem. Amer. Math. Soc. 303 (2024), no. 1521, v+85 pp. 



\bibitem{HarveyLawson} 
Harvey, Reese; Lawson, H. Blaine, Jr. Calibrated geometries. Acta Math. 148 (1982), 47--157. 



\bibitem{Haskins} Foscolo, Lorenzo; Haskins, Mark; Nordström, Johannes. Complete noncompact $\rm G_2$-manifolds from asymptotically conical Calabi-Yau 3-folds. Duke Math. J. 170 (2021), no. 15, 3323--3416. 


		
		\bibitem{LiSaman} 
		
		Li, Y; Esfahani, S. Fueter sections and $\Z_2$-harmonic 1-forms. arxiv: 2410.06367.
		
		
		
		
		\bibitem{Naber}  Naber, Aaron; Valtorta, Daniele. Energy identity for stationary Yang Mills. Invent. Math. 216 (2019), no. 3, 847--925.
		
		

		
		
		
		\bibitem{Oliveiraexample1} Oliveira, Goncalo. Monopoles on the Bryant-Salamon $G_2$-manifolds. J. Geom. Phys. 86 (2014), 599--632. 
		
		
		\bibitem{Oliveiraexample2} Oliveira, Goncalo. Calabi-Yau monopoles for the Stenzel metric. Comm. Math. Phys. 341 (2016), no. 2, 699--728.
		
		
		
		
		\bibitem{PacardRitore} Pacard, Frank; Ritoré, Manuel. From constant mean curvature hypersurfaces to the gradient theory of phase transitions. J. Differential Geom. 64 (2003), no. 3, 359--423.
		
		
			\bibitem{Stern} 
		Parise, Davide; Pigati, Alessandro; Stern, Daniel. Convergence of the self-dual $U(1)$-Yang-Mills-Higgs energies to the $(n-2)$-area functional. Comm. Pure Appl. Math. 77 (2024), no. 1, 670--730.
		
		
		
		\bibitem{ParisePigatiStern} 	Parise, Davide; Pigati, Alessandro; Stern, Daniel.
		Nonabelian Yang-Mills-Higgs and Plateau's problem in codimension three. 	arXiv:2502.07756.
		
		

		
		
		\bibitem{SternPigati}  Pigati, Alessandro; Stern, Daniel. Minimal submanifolds from the abelian Higgs model. Invent. Math. 223 (2021), no. 3, 1027--1095.
		
		
		
		
		\bibitem{LeonSimon}
		Simon, Leon. Lectures on geometric measure theory. Proceedings of the Centre for Mathematical Analysis, Australian National University, 3. Australian National University, Centre for Mathematical Analysis, Canberra, 1983. {\rm vii}+272 pp. ISBN: 0-86784-429-9
		
		
	
		\bibitem{Singer}
		Fritzsch, K.; Kottke, C.; Singer, M.  
	Monopoles and the Sen Conjecture: Part I. https://arxiv.org/abs/1811.00601.
		
		
		
		
		\bibitem{Tian} Tian, Gang. Gauge theory and calibrated geometry. I. Ann. of Math. (2) 151 (2000), no. 1, 193--268.
		

		
		
		\bibitem{Tao} Tao, Terence; Tian, Gang. A singularity removal theorem for Yang-Mills fields in higher dimensions. J. Amer. Math. Soc. 17 (2004), no. 3, 557--593.
		
		

		
		
		\bibitem{UhlenbeckSmith}
		Smith, Penny; Uhlenbeck, Karen. Removeability of a codimension four singular set for solutions of a Yang Mills Higgs equation with small energy. Surveys in differential geometry 2019. %Differential geometry, Calabi-Yau theory, and general relativity. Part 2, 257--291, Surv. Differ. Geom., 24, Int. Press, Boston, MA, [2022], ©2022.
		

	
		

		
		
		\bibitem{Walpuski} Walpuski, Thomas. $G_2$-instantons, associative submanifolds and Fueter sections. Comm. Anal. Geom. 25 (2017), no. 4, 847--893.
		
		
		
		
	\end{thebibliography}
\end{document}